\documentclass[]{article}

\usepackage{graphicx}

\usepackage{amsmath}
\usepackage{amssymb}
\usepackage{amsthm}
\usepackage{pxfonts}
\usepackage{enumerate}
\usepackage{color}
\usepackage{mathdots}
\usepackage{sectsty}
\usepackage[hidelinks]{hyperref}

\sectionfont{\scshape\centering\fontsize{11}{14}\selectfont}
\subsectionfont{\scshape\fontsize{11}{14}\selectfont}
\usepackage{fancyhdr}

\newcommand\shorttitle{Interlacing Diffusions}
\newcommand\authors{T. Assiotis, N. O'Connell and J. Warren}

\newtheorem{thm}{Theorem}[section]
\newtheorem{cor}[thm]{Corollary}
\newtheorem{lem}[thm]{Lemma}
\newtheorem{defn}[thm]{Definition}
\newtheorem{rmk}[thm]{Remark}
\newtheorem{prop}[thm]{Proposition}

\title{\large \bf INTERLACING DIFFUSIONS}
\author{\small THEODOROS ASSIOTIS, NEIL O'CONNELL AND JON WARREN}
\date{}

\begin{document}

\maketitle

\begin{abstract}
We study in some generality intertwinings between $h$-transforms of Karlin-McGregor semigroups associated with one dimensional diffusion processes and those of their Siegmund duals. We obtain couplings so that the corresponding processes are interlaced and furthermore give formulae in terms of block determinants for the transition densities of these coupled processes. This allows us to build diffusion processes in the space of Gelfand-Tsetlin patterns so that the evolution of each level is Markovian. We show how known examples naturally fit into this framework and construct new processes related to minors of matrix valued diffusions. We also provide explicit formulae for the transition densities of the particle systems with one-sided collisions at either edge of such patterns. 
\end{abstract}

{\small\tableofcontents}

\section{Introduction}
In this work we study in some generality intertwinings and couplings between Karlin-McGregor semigroups (see \cite{KarlinMcGregor}, also \cite{Karlin}) associated with one dimensional diffusion processes and their duals. Let $X(t)$ be a diffusion process with state space an interval $I\subset \mathbb{R}$ with end points $l<r$ and transition density $p_t(x,y)$. We define the Karlin-McGregor semigroup associated with $X$, with $n$ particles, by its transition densities (with respect to Lebesgue measure) given by,
\begin{align*}
\det(p_t(x_i,y_j))^n_{i,j=1} ,
\end{align*}
for $x,y \in W^{n}(I^{\circ})$ where $W^{n}(I^{\circ})=(x=(x_1,\cdots,x_n):l < x_1 \le \cdots \le x_{n}< r )$. This sub-Markov semigroup is exactly the semigroup of $n$ independent copies of the diffusion process $X$ which are killed when they intersect. For such a diffusion process $X(t)$ we consider the conjugate (see \cite{ConjugateToth}) or Siegmund dual (see \cite{CoxRosler} or the original paper \cite{Siegmund}) diffusion process $\hat{X}(t)$ via a description of its generator and boundary behaviour in the next subsection. The key relation dual/conjugate diffusion processes satisfy is the following (see Lemma \ref{ConjugacyLemma}), with $z,z'\in I^{\circ}$,
\begin{align*}
\mathbb{P}_z(X(t)\le z')=\mathbb{P}_{z'}(\hat{X}(t)\ge z) .
\end{align*}
We will obtain \textit{couplings} of $h$-transforms of Karlin-McGregor semigroups associated with a diffusion process and its dual so that the corresponding processes \textit{interlace}. We say that  $y\in W^{n}(I^{\circ})$ and $x\in W^{n+1}(I^{\circ})$ interlace and denote this by $y\prec x$ if $x_1 \le y_1 \le x_2 \le \cdots \le x_{n+1}$. Note that this defines a space denoted by $W^{n,n+1}(I^{\circ})=((x,y):l < x_1 \le y_1 \le x_2 \le \cdots \le x_{n+1}< r)$,

\begin{center}
\begin{tabular}{ c c c c c c c c c }
 $\overset{x_1}{\bullet}$&$\overset{\textcolor{red}{y_1}}{\textcolor{red}{\bullet}}$ &$\overset{x_2}{\bullet}$   &$\overset{\textcolor{red}{y_1}}{\textcolor{red}{\bullet}}$ &$\overset{x_3}{\bullet}$ &$\cdots$ &$\overset{x_{n}}{\bullet}$&$\overset{\textcolor{red}{y_n}}{\textcolor{red}{\bullet}}$&$\overset{x_{n+1}}{\bullet}$ \ ,
 \end{tabular}
\end{center}
with the following two-level representation,

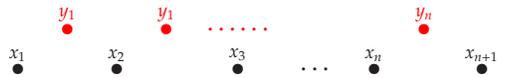
\begin{center}
\begin{tabular}{ c c c c c c c c c }
 &$\overset{\textcolor{red}{y_1}}{\textcolor{red}{\bullet}}$ &  &$\overset{\textcolor{red}{y_1}}{\textcolor{red}{\bullet}}$ &$\textcolor{red}{\cdots\cdots}$ &&&$\overset{\textcolor{red}{y_n}}{\textcolor{red}{\bullet}}$&\\ 
 $\overset{x_1}{\bullet}$&  &$\overset{x_2}{\bullet}$  & &$\overset{x_3}{\bullet}$&$\cdots$&$\overset{x_{n}}{\bullet}$&&$\overset{x_{n+1}}{\bullet}$ \ .
 \end{tabular}
\end{center}
Similarly, we say that $x,y \in W^n(I^{\circ})$ interlace if  $l< y_1 \le x_1 \le y_2 \le \cdots \le x_{n}< r$ (we still denote this by $y\prec x$). Again, this defines the space $W^{n,n}(I^{\circ})=((x,y):l < y_1 \le x_1 \le y_2 \le \cdots \le x_{n}< r )$,
\begin{center}
\begin{tabular}{ c c c c c c c c }
 $\overset{\textcolor{red}{y_1}}{\textcolor{red}{\bullet}}$ &$\overset{x_1}{\bullet}$   &$\overset{\textcolor{red}{y_2}}{\textcolor{red}{\bullet}}$ &$\overset{x_2}{\bullet}$ &$\cdots$ &$\overset{x_{n-1}}{\bullet}$&$\overset{\textcolor{red}{y_n}}{\textcolor{red}{\bullet}}$&$\overset{x_{n}}{\bullet}$ \  ,
 \end{tabular}
\end{center}
with the two-level representation,
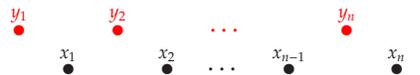
\begin{center}
\begin{tabular}{ c c c c c c c c c }
 &$\overset{\textcolor{red}{y_1}}{\textcolor{red}{\bullet}}$ &  &$\overset{\textcolor{red}{y_2}}{\textcolor{red}{\bullet}}$ & &$\textcolor{red}{\cdots}$&&$\overset{\textcolor{red}{y_n}}{\textcolor{red}{\bullet}}$&\\ 
 &  &$\overset{x_1}{\bullet}$  & &$\overset{x_2}{\bullet}$&$\cdots$&$\overset{x_{n-1}}{\bullet}$&&$\overset{x_{n}}{\bullet}$ \ .
 \end{tabular}
\end{center}

Our starting point in this paper are explicit transition kernels, actually arising from the consideration of stochastic coalescing flows. These kernels defined on $W^{n,n+1}(I^{\circ})$ (or $W^{n,n}(I^{\circ})$) are given in terms of block determinants and give rise to a Markov process $Z=(X,Y)$ with (sub-)Markov transition semigroup $Q_t$ with joint dynamics described as follows. Let $L$ and $\hat{L}$ be the generators of a pair of one dimensional diffusions in Siegmund duality. Then, after an appropriate Doob's $h$-transformation $Y$ evolves \textit{autonomously} as $n$  $\hat{L}$-diffusions conditioned not to intersect. The $X$ components then evolve as $n+1$ (or $n$) independent $L$-diffusions reflected off the random $Y$ barriers, a notion made precise in the next subsection. Our main result, Theorem \ref{MasterDynamics} in the text, states (modulo technical assumptions) that under a special initial condition for $Z=(X,Y)$, the \textit{non-autonomous} $X$ component is distributed as a Markov process in its own right. Its evolution governed by an explicit Doob's $h$-transform of the Karlin-McGregor semigroup associated with $n+1$ (or $n$) $L$-diffusions.

 At the heart of this result lie certain intertwining relations, obtained immediately from the special structure of $Q_t$, of the form,
\begin{align}
P_t\Lambda&=\Lambda Q_t \label{introinter1} \ ,\\
\Pi\hat{P}_t&=Q_t\Pi \label{introinter2} \ ,
\end{align}
where $\Lambda$ is an explicit positive kernel (not yet normalized), $\Pi$ is the operator induced by the projection on the $Y$ level, $P_t$ is the Karlin-McGregor semigroup associated with the one dimensional diffusion process with transition density $p_t(x,y)$ and $\hat{P}_t$ the corresponding semigroup associated with its dual/conjugate (some conditions and more care is needed regarding boundary behaviour for which the reader is referred to the next section). 

Now we move towards building a multilevel process. First, note that by concatenating $W^{1,2}(I^{\circ}), W^{2,3}(I^{\circ}),\cdots,W^{N-1,N}(I^{\circ})$ we obtain the space of Gelfand-Tsetlin patterns of depth $N$ denoted by $\mathbb{GT}(N)$,
\begin{align*}
\mathbb{GT}(N)=\{(X^{(1)},\cdots,X^{(N)}):X^{(n)}\in W^{n}(I^{\circ}), \ X^{(n)}\prec X^{(n+1)}\} \ .
\end{align*}
A point $(X^{(1)},\cdots,X^{(N)})\in\mathbb{GT}(N)$ is typically depicted as an array as shown in the following diagram:
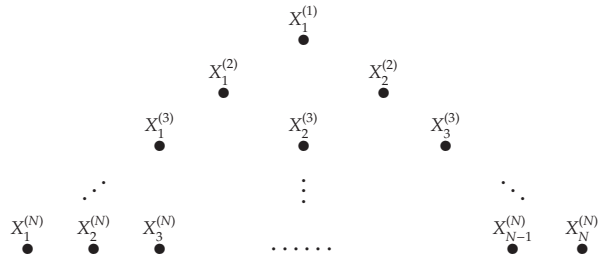
\begin{center}
\begin{tabular}{ c c c c c c c c c }
 &  &  &&$\overset{X^{(1)}_1}{\bullet}$ &&&&\\ 
 &  &  & $\overset{X^{(2)}_1}{\bullet} $&&$\overset{X^{(2)}_2}{\bullet} $&&&  \\ 
     &  &$\overset{X^{(3)}_1}{\bullet} $   &&$\overset{X^{(3)}_2}{\bullet} $ &&$\overset{X^{(3)}_3}{\bullet} $&&  \\ 
  &$\iddots$ & & &$\vdots$& &&$\ddots$&\\
$\overset{X^{(N)}_1}{\bullet} $ &  $\overset{X^{(N)}_2}{\bullet} $  &$\overset{X^{(N)}_3}{\bullet} $&&$\cdots\cdots$&& & $\overset{X^{(N)}_{N-1}}{\bullet} $&$\overset{X^{(N)}_N}{\bullet} $
\end{tabular}
\end{center}
Similarly, by concatenating  $W^{1,1}(I^{\circ}), W^{1,2}(I^{\circ}),W^{2,2}(I^{\circ}),\cdots,W^{N,N}(I^{\circ})$ we obtain the space of symplectic Gelfand-Tsetlin patterns of depth $N$ denoted by $\mathbb{GT}_{\textbf{s}}(N)$,
\begin{align*}
\mathbb{GT}_{\textbf{s}}(N)=\{(X^{(1)},\hat{X}^{(1)}\cdots,X^{(N)}, \hat{X}^{(N)}):X^{(n)},\hat{X}^{(n)}\in W^{n}(I^{\circ}), \ X^{(n)}\prec \hat{X}^{(n)}\prec X^{(n+1)} \} \ ,
\end{align*}

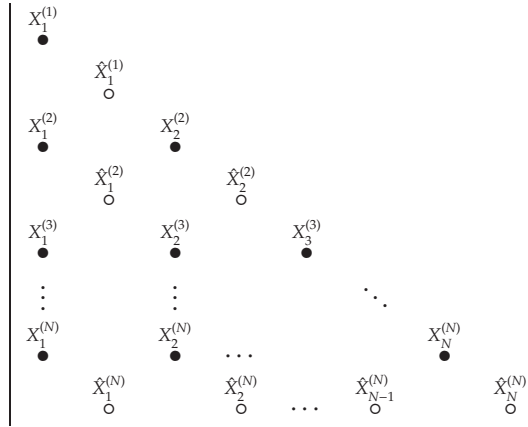
\begin{center}
\begin{tabular}{ | c c c c c c c c}
 $\overset{X^{(1)}_1}{\bullet}$ &  &  &&&&&\\ 
  & $\overset{\hat{X}^{(1)}_1}{\circ}$&  &&&&&\\  
 $\overset{X^{(2)}_1}{\bullet} $ &  & $\overset{X^{(2)}_2}{\bullet} $ &&&&&  \\ 
  & $\overset{\hat{X}^{(2)}_1}{\circ}$& &$\overset{\hat{X}^{(2)}_2}{\circ}$ &&&&\\
    $\overset{X^{(3)}_1}{\bullet} $ &  &  $\overset{X^{(3)}_2}{\bullet} $ && $\overset{X^{(3)}_3}{\bullet} $&&&  \\ 
  $\vdots $& &$\vdots$ & &&$\ddots$&&\\
$\overset{X^{(N)}_1}{\bullet} $ &  & $\overset{X^{(N)}_2}{\bullet} $  &$\cdots$&&&$\overset{X^{(N)}_N}{\bullet} $ &  \\ 
  & $\overset{\hat{X}^{(N)}_{1}}{\circ}$ & &$\overset{\hat{X}^{(N)}_{2}}{\circ} $ &$\cdots$&$\overset{\hat{X}^{(N)}_{N-1}}{\circ}$&&$\overset{\hat{X}^{(N)}_N}{\circ}$
\end{tabular}
\end{center}

Theorem \ref{MasterDynamics} allows us to concatenate a sequence of $W^{n,n+1}$-valued processes (or two-level processes), by a procedure described at the beginning of Section 3, in order to build diffusion processes in the space of Gelfand Tsetlin patterns so that each level is Markovian with explicit transition densities. Such examples of dynamics on \textit{discrete} Gelfand-Tsetlin patterns have been extensively studied over the past decade as models for random surface growth, see in particular \cite{BorodinKuan}, \cite{BorodinFerrari}, \cite{WarrenWindridge} and the more recent paper \cite{CerenziaKuan} and the references therein. They have also been considered in relation to building infinite dimensional Markov processes, preserving some distinguished measures of representation theoretic origin, on the boundary of these Gelfand-Tsetlin graphs via the \textit{method of intertwiners}; see Borodin and Olshanski \cite{BorodinOlshanski} for the type A case and more recently Cuenca \cite{Cuenca} for the type BC. In the paper \cite{Random growth} we pursued these directions in some detail.

 Returning to the continuum discussion both the process considered by Warren in \cite{Warren} which originally provided motivation for this work and a process recently constructed by Cerenzia in \cite{Cerenzia} that involves a hard wall fit in the framework introduced here. The techniques developed in this paper also allow us to study at the process level (and not just at fixed times) the process constructed by Ferrari and Frings in \cite{FerrariFrings}. The main new examples considered in this paper are:
 \begin{itemize}
\item Interlacing diffusion processes built from non-intersecting squared Bessel processes, that are related to the $LUE$ matrix diffusion process minors studied by K$\ddot{o}$nig and O'Connell in \cite{O Connell} and a dynamical version of a model considered by Dieker and Warren in \cite{DiekerWarren}. More generally, we study all diffusion processes associated with the classical orthogonal polynomials in a uniform way. This includes non-intersecting Jacobi diffusions and is related to the $JUE$ matrix diffusion, see \cite{Doumerc}. 
\item Interlacing Brownian motions in an interval, related to the eigenvalue processes of Brownian motions on some classical compact groups.
\item  A general study of interlacing diffusion processes with discrete spectrum and connections to the classical theory of total positivity and Chebyshev systems, see for example the monograph of Karlin \cite{Karlin}. 
 \end{itemize}

 We now mention a couple of recent works in the literature that are related to ours. Firstly a different approach based on generators for obtaining couplings of intertwined multidimensional diffusion processes via hard reflection is investigated in Theorem 3 of \cite{PalShkolnikov}. This has subsequently been extended by Sun \cite{Sun} to isotropic diffusion coefficients, who making use of this has independently obtained similar results to us for the specific $LUE$ and $JUE$ processes. Moreover, a general $\beta$ extension of the intertwining relations for the random matrix related aforementioned processes was also established in the note \cite{Assiotis} by one of us. Finally, some results from this paper have been used recently in \cite{HuaPickrell} to construct an infinite dimensional Feller process on the so called \textit{graph of spectra}, that is the continuum analogue of the Gelfand-Tsetlin graph, which leaves the celebrated Hua-Pickrell measures invariant.

We also study the interacting particle systems with one-sided collisions at either edge of such Gelfand-Tsetlin pattern valued processes and give explicit Schutz-type determinantal transition densities for them in terms of derivatives and integrals of the one dimensional kernels. This also leads to formulas for the largest and smallest eigenvalues of the $LUE$ and $JUE$ ensembles in analogy to the ones obtained in \cite{Warren} for the $GUE$.

Finally, we briefly explain how this work is connected to superpositions/ decimations
of random matrix ensembles (see e.g.\cite{ForresterRains}) and in a different direction to the study of strong stationary duals. This notion was considered by Fill and Lyzinski in \cite{FillLyzinski} motivated in turn by the study of strong stationary times for diffusion processes (first introduced by Diaconis and Fill in \cite{DiaconisFill} in the Markov chain setting).

The rest of this paper is organised as follows:
\begin{enumerate}[(i)]
\item In Section 2 we introduce the basic setup of dual/conjugate diffusion processes, give the transition kernels on interlacing spaces and our main results on intertwinings and Markov functions.
\item In Section 3 we apply the theory developed in this paper to show how known examples easily fit into this framework and construct new ones, among others the ones alluded to above. 
\item In Section 4 we study the interacting particle systems at the edges of the Gelfand-Tsetlin patterns.
\item In Section 5 we prove well-posedness of the simple systems of $SDEs$ with reflection described informally in the first paragraphs of the introduction and under assumptions that their transition kernels are given by those in Section 2.
\item In the Appendix we elaborate on and give proofs of some of the facts stated about dual diffusion processes in Section 2 and also discuss entrance laws.
\end{enumerate}

\paragraph{Acknowledgements.}Research of N.O'C. supported by ERC Advanced Grant 669306. Research of T.A. supported through the MASDOC DTC grant number EP/HO23364/1. We would like to thank an anonymous referee for many useful comments and suggestions which have led to many improvements in presentation.

\section{Two-level construction}

\subsection{Set up of conjugate diffusions}
Since our basic building blocks will be one dimensional diffusion processes and their conjugates we introduce them here and collect a number of facts about them (for justifications and proofs see the Appendix). The majority of the facts below can be found in the seminal book of Ito and McKean \cite{ItoMckean}, and also more specifically regarding the transition densities of general one dimensional diffusion processes, in the classical paper of McKean \cite{McKean} and also section 4.11 of \cite{ItoMckean} which we partly follow at various places.

 We consider  $(X_t)_{t\ge 0}$ a time homogeneous one dimensional diffusion process with state space an interval $I$  with endpoints $l<r$ which can be open or closed, finite or infinite (interior denoted by $I^{\circ}$) with infinitesimal generator given by,
\begin{align*}
L=a(x)\frac{d^2}{dx^2}+b(x)\frac{d}{dx},
\end{align*}
with domain to be specified later in this section. In order to be more concise, we will frequently refer to such a diffusion process with generator $L$ as an $L$-diffusion. We make the following regularity assumption throughout the paper.
\begin{defn}[Assumption (\textbf{R})]
We assume that $a(\cdot)\in C^1(I^{\circ})$ with $a(x)>0$ for $x\in I^{\circ}$ and $b(\cdot)\in C(I^{\circ})$. 
\end{defn}

 We start by giving the very convenient description of the generator $L$ in terms of its speed measure and scale function. Define its scale function $s(x)$ by
$s'(x)=\exp\big(-\int_{c}^{x}\frac{b(y)}{a(y)}dy\big)$ (the scale function is defined up to affine transformations) where $c$ is an arbitrary point in $I^\circ$, its speed measure with density $m(x)=\frac{1}{s'(x)a(x)}$ in $I^\circ$ with respect to the Lebesgue measure (note that it is a Radon measure in $I^\circ$ and also strictly positive in $I^{\circ}$) and speed function $M(x)=\int_{c}^{x}m(y)dy$. With these definitions the formal infinitesimal generator $L$ can be written as,
\begin{align*}
L=\mathcal{D}_m \mathcal{D}_s \ ,
\end{align*}
where $\mathcal{D}_m=\frac{1}{m(x)}\frac{d}{dx}=\frac{d}{dM}$ and $\mathcal{D}_s=\frac{1}{s'(x)}\frac{d}{dx}=\frac{d}{ds}$.

We now define the conjugate diffusion (see \cite{ConjugateToth}) or Siegmund dual (see \cite{Siegmund}) $(\hat{X}_t)_{t \ge 0}$ of $X$  to be a diffusion process with generator,
\begin{align*}
\hat{L}=a(x)\frac{d^2}{dx^2}+(a'(x)-b(x))\frac{d}{dx},
\end{align*}
and domain to be given shortly.

 The following relations are easy to verify and are key to us.
 \begin{align*}
\hat{s}'(x)=m(x) \ and \ \hat{m}(x)=s'(x) .
 \end{align*}
So the conjugation operation swaps the scale functions and speed measures. In particular
 \begin{align*}
\hat{L}=\mathcal{D}_{\hat{m}}\mathcal{D}_{\hat{s}}=\mathcal{D}_s\mathcal{D}_m \ .
 \end{align*}
 
 Using Feller's classification of boundary points (see Appendix) we obtain the following table for the boundary behaviour of the diffusion processes with generators $L$ and $\hat{L}$ at $l$ or $r$,
 \begin{align*}
\begin{tabular}{ l | r }
   Bound. Class. of L & Bound. Class. of $\hat{L}$ \\
  \hline			
  natural & natural \\
  entrance & exit  \\
  exit & entrance \\
  regular & regular \\
  \hline  
\end{tabular}
 \end{align*}
We briefly explain what these boundary behaviours mean. A process can neither be started at, nor reach in finite time a \textit{natural} boundary point. It can be started from an \textit{entrance} point but such a boundary point cannot be reached from the interior $I^\circ$. Such points are called \textit{inaccessible} and can be removed from the state space. A diffusion can reach an \textit{exit} boundary point from $I^\circ$ and once it does it is absorbed there. Finally, at a \textit{regular} (also called entrance and exit) boundary point a variety of behaviours is possible and we need to \textit{specify} one such. We will only be concerned with the two extreme possibilities namely \textit{instantaneous reflection} and \textit{absorption} ( sticky behaviour interpolates between the two and is not considered here). Furthermore, note that if $l$ is \textit{instantaneously reflecting} then (see for example Chapter 2 paragraph 7 in \cite{BorodinSalminen}) $Leb\{t: X_t=l \}=0 \ a.s.$ and analogously for the upper boundary point $r$.

 Now in order to describe the domain, $Dom(L)$, of the diffusion process with formal generator $L$ we first define the following function spaces (with the obvious abbreviations),
\begin{align*}
C(\bar{I})&=\{f\in C(I^\circ):\lim_{x \downarrow l}f(x), \lim_{x \uparrow r}f(x)\textnormal{ exist and are finite}\} \ , \\
\mathfrak{D}&=\{f\in C(\bar{I})\cap C^2(I^\circ):Lf\in C(\bar{I})\} \ ,\\
\mathfrak{D}_{nat}&=\mathfrak{D}\ ,\\
\mathfrak{D}_{entr}&=\mathfrak{D}_{refl}=\{f\in \mathfrak{D}:(\mathcal{D}_sf)(l^+)=0 \}\ ,\\
\mathfrak{D}_{exit}&=\mathfrak{D}_{abs}=\{f\in \mathfrak{D}:(Lf)(l^+)=0 \} .
\end{align*}

Similarly, define $\mathfrak{D}^{nat},\mathfrak{D}^{entr},\mathfrak{D}^{refl},\mathfrak{D}^{exit},\mathfrak{D}^{abs}$ by replacing $l$ with $r$ in the definitions above. Then the domain of the generator of the $\left(X_t\right)_{t\ge 0}$ diffusion process (with generator $L$) with boundary behaviour $i$ at $l$ and $j$ at $r$ where $i,j \in \{nat, entr, refl, exit, abs\}$ is given by,
\begin{align*}
Dom(L)=\mathfrak{D}_i \cap \mathfrak{D}^j \ .
\end{align*}
For justifications see for example Chapter 8 in \cite{EthierKurtz} and for an entrance boundary point also Theorem 12.2 of \cite{StochasticBook} or page 122 of \cite{McKean}.

 Coming back to conjugate diffusions note that the boundary behaviour of $X_t$, the $L$-diffusion, determines the boundary behaviour of $\hat{X}_t$, the $\hat{L}$-diffusion, except at a regular point. At such a point we define the boundary behaviour of the $\hat{L}$-diffusion to be dual to that of the $L$-diffusion. Namely, if $l$ is regular reflecting for $L$ then we define it to be regular absorbing for $\hat{L}$. Similarly, if $l$ is regular absorbing for $L$ we define it to be regular reflecting for $\hat{L}$. The analogous definition being enforced at the upper boundary point $r$. Furthermore, we denote the semigroups associated with $X_t$ and $\hat{X}_t$ by $\mathsf{P}_t$ and $\mathsf{\hat{P}}_t$ respectively and note that $\mathsf{P}_t1=\mathsf{\hat{P}}_t1=1$. We remark that at an \textit{exit} or \textit{regular absorbing} boundary point the transition \textit{kernel} $p_t(x,dy)$ associated with $\mathsf{P}_t$ has an \textit{atom} there with  mass (depending on $t$ and $x$) the probability that the diffusion has reached that point by time $t$ started from $x$.
 
 We finally arrive at the following duality relation, going back in some form to Siegmund. This is proven via an approximation by birth and death chains in Section 4 of \cite{CoxRosler}. We also give a proof in the Appendix in Section \ref{Appendix} following \cite{WarrenWatanabe} (where the proof is given in a special case). The reader should note the restriction to the interior $I^{\circ}$.

\begin{lem} \label{ConjugacyLemma}
$\mathsf{P}_t \textbf{1}_{[l,y]}(x)=\mathsf{\hat{P}}_t \textbf{1}_{[x,r]}(y)$ for $x,y \in I^{\circ}$.
\end{lem}

Now, it is well known that, the transition density $p_t(x,y):(0,\infty)\times I^\circ \times I^\circ \to (0,\infty)$ of any one dimensional diffusion process with a speed measure which has a continuous density with respect to the Lebesgue measure in $I^\circ$ (as is the case in our setting) is continuous in $(t,x,y)$. Moreover, under our assumptions $\partial_xp_t(x,y)$ exists for $x\in I^\circ$ and as a function of $(t,y)$ is continuous in $(0,\infty)\times I^\circ$ (see Theorem 4.3 of \cite{McKean}). 

This fact along with Lemma \ref{ConjugacyLemma} gives the following relationships between the transition densities for $x,y \in I^{\circ}$,
\begin{align}
p_t(x,y)&=\partial_y\mathsf{\hat{P}}_t \textbf{1}_{[x,r]}(y)=\partial_y\int_{x}^{r}\hat{p}_t(y,dz) \ ,  \\
\hat{p}_t(x,y)&=-\partial_{y} \mathsf{P}_t \textbf{1}_{[l,x]}(y)=-\partial_y\int_{l}^{x}p_t(y,dz). \label{conjtrans}
\end{align}

Before closing this section, we note that the speed measure is the \textit{symmetrizing} measure of the diffusion process and this shall be useful in what follows. In particular, for $x,y \in I^\circ$ we have,
 \begin{align}\label{symmetrizing}
 \frac{m(y)}{m(x)}p_t(y,x)=p_t(x,y).
 \end{align}

\subsection{Transition kernels for two-level processes}

First, we recall the definitions of the interlacing spaces our processes will take values in,
\begin{align*}
W^{n}(I^{\circ})&=((x):l < x_1 \le \cdots \le x_{n}< r ) \ ,\\
W^{n,n+1}(I^{\circ})&=((x,y):l < x_1 \le y_1 \le x_2 \le \cdots \le x_{n+1}< r ) \ ,\\
W^{n,n}(I^{\circ})&=((x,y):l < y_1 \le x_1 \le y_2 \le \cdots \le x_{n}< r ) \ , \\
W^{n+1,n}(I^{\circ})&=((x,y):l < y_1 \le x_1 \le y_2 \le \cdots \le y_{n+1}< r ) .
\end{align*}
Note that, for $(x,y)\in W^{n,n+1}(I^{\circ})$ we have $x\in W^{n+1}(I^{\circ})$ and $y\in W^{n}(I^{\circ})$, this is a minor difference in notation to the one used in \cite{Warren}; in the notations of that paper $W^{n+1,n}$ is our $W^{n,n+1}\left(\mathbb{R}\right)$.
Also define for $x\in W^{n}(I^\circ)$,
\begin{align*}
 W^{\bullet,n}(x)=\{y\in W^{\bullet}(I^{\circ}):(x,y)\in W^{\bullet,n}(I^{\circ})\}.
\end{align*}

We now make the following standing assumption, enforced throughout the paper, on the boundary behaviour of the one dimensional diffusion process with generator $L$, depending on which interlacing space our two-level process defined next takes values in. Its significance will be explained later on. Note that any possible combination is allowed between the behaviour at $l$ and $r$.

\begin{defn}[Assumption (\textbf{BC})] Assume the $L$-diffusion has the following boundary behaviour:\\
When considering $\boldsymbol{W^{n,n+1}(I^{\circ})}$:
\begin{align}
& l \textnormal{ is either } Natural \textnormal{ or } Entrance \textnormal{ or } Regular \ Reflecting\label{bc1l} \ , \\
& r \textnormal{ is either } Natural \textnormal{ or } Entrance \textnormal{ or } Regular \ Reflecting\label{bc1r}.
\end{align}
When considering $\boldsymbol{W^{n,n}(I^{\circ})}$:
\begin{align}
& l \textnormal{ is either } Natural \textnormal{ or } Exit \textnormal{ or } Regular \ Absorbing\label{bc2l} \ ,\\
& r \textnormal{ is either } Natural \textnormal{ or } Entrance \textnormal{ or } Regular \ Reflecting\label{bc2r} .
\end{align}
When considering $\boldsymbol{W^{n+1,n}(I^{\circ})}$:
\begin{align}
& l \textnormal{ is either } Natural \textnormal{ or } Exit \textnormal{ or } Regular \ Absorbing\label{bc3l} \ ,\\
& r \textnormal{ is either } Natural \textnormal{ or } Exit \textnormal{ or } Regular \ Absorbing\label{bc3r} .
\end{align}
\end{defn}

We will need to enforce a further regularity and non-degeneracy  assumption at regular boundary points for some of our results. This is a technical condition and presumably can be removed.

\begin{defn}[Assumption (\textbf{BC+})] Assume condition (\textbf{BC}) above. Moreover, if a boundary point $\mathfrak{b} \in \{l,r\}$ is regular we assume that 
$\underset{x \to \mathfrak{b}}{\lim} a(x)>0$ and the limits $\underset{x \to \mathfrak{b}}{\lim} b(x)$, $\underset{x \to \mathfrak{b}}{\lim} \left(a'(x)-b(x)\right)$ exist and are finite.
\end{defn}

We shall begin by considering the following stochastic process which we will denote by $\left(\boldsymbol{\Phi}_{0,t}(x_1),\cdots,\boldsymbol{\Phi}_{0,t}(x_n);t \ge 0\right)$. It consists of a system of $n$ independent $L$-diffusions started from $x_1\le \cdots \le x_n$ which \textit{coalesce} and move together once they meet. This is a process in $W^n(I)$ which once it reaches any of the hyperplanes $\{x_i=x_{i+1}\}$ continues there forever. We have the following proposition for the finite dimensional distributions of the coalescing process:
\begin{prop}\label{flowfdd}
For $z,z'\in W^n(I^{\circ})$,
\begin{align*}
\mathbb{P}\big(\boldsymbol{\Phi}_{0,t}(z_i)\le z_i' \ \ \text{for} \ 1 \le i \le n\big)=\det \big(\mathsf{P}_t\textbf{1}_{[l,z_j']}(z_i)-\textbf{1}(i<j)\big)_{i,j=1}^n \ .
\end{align*}
\end{prop}
\begin{proof}
This is done for Brownian motions in Proposition 9 of \cite{Warren} using a generic argument based on continuous non-intersecting paths. The only variation here is that there might be an atom at $l$ which however does not alter the proof.
\end{proof}

We now define the kernel $q_t^{n,n+1}((x,y),(x',y'))dx'dy'$ on $W^{n,n+1}(I^{\circ})$ as follows:

\begin{defn}
For $(x,y),(x',y')\in W^{n,n+1}(I^{\circ})$ define $q_t^{n,n+1}((x,y),(x',y'))$ by,
\begin{align*}
& q_t^{n,n+1}((x,y),(x',y'))=\\
&=\frac{\prod_{i=1}^{n}\hat{m}(y'_i)}{\prod_{i=1}^{n}\hat{m}(y_i)}(-1)^n\frac{\partial^n}{\partial_{y_1}\cdots\partial_{y_n}}\frac{\partial^{n+1}}{\partial_{x'_1}\cdots\partial_{x'_{n+1}}}\mathbb{P}\big(\boldsymbol{\Phi}_{0,t}(x_i)\le x_i',\boldsymbol{\Phi}_{0,t}(y_j)\le y_j' \ \ \text{for all} \ \ i,j \big) \ .
\end{align*}
\end{defn}
This density exists by virtue of the regularity of the one dimensional transition densities. It is then an elementary computation using Proposition \ref{flowfdd} and Lemma \ref{ConjugacyLemma}, along with relation (\ref{conjtrans}), that $q_t^{n,n+1}$ can be written out explicitly as shown below. Note that each $y_i$ and $x_j'$ variable appears only in a certain row or column respectively.
\begin{align}
{q}_t^{n,n+1}((x,y),(x',y'))=\det\
 \begin{pmatrix}
{A}_t(x,x') & {B}_t(x,y')\\
  {C}_t(y,x') & {D}_t(y,y') 
 \end{pmatrix} \ 
\end{align}
where,
\begin{align*}
{A}_t(x,x')_{ij} &=\partial_{x'_j}\mathsf{P}_t \textbf{1}_{[l,x_j']} (x_i)= p_t(x_i,x_j')  \ , \\
{B}_t(x,y')_{ij}&=\hat{m}(y'_j)(\mathsf{P}_t \textbf{1}_{[l,y'_j]}(x_i) -\textbf{1}(j\ge i)) \ ,\\
{C}_t(y,x')_{ij}&=-\hat{m}^{-1}(y_i)\partial_{y_i}\partial_{x'_j}\mathsf{P}_t \textbf{1}_{[l,x'_j]}(y_i)=-\mathcal{D}_s^{y_i}p_t(y_i,x_j') \ ,\\
{D}_t(y,y')_{ij}&=-\frac{\hat{m}(y'_j)}{\hat{m}(y_i)}\partial_{y_i} \mathsf{P}_t \textbf{1}_{[l,y'_j]}(y_i)=\hat{p}_t(y_i,y_j').
\end{align*}

We now define for $t>0$ the operators $Q_t^{n,n+1}$  acting on the bounded Borel functions on $W^{n,n+1}(I^{\circ})$ by,
\begin{align}\label{definitionofsemigroup1}
(Q_t^{n,n+1}f)(x,y)=\int_{W^{n,n+1}(I^{\circ})}^{}q_t^{n,n+1}((x,y),(x',y'))f(x',y')dx'dy'.
\end{align}
Then the following facts hold:
\begin{lem} Assume (\textbf{R}) and (\textbf{BC}) hold for the $L$-diffusion. Then,
\begin{align*}
&Q_t^{n,n+1} 1\le 1 ,\\
& Q_t^{n,n+1}f \ge 0 \ \textnormal{for } f \ge 0.
\end{align*}
\end{lem}

\begin{proof}
The first property will follow from performing the $dx'$ integration first in equation (\ref{definitionofsemigroup1}) with $f\equiv 1$. This is easily done by the very structure of the entries of $q_t^{n,n+1}$: noting that each $x_i'$ variable appears in a single column, then using multilinearity to bring the integrals inside the determinant and the relations:
\begin{align*}
\int_{y_{j-1}'}^{y_j'}{A}_t(x,x')_{ij}dx_j'&=\mathsf{P}_t \textbf{1}_{[l,y'_j]}(x_i)-\mathsf{P}_t \textbf{1}_{[l,y'_{j-1}]}(x_i),\\
\int_{y_{j-1}'}^{y_j'}{C}_t(y,x')_{ij}dx_j'&=-\frac{1}{\hat{m}(y_i)}\partial_{y_i} \mathsf{P}_t \textbf{1}_{[l,y'_j]}(y_i)+\frac{1}{\hat{m}(y_i)}\partial_{y_i} \mathsf{P}_t \textbf{1}_{[l,y'_{j-1}]}(y_i),
\end{align*}
and observing that by (\textbf{BC}) the boundary terms are:
\begin{align*}
\int_{y_N'}^{r}A_t(x,x')_{iN+1}dx_{N+1}'&=1-\mathsf{P}_t \textbf{1}_{[l,y'_{N}]}(x_i), &\int_{y_N'}^{r}C_t(y,x')_{iN+1}dx_{N+1}'=\frac{1}{\hat{m}(y_i)}\partial_{y_i} \mathsf{P}_t \textbf{1}_{[l,y'_N]}(y_i),\\
\int_{l}^{y_1'}A_t(x,x')_{i1}dx_{1}'&=\mathsf{P}_t \textbf{1}_{[l,y'_{1}]}(x_i), &\int_{l}^{y_1'}C_t(y,x')_{i1}dx_{1}'=-\frac{1}{\hat{m}(y_i)}\partial_{y_i} \mathsf{P}_t \textbf{1}_{[l,y'_1]}(y_i),
\end{align*}
we are left with the integral:
\begin{align*}
Q_t^{n,n+1} 1=\int_{W^n(I^{\circ})}^{}\det(\hat{p}_t(y_i,y'_j))^{n}_{i,j=1}dy'\le 1 .
\end{align*}
This is just a restatement of the fact that a Karlin-McGregor semigroup, to be defined shortly in this subsection, is sub-Markov.

The \textit{positivity} preserving property also follows immediately from the original definition, since $\mathbb{P}\big(\boldsymbol{\Phi}_{0,t}(x_i)\le x_i',\boldsymbol{\Phi}_{0,t}(y_j)\le y_j' \ \ \text{for all} \ \ i,j \big)$ is increasing in the $x'_i$ and decreasing in the $y_j$ respectively: Obviously for any $k$ with $x_k'\le \tilde{x}_k'$ and $\tilde{y}_k\le y_k$ each of the events:
\begin{align*}
&\big\{\boldsymbol{\Phi}_{0,t}(x_i)\le x_i',  i \ne k, \boldsymbol{\Phi}_{0,t}(x_k)\le \tilde{x}_k', \boldsymbol{\Phi}_{0,t}(y_j)\le y_j', \text{ for all} \ j \big\},\\
&\big\{\boldsymbol{\Phi}_{0,t}(x_i)\le x_i',\text{ for all} \ i,  \boldsymbol{\Phi}_{0,t}(y_j)\le y_j', j\ne k, \boldsymbol{\Phi}_{0,t}(\tilde{y}_k)\le y_k'\big\},
\end{align*}
contain the event:
\begin{align*}
\big\{\boldsymbol{\Phi}_{0,t}(x_i)\le x_i',\boldsymbol{\Phi}_{0,t}(y_j)\le y_j' \ \ \text{for all} \ \ i,j \big\}.
\end{align*}
Thus, the partial derivatives $\partial_{x_i'}$ and $-\partial_{y_j}$ of  $\mathbb{P}\big(\boldsymbol{\Phi}_{0,t}(x_i)\le x_i',\boldsymbol{\Phi}_{0,t}(y_j)\le y_j' \ \ \text{for all} \ \ i,j \big)$ are positive.
\end{proof}

In fact, $Q_t^{n,n+1}$ defined above, forms a sub-Markov semigroup, associated with a Markov process $Z=(X,Y)$, with possibly finite lifetime, described informally as follows:  the $X$ components follow independent $L$-diffusions reflected off the $Y$ components. More precisely assume that the $L$-diffusion is given as the pathwise unique solution $\mathsf{X}$ to the $SDE$,
  \begin{align*}
  d\mathsf{X}(t)=\sqrt{2a(\mathsf{X}(t))}d\beta(t)+b(\mathsf{X}(t))dt+dK^l(t)-dK^r(t)
  \end{align*}
  where $\beta$ is a standard Brownian motion and $K^l$ and $K^r$ are (possibly zero) positive finite variation processes that only increase when $\mathsf{X}=l$ or $\mathsf{X}=r$, so that $\mathsf{X}\in I$ and $Leb\{t:\mathsf{X}(t)=l \textnormal{ or } r \}=0$ a.s. We write $\mathfrak{s}_{L}$ for the corresponding measurable solution map on path space, namely so that $\mathsf{X}=\mathfrak{s}_{L}\left(\beta\right)$.
  
Consider the following system of $SDEs$ with reflection in $W^{n,n+1}$ which can be described in words as follows. The $Y$ components evolve as $n$ autonomous $\hat{L}$-diffusions stopped when they collide or when (if) they hit $l$ or $r$, and we denote this time by $T^{n,n+1}$. The $X$ components evolve as $n+1$ $L$-diffusions reflected off the $Y$ particles.
  \begin{align}\label{System1SDEs}
  dX_1(t)&=\sqrt{2a(X_1(t))}d\beta_1(t)+b(X_1(t))dt+dK^l(t)-dK_1^+(t)\nonumber ,\\
  dY_1(t)&=\sqrt{2a(Y_1(t))}d\gamma_1(t)+(a'(Y_1(t))-b(Y_1(t)))dt\nonumber ,\\
  dX_2(t)&=\sqrt{2a(X_2(t))}d\beta_2(t)+b(X_2(t))dt+dK_2^-(t)-dK_2^+(t)\nonumber ,\\
  \vdots\\
  dY_n(t)&=\sqrt{2a(Y_n(t))}d\gamma_n(t)+(a'(Y_n(t))-b(Y_n(t)))dt\nonumber , \\
  dX_{n+1}(t)&=\sqrt{2a(X_{n+1}(t))}d\beta_{n+1}(t)+b(X_{n+1}(t))dt+dK_{n+1}^-(t)-dK^r(t)\nonumber.
  \end{align} 
  Here $\beta_1,\cdots,\beta_{n+1},\gamma_1,\cdots,\gamma_n$ are independent standard Brownian motions and the positive finite variation processes $K^l,K^r,K_i^+,K_i^-$ are such that $K^l$ (possibly zero) increases only when $X_1=l$, $K^r$  (possibly zero) increases only when $X_{n+1}=r$,
    $K^+_i(t)$ increases only when $Y_i=X_i$ and $K^-_{i}(t)$ only when $Y_{i-1}=X_i$, so that  $\left(X_1(t)\le Y_1(t)\le \cdots \le X_{n+1}(t);t\ge 0\right)\in W^{n,n+1}(I)$ up to time $T^{n,n+1}$. Note that, $X$ either reflects at $l$ or $r$ or does not visit them at all by our boundary conditions (\ref{bc1l}) and (\ref{bc1r}). The problematic possibility of an $X$ component being trapped between a $Y$ particle and a boundary point and pushed in opposite directions does not arise, since the whole process is then instantly stopped.
    
 The fact that these $SDEs$ are well-posed, so that in particular $(X,Y)$ is Markovian, is proven in Proposition \ref{wellposedness} under a Yamada-Watanabe condition (incorporating a linear growth assumption), that we now define precisely and abbreviate throughout by $(\textbf{YW})$. Note that, the functions $a(\cdot)$ and $b(\cdot)$ initially defined in $I^{\circ}$ can in certain circumstances be continuously extended to the boundary points $l$ and $r$ and this is implicit in assumption (\textbf{YW}).
 \begin{defn}[Assumption (\textbf{YW})]
Let $I$ be an interval with endpoints $l<r$ and suppose $\rho$ is a non-decreasing function from $(0,\infty)$ to itself such that $\int_{0^+}^{}\frac{dx}{\rho(x)}=\infty$. Consider the following condition on functions $a:I\to \mathbb{R}_+$ and $b: I\to \mathbb{R}$,
\begin{align*}
&|\sqrt {a(x)}-\sqrt {a(y)}|^2 \le \rho(|x-y|),\\
&|b(x)-b(y)|\le C|x-y|.
\end{align*}
Moreover, we assume that the functions $\sqrt {a(\cdot)}$ and $b(\cdot)$ are of at most linear growth (for $b(\cdot)$ this is immediate by Lipschitz continuity above). 

We will say that the $L$-diffusion satisfies (\textbf{YW}) if its diffusion and drift coefficients $a$ and $b$ satisfy (\textbf{YW}).
 \end{defn}  
 
 Moreover, by virtue of the following result these $SDEs$ provide a precise description of the dynamics of the two-level process $Z=(X,Y)$ associated with $Q_t^{n,n+1}$. Proposition \ref{dynamics1} below will be proven in Section \ref{SectionTransitionDensities} as either Proposition \ref{transititiondensities1} or Proposition \ref{transitiondensitiesreflecting}, depending on the boundary conditions.
 
 \begin{prop}\label{dynamics1}
Assume $(\textbf{R})$ and $(\textbf{BC+})$ hold for the $L$-diffusion and $(\textbf{YW})$ holds for both the $L$ and $\hat{L}$ diffusions. Then, $Q_t^{n,n+1}$ is the sub-Markov semigroup associated with the (Markovian) system of $SDEs$ (\ref{System1SDEs}) in the sense that if $\textbf{Q}^{n,n+1}_{x,y}$ governs the processes $(X,Y)$ satisfying the $SDEs$ (\ref{System1SDEs}) and with initial condition $(x,y)$ then for any $f$ continuous with compact support and fixed $T>0$,
 \begin{align*}
 \int_{W^{n,n+1}(I^\circ)}^{}q_T^{n,n+1}((x,y),(x',y'))f(x',y')dx'dy'=\textbf{Q}^{n,n+1}_{x,y}\big[f(X(T),Y(T))\textbf{1}(T<T^{n,n+1})\big].
 \end{align*}
 \end{prop}

For further motivation regarding the definition of $Q_t^{n,n+1}$ and moreover, a completely different argument for its semigroup property, that however does not describe explicitly the dynamics of $X$ and $Y$, we refer the reader to the next subsection \ref{coalescinginterpretation}.

We now briefly study some properties of $Q_t^{n,n+1}$, that are immediate from its algebraic structure (with no reference to the $SDEs$ above required). In order to proceed and fix notations for the rest of this section, start by defining the Karlin-McGregor semigroup $P_t^n$ associated with $n$ $L$-diffusions in $I^\circ$ given by the transition density, with $x,y \in W^n(I^\circ)$,
 \begin{align}\label{KM1}
 p^n_t(x,y)dy=\det(p_t(x_i,y_j))_{i,j=1}^ndy .
 \end{align}
Note that, in the case an exit or regular absorbing boundary point exists, $P_t^1$ is the semigroup of the $L$-diffusion \textit{killed} and not absorbed at that point. In particular it is not the same as $\mathsf{P}_t$ which is a Markov semigroup. Similarly, define the Karlin-McGregor semigroup $\hat{P}_t^n$ associated with $n$ $\hat{L}$-diffusions by,
 \begin{align}\label{KM2}
 \hat{p}^n_t(x,y)dy=\det(\hat{p}_t(x_i,y_j))_{i,j=1}^ndy  ,
 \end{align}
 with $x,y \in W^n(I^\circ)$. The same comment regarding absorbing and exit boundary points applies here as well.
 
 Now, define the operators $\Pi_{n,n+1}$, induced by the projections on the $Y$ level as follows with $f$ a bounded Borel function on $W^{n}(I^{\circ})$,
 \begin{align*}
 (\Pi_{n,n+1}f)(x,y)=f(y).
 \end{align*}
 The following proposition immediately follows by performing the $dx'$ integration in the explicit formula for the block determinant (as already implied in the proof that $Q_t^{n,n+1}1\le1$). 
 
 \begin{prop} Assume $(\textbf{R})$ and $(\textbf{BC})$ hold for the $L$-diffusion. For $t>0$ and $f$ a bounded Borel function on $W^{n}(I^{\circ})$ we have,
 \begin{align}\label{firstdynkinintertwining}
 \Pi_{n,n+1}\hat{P}_t^{n}f&=Q_t^{n,n+1}\Pi_{n,n+1}f .
 \end{align}
 \end{prop}
 The fact that $Y$ is distributed as $n$ independent $\hat{L}$-diffusions killed when they collide or when they hit $l$ or $r$ is already implicit in the statement of Proposition \ref{dynamics1}. However, it is also the probabilistic consequence of the proposition above. Namely, the intertwining relation (\ref{firstdynkinintertwining}), being an instance of Dynkin's criterion (see for example Exercise 1.17 Chapter 3 of \cite{RevuzYor}), implies that the evolution of $Y$ is Markovian with respect to the joint filtration of $X$ and $Y$ i.e. of the process $Z$ and we take this as the definition of $Y$ being autonomous. Moreover, $Y$ is evolving according to $\hat{P}_t^n$. Thus, the $Y$ components form an \textit{autonomous} \textit{diffusion} process. Finally, by taking $f\equiv 1$ above we get that the finite lifetime of $Z$ exactly corresponds to the killing time of $Y$, which we denote by $T^{n,n+1}$.

Similarly, we define the kernel $q_t^{n,n}((x,y),(x',y'))dx'dy'$ on $W^{n,n}(I^\circ)$ as follows:
\begin{defn}
For $(x,y),(x',y')\in W^{n,n}(I^\circ)$ define $q_t^{n,n}((x,y),(x',y'))$ by,
\begin{align*}
& q_t^{n,n}((x,y),(x',y'))=\\
&=\frac{\prod_{i=1}^{n}\hat{m}(y'_i)}{\prod_{i=1}^{n}\hat{m}(y_i)}(-1)^n\frac{\partial^n}{\partial_{y_1}\cdots\partial_{y_n}}\frac{\partial^{n}}{\partial_{x'_1}\cdots\partial_{x'_{n}}}\mathbb{P}\big(\boldsymbol{\Phi}_{0,t}(x_i)\le x_i',\boldsymbol{\Phi}_{0,t}(y_j)\le y_j' \ \ \text{for all} \ \ i,j \big).
\end{align*}
\end{defn}
We note that as before $q_t^{n,n}$ can in fact be written out explicitly,
\begin{align}
{q}_t^{n,n}((x,y),(x',y'))=\det\
 \begin{pmatrix}
{A}_t(x,x') & {B}_t(x,y')\\
  {C}_t(y,x') & {D}_t(y,y') 
 \end{pmatrix}
\end{align}
where,
\begin{align*}
{A}_t(x,x')_{ij} &=\partial_{x'_j}\mathsf{P}_t \textbf{1}_{[l,x_j']} (x_i) = p_t(x_i,x_j') ,\\
{B}_t(x,y')_{ij}&=\hat{m}(y'_j)(\mathsf{P}_t \textbf{1}_{[l,y'_j]}(x_i) -\textbf{1}(j> i)),\\
{C}_t(y,x')_{ij}&=-\hat{m}^{-1}(y_i)\partial_{y_i}\partial_{x'_j}\mathsf{P}_t \textbf{1}_{[l,x'_j]}(y_i)=-\mathcal{D}_s^{y_i}p_t(y_i,x_j'),\\
{D}_t(y,y')_{ij}&=-\frac{\hat{m}(y'_j)}{\hat{m}(y_i)}\partial_{y_i} \mathsf{P}_t \textbf{1}_{[l,y'_j]}(y_i)=\hat{p}_t(y_i,y_j').
\end{align*}

\begin{rmk}
Comparing with the $q_t^{n,n+1}$ formulae everything is the same except for the indicator function being $\textbf{1}(j> i)$ instead of $\textbf{1}(j \ge i)$.
\end{rmk}

Define the operator $Q_t^{n,n}$ for $t>0$ acting on bounded Borel functions on $W^{n,n}(I^\circ)$ by,
\begin{align}\label{semigroupdefinition2}
(Q_t^{n,n}f)(x,y)=\int_{W^{n,n}(I^\circ)}^{}q_t^{n,n}((x,y),(x',y'))f(x',y')dx'dy'.
\end{align}
Then with the analogous considerations as for $Q_t^{n,n+1}$ (see subsection \ref{coalescinginterpretation} as well), we can see that $Q_t^{n,n}$ should form a sub-Markov semigroup, to which we can associate a Markov process $Z$, with possibly finite lifetime, taking values in $W^{n,n}(I^\circ)$, the evolution of which we now make precise. 

To proceed as before, we assume that the $L$-diffusion is given by an $SDE$ and we consider the following system of $SDEs$ with reflection in $W^{n,n}$ which can be described as follows. The $Y$ components evolve as $n$ autonomous $\hat{L}$-diffusions killed when they collide or when (if) they hit the boundary point $r$, a time which we denote by $T^{n,n}$. The $X$ components evolve as $n$ $L$-diffusions being kept apart by hard reflection on the $Y$ particles. 
  \begin{align}\label{System2}
  dY_1(t)&=\sqrt{2a(Y_1(t))}d\gamma_1(t)+(a'(Y_1(t))-b(Y_1(t)))dt+dK^l(t),\nonumber\\
  dX_1(t)&=\sqrt{2a(X_1(t))}d\beta_1(t)+b(X_1(t))dt+dK_1^+(t)-dK_1^-(t),\nonumber\\
  \vdots\\
  dY_n(t)&=\sqrt{2a(Y_n(t))}d\gamma_n(t)+(a'(Y_n(t))-b(Y_n(t)))dt,\nonumber\\
  dX_{n}(t)&=\sqrt{2a(X_{n}(t))}d\beta_{n}(t)+b(X_{n}(t))dt+dK_{n}^+(t)-dK^r(t).\nonumber
  \end{align} 
  Here $\beta_1,\cdots,\beta_{n},\gamma_1,\cdots,\gamma_n$ are independent standard Brownian motions and the positive finite variation processes $K^l,K^r, K_i^+,K_i^-$ are such that $\bar{K}^l$ (possibly zero) increases only when $Y_1=l$, $K^r$ (possibly zero) increases only when $X_n=r$,
  $K^+_i(t)$ increases only when $Y_i=X_i$ and $K^-_{i}(t)$ only when $Y_{i-1}=X_i$, so that  $\left(Y_1(t)\le X_1(t)\le \cdots \le X_{n}(t);t\ge 0\right)\in W^{n,n}(I)$ up to $T^{n,n}$. Note that, $Y$ reflects at the boundary point $l$ or does not visit it all and similarly $X$ reflects at $r$ or does not reach it all by our boundary assumptions (\ref{bc2l}) and (\ref{bc2r}). The intuitively problematic issue of $Y_n$ pushing $X_n$ upwards at $r$ does not arise since the whole process is stopped at such instance. 
    
 That these $SDEs$ are well-posed, so that in particular $(X,Y)$ is Markovian, again follows from the arguments of Proposition \ref{wellposedness}. As before, we have the following precise description of the dynamics of the two-level process $Z=(X,Y)$ associated with $Q_t^{n,n}$.

 \begin{prop}\label{dynamics2}
Assume $(\textbf{R})$ and $(\textbf{BC+})$ hold for the $L$-diffusion and $(\textbf{YW})$ holds for both the $L$ and $\hat{L}$ diffusions. Then, $Q_t^{n,n}$ is the sub-Markov semigroup associated with the (Markovian) system of $SDEs$ (\ref{System2}) in the sense that if $\textbf{Q}^{n,n}_{x,y}$ governs the processes $(X,Y)$ satisfying the $SDEs$ (\ref{System2}) with initial condition $(x,y)$ then for any $f$ continuous with compact support and fixed $T>0$,
 \begin{align*}
 \int_{W^{n,n}(I^\circ)}^{}q_T^{n,n}((x,y),(x',y'))f(x',y')dx'dy'=\textbf{Q}^{n,n}_{x,y}\big[f(X(T),Y(T))\textbf{1}(T<T^{n,n})\big].
 \end{align*}
 \end{prop}
 
 We also define, analogously to before, an operator $\Pi_{n,n}$, induced by the projection on the $Y$ level by,
 \begin{align*}
 (\Pi_{n,n}f)(x,y)=f(y).
 \end{align*}
We have the following proposition which immediately follows by performing the $dx'$ integration in equation (\ref{semigroupdefinition2}).
 
 \begin{prop} Assume $(\textbf{R})$ and $(\textbf{BC})$ hold for the $L$-diffusion. For $t>0$ and $f$ a bounded Borel function on $W^{n}(I^\circ)$ we have,
 \begin{align}
 \Pi_{n,n}\hat{P}_t^{n}f&=Q_t^{n,n}\Pi_{n,n}f.
 \end{align}
 \end{prop}
  This, again implies that the evolution of $Y$ is Markovian with respect to the joint filtration of $X$ and $Y$. Furthermore, $Y$ is distributed as $n$ $\hat{L}$-diffusions killed when they collide or when (if) they hit the boundary point $r$ (note the difference here to $W^{n,n+1}$ is because of the asymmetry between $X$ and $Y$ and our standing assumption (\ref{bc2l}) and (\ref{bc2r})). Hence, the $Y$ components form a \textit{diffusion} process and they are \textit{autonomous}. The finite lifetime of $Z$ analogously to before (by taking $f\equiv 1$ in the proposition above), exactly corresponds to the killing time of $Y$ which we denote by $T^{n,n}$. As before, this is already implicit in the statement of Proposition \ref{dynamics2}.

Finally, we can define the kernel $q_t^{n+1,n}((x,y),(x',y'))dx'dy'$ on $W^{n+1,n}(I^\circ)$ in an analogous way and also the operator $Q_t^{n+1,n}$ for $t>0$ acting on bounded Borel functions on $W^{n+1,n}(I^\circ)$ as well. The description of the associated process $Z$ in $W^{n+1,n}(I^\circ)$ in words is as follows. The $Y$ components evolve as $n+1$ autonomous $\hat{L}$-diffusions killed when they collide (by our boundary conditions (\ref{bc3l}) and (\ref{bc3r}) if the $Y$ particles do visit $l$ or $r$ they are reflecting there) and the $X$ components evolve as $n$ $L$-diffusions reflected on the $Y$ particles. These dynamics can be described in terms of $SDEs$ with reflection under completely analogous assumptions. The details are omitted.

\subsection{Stochastic coalescing flow interpretation} \label{coalescinginterpretation}

The definition of $q_t^{n,n+1}$, and similarly of $q_t^{n,n}$, might look rather mysterious and surprising. It is originally motivated from considering stochastic coalescing flows. Briefly, the finite system $\left(\boldsymbol{\Phi}_{0,t}(x_1),\cdots,\boldsymbol{\Phi}_{0,t}(x_n);t \ge 0\right)$ can be extended to an infinite system of coalescing $L$-diffusions starting from each space time point and denoted by $(\boldsymbol{\Phi}_{s,t}(\cdot),s\le t)$. This is well documented in Theorem 4.1 of \cite{LeJan} for example. The random family of maps $(\boldsymbol{\Phi}_{s,t},s\le t)$ from $I$ to $I$ enjoys among others the following natural looking and intuitive properties: the \textit{cocycle} or \textit{flow} property $\boldsymbol{\Phi}_{t_1,t_3}=\boldsymbol{\Phi}_{t_2,t_3}\circ\boldsymbol{\Phi}_{t_1,t_2}$, \textit{independence of its increments} $\boldsymbol{\Phi}_{t_1,t_2}\perp\boldsymbol{\Phi}_{t_3,t_4}$ for $t_2\le t_3$ and \textit{stationarity} $\boldsymbol{\Phi}_{t_1,t_2}\overset{law}{=}\boldsymbol{\Phi}_{0,t_2-t_1}$. Finally, we can consider its generalized inverse by $\boldsymbol{\Phi}^{-1}_{s,t}(x)=sup\{w:\boldsymbol{\Phi}_{s,t}(w)\le x\}$ which is well defined since $\boldsymbol{\Phi}_{s,t}$ is non-decreasing.

With these notations in place $q_t^{n,n+1}$ can also be written as,

\begin{align}\label{flowtransition1}
 q_t^{n,n+1}((x,y),(x',y'))dx'dy=\frac{\prod_{i=1}^{n}\hat{m}(y'_i)}{\prod_{i=1}^{n}\hat{m}(y_i)}\mathbb{P}\big(\boldsymbol{\Phi}_{0,t}(x_i)\in dx_i',\boldsymbol{\Phi}^{-1}_{0,t}(y'_j)\in dy_j \ \ \text{for all} \ \ i,j \big).
\end{align}

We now sketch an argument that gives the semigroup property $ Q_{t+s}^{n,n+1}=Q_t^{n,n+1}Q_s^{n,n+1}$. We do not try to give all the details that would render it completely rigorous, mainly because it cannot be used to precisely describe the dynamics of $Q_t^{n,n+1}$, but nevertheless all the main steps are spelled out.

All equalities below should be understood after being integrated with respect to $dx''$ and $dy$ over arbitrary Borel sets. The first equality is by definition. The second equality follows from the \textit{cocycle} property and conditioning on the values of $\boldsymbol{\Phi}_{0,s}(x_i)$ and $\boldsymbol{\Phi}^{-1}_{s,s+t}(y''_j)$. Most importantly, this is where the \textit{boundary behaviour assumptions} (\ref{bc1l}) and (\ref{bc1r}) we made at the beginning of this subsection are used. These ensure that no possible contributions from atoms on $\partial I$ are missed; namely the random variable $\boldsymbol{\Phi}_{0,s}(x_i)$ is supported (its distribution gives full mass) in $I^\circ$. Moreover, it is not too hard to see from the coalescing property of the flow that, we can restrict the integration over $(x',y')\in W^{n,n+1}(I^\circ)$ for otherwise the integrand vanishes. Finally, the third equality follows from \textit{independence} of the increments and the fourth one by \textit{stationarity} of the flow.
\begin{align*}
 &q_{s+t}^{n,n+1}((x,y),(x'',y''))dx''dy=\frac{\prod_{i=1}^{n}\hat{m}(y''_i)}{\prod_{i=1}^{n}\hat{m}(y_i)}\mathbb{P}\big(\boldsymbol{\Phi}_{0,s+t}(x_i)\in dx_i'',\boldsymbol{\Phi}^{-1}_{0,s+t}(y''_j)\in dy_j \ \ \text{for all} \ \ i,j \big) \\
 &=\frac{\prod_{i=1}^{n}\hat{m}(y''_i)}{\prod_{i=1}^{n}\hat{m}(y_i)}\int_{(x',y')\in W^{n,n+1}(I^\circ)}^{}\mathbb{P}\big(\boldsymbol{\Phi}_{0,s}(x_i)\in dx_i',\boldsymbol{\Phi}_{s,s+t}(x_i')\in dx_i'',\boldsymbol{\Phi}^{-1}_{0,s}(y'_j)\in dy_j,\boldsymbol{\Phi}^{-1}_{s,s+t}(y''_j)\in dy_j' \big)
 \\ 
& =\int_{(x',y')\in W^{n,n+1}(I^\circ)}^{} \frac{\prod_{i=1}^{n}\hat{m}(y'_i)}{\prod_{i=1}^{n}\hat{m}(y_i)}\mathbb{P}\big(\boldsymbol{\Phi}_{0,s}(x_i)\in dx_i',\boldsymbol{\Phi}^{-1}_{0,s}(y'_j)\in dy_j \big)\\
&\times \frac{\prod_{i=1}^{n}\hat{m}(y''_i)}{\prod_{i=1}^{n}\hat{m}(y'_i)}\mathbb{P}\big(\boldsymbol{\Phi}_{s,s+t}(x'_i)\in dx_i'',\boldsymbol{\Phi}^{-1}_{s,s+t}(y''_j)\in dy_j' \big)\\
&=\int_{(x',y')\in W^{n,n+1}(I^\circ)}^{} q_s^{n,n+1}((x,y),(x',y'))q_t^{n,n+1}((x',y'),(x'',y''))dx'dy' dx''dy .
\end{align*}

\subsection{Intertwining and Markov functions}
In this subsection $(n_1,n_2)$ denotes one of $\{(n,n-1),(n,n),(n,n+1)\}$. First, recall the definitions of $P_t^n$ and $\hat{P}_t^n$ given in (\ref{KM1}) and (\ref{KM2}) respectively. Similarly, we record here again, the following proposition and recall that it can in principle completely describe the evolution of the $Y$ particles and characterizes the finite lifetime of the process $Z$ as the killing time of $Y$.

\begin{prop}\label{DynkinProposition}Assume $(\textbf{R})$ and $(\textbf{BC})$ hold for the $L$-diffusion. For $t>0$ and $f$ a bounded Borel function on $W^{n_1}(I^\circ)$ we have,
\begin{align}
\Pi_{n_1,n_2}\hat{P}_t^{n_1}f&=Q_t^{n_1,n_2}\Pi_{n_1,n_2}f.
\end{align}
\end{prop}

Now, we define the following integral operator $\Lambda_{n_1,n_2}$ acting on Borel functions on $W^{n_1,n_2}(I^\circ)$, whenever $f$ is integrable as,
\begin{align}\label{PreIntertwiningKernel}
(\Lambda_{n_1,n_2}f)(x) &= \int_{W^{n_1,n_2}(x)}^{}\prod_{i=1}^{n_1}\hat{m}(y_i)f(x,y)dy,
\end{align}
where we remind the reader that $\hat{m}(\cdot)$ is the density with respect to Lebesgue measure of the speed measure of the diffusion with generator $\hat{L}$.

The following intertwining relation is the fundamental ingredient needed for applying the theory of Markov functions, originating with the seminal paper of Rogers and Pitman \cite{RogersPitman}. This proposition directly follows by performing the $dy$ integration in the explicit formula of the block determinant (or alternatively by invoking the coalescing property of the stochastic flow $\left(\boldsymbol{\Phi}_{s,t}(\cdot);s \le t\right)$ and the original definitions).

 \begin{prop} Assume $(\textbf{R})$ and $(\textbf{BC})$ hold for the $L$-diffusion. For $t>0$ we have the following equality of positive kernels,
 \begin{align} 
 P_t^{n_2}\Lambda_{n_1,n_2}&=\Lambda_{n_1,n_2}Q_t^{n_1,n_2}.
 \end{align}
 \end{prop}

 Combining the two propositions above gives the following relation for the Karlin-McGregor semigroups,
\begin{align}\label{KMintertwining}
P_t^{n_2}\Lambda_{n_1,n_2}\Pi_{n_1,n_2}&=\Lambda_{n_1,n_2}\Pi_{n_1,n_2}\hat{P}_t^{n_1}.
\end{align}
Namely, the two semigroups are themselves intertwined with kernel,
\begin{align*}
\left(\Lambda_{n_1,n_2}\Pi_{n_1,n_2}f\right)(x)=\int_{W^{n_1,n_2}(x)}^{}\prod_{i=1}^{n_1}\hat{m}(y_1)f(y)dy.
\end{align*}
This implies the following. Suppose $\hat{h}_{n_1}$ is a strictly positive (in $\mathring{W}^{n_1}$) eigenfunction for $\hat{P}_t^{n_1}$ namely, $\hat{P}_t^{n_1}\hat{h}_{n_1}=e^{\lambda_{n_1}t}\hat{h}_{n_1}$, then (with both sides possibly being infinite),
\begin{align*}
(P_t^{n_2}\Lambda_{n_1,n_2}\Pi_{n_1,n_2}\hat{h}_{n_1})(x)&=e^{\lambda_{n_1}t}(\Lambda_{n_1,n_2}\Pi_{n_1,n_2}\hat{h}_{n_1})(x).
\end{align*}

We are interested in strictly positive eigenfunctions because they allow us to define Markov processes, however non positive eigenfunctions can be built this way as well. 

We now finally arrive at our main results. We need to make precise one more notion, already referenced several times in the introduction. For a possibly sub-Markov semigroup $\left(\mathfrak{P}_t;t \ge 0\right)$ or more generally, for fixed $t$, a sub-Markov kernel with eigenfunction $\mathfrak{h}$ with eigenvalue $e^{ct}$ we define the Doob's $h$-transform by $
e^{-ct}\mathfrak{h}^{-1}\circ \mathfrak{P}_t \circ \mathfrak{h}$. 
Observe that, this is now an honest Markov semigroup (or Markov kernel). 

If $\hat{h}_{n_1}$ is a strictly positive in $\mathring{W}^{n_1}$ eigenfunction for $\hat{P}_t^{n_1}$ then so is the function $\hat{h}_{n_1,n_2}(x,y)=\hat{h}_{n_1}(y)$ for $Q_t^{n_1,n_2}$ from Proposition \ref{DynkinProposition}. We can thus define the proper Markov kernel $Q_t^{n_1,n_2,\hat{h}_{n_1}}$ which is the $h$-transform of $Q_t^{n_1,n_2}$ by $\hat{h}_{n_1}$. Define $h_{n_2}(x)$, strictly positive in $\mathring{W}^{n_2}$, as follows, assuming that the integrals are finite in the case of $W^{n,n}(I^\circ)$ and $W^{n+1,n}(I^\circ)$,
\begin{align*}
h_{n_2}(x)=(\Lambda_{n_1,n_2}\Pi_{n_1,n_2}\hat{h}_{n_1})(x),
\end{align*}
and the Markov Kernel $\Lambda^{\hat{h}_{n_1}}_{n_1,n_2}(x,\cdot)$ with $x \in \mathring{W}^{n_2}$ by,
\begin{align*}
(\Lambda^{\hat{h}_{n_1}}_{n_1,n_2}f)(x)=\frac{1}{h_{n_2}(x)}\int_{W^{n_1,n_2}(x)}^{}\prod_{i=1}^{n_1}\hat{m}(y_i)\hat{h}_{n_1}(y)f(x,y)dy.
\end{align*}
Finally, defining $P_t^{n_2,h_{n_2}}$ to be the Karlin-McGregor semigroup $P_t^{n_2}$  $h$-transformed by $h_{n_2}$ we obtain:
\begin{prop} \label{Master} Assume $(\textbf{R})$ and $(\textbf{BC})$ hold for the $L$-diffusion. Let $Q_t^{n_1,n_2}$ denote one of the operators induced by the sub-Markov kernels on $W^{n_1,n_2}(I^\circ)$ defined in the previous subsection. Let $\hat{h}_{n_1}$ be a strictly positive eigenfunction for $\hat{P}_t^{n_1}$ and assume that $h_{n_2}(x)=(\Lambda_{n_1,n_2}\Pi_{n_1,n_2}\hat{h}_{n_1})(x)$ is finite in $W^{n_2}(I^\circ)$, so that in particular $\Lambda^{\hat{h}_{n_1}}_{n_1,n_2}$ is a Markov kernel. Then, with the notations of the preceding paragraph we have the following relation for $t>0$,
\begin{align} \label{MasterIntertwining}
P_t^{n_2,h_{n_2}}\Lambda^{\hat{h}_{n_1}}_{n_1,n_2}f&=\Lambda^{\hat{h}_{n_1}}_{n_1,n_2}Q_t^{n_1,n_2,\hat{h}_{n_2}}f ,
\end{align}
with $f$ a bounded Borel function in $W^{n_1,n_2}(I^{\circ})$.
\end{prop}
This intertwining relation and the theory of Markov functions (see Section 2 of \cite{RogersPitman} for example) immediately imply the following corollary:
\begin{cor}\label{mastercorollary}
Assume $Z=(X,Y)$ is a Markov process with semigroup $Q_t^{n_1,n_2,\hat{h}_{n_2}}$, then the $X$ component is distributed as a Markov process with semigroup $P_t^{n_2,h_{n_2}}$ started from $x$ if $(X,Y)$ is started from $\Lambda^{\hat{h}_{n_1}}_{n_1,n_2}(x,\cdot)$. Moreover, the conditional distribution of $Y(t)$ given $\left(X(s);s \le t\right)$ is $\Lambda^{\hat{h}_{n_1}}_{n_1,n_2}(X(t),\cdot)$.
\end{cor}

We give a final definition in the case of $W^{n,n+1}$ only, that has a natural analogue for  $W^{n,n}$ and  $W^{n+1,n}$ (we shall elaborate on the notion introduced below in Section 5.1 on well-posedness of $SDEs$ with reflection). Take $Y=(Y_1,\cdots,Y_n)$ to be an $n$-dimensional system of \textit{non-intersecting} paths in $\mathring{W}^n(I^{\circ})$, so that in particular $Y_1<Y_2<\cdots<Y_n$. Then, by \textit{$X$ is a system of $n+1$ $L$-diffusions reflected off $Y$} we mean processes $\left(X_1(t),\cdots,X_{n+1}(t);t \ge 0\right)$, satisfying $X_1(t)\le Y_1(t) \le X_2(t) \le \cdots \le X_{n+1}(t)$ for all $t \ge 0$, and so that the following $SDEs$ hold,
\begin{align}\label{systemofreflectingLdiffusions}
dX_1(t)&=\sqrt{2a(X_1(t))}d\beta_1(t)+b(X_1(t))dt+dK^l(t)-{dK_1^+(t)},\nonumber\\
\vdots\nonumber\\
dX_j(t)&=\sqrt{2a(X_j(t))}d\beta_j(t)+b(X_j(t))dt+{dK_j^-(t)}-{dK_j^+(t)},\\
\vdots\nonumber\\
dX_{n+1}(t)&=\sqrt{2a(X_{n+1}(t))}d\beta_{n+1}(t)+b(X_{n+1}(t))dt+{dK_{n+1}^-(t)}-dK^r(t),\nonumber
\end{align}
where the positive finite variation processes $K^l,K^r,K_i^+,K_i^-$ are such that $K^l$ increases only when $X_1=l$, $K^r$ increases only when $X_{n+1}=r$,
$K^+_i(t)$ increases only when $Y_i=X_i$ and $K^-_{i}(t)$ only when $Y_{i-1}=X_i$, so that  $(X_1(t)\le Y_1(t)\le \cdots \le X_{n+1}(t))\in W^{n,n+1}(I)$ forever. Here $\beta_1,\cdots,\beta_{n+1}$ are independent standard Brownian motions which are moreover \textit{independent} of $Y$. The reader should observe that the dynamics between $(X,Y)$ are exactly the ones prescribed in the system of $SDEs$ (\ref{System1SDEs}) with the difference being that now the process has infinite lifetime. This can be achieved from (\ref{System1SDEs}) by $h$-transforming the $Y$ process as explained in this section to have infinite lifetime. By pathwise uniqueness of solutions to reflecting $SDEs$, with coefficients satisfying $(\mathbf{YW})$, in continuous time-dependent domains proven in Proposition \ref{wellposedness}, under any absolutely continuous change of measure for the $(X,Y)$-process that depends only on $Y$ (a Doob $h$-transform in particular), the equations (\ref{systemofreflectingLdiffusions}) still hold with the $\beta_i$ independent Brownian motions which moreover remain independent of the $Y$ process. We thus arrive at our main theorem:

\begin{thm} \label{MasterDynamics}
Assume $(\textbf{R})$ and $(\textbf{BC+})$ hold for the $L$-diffusion and $(\textbf{YW})$ holds for both the $L$ and $\hat{L}$ diffusions. Moreover, assume $\hat{h}_n$ is a strictly positive eigenfunction for $\hat{P}_t^{n}$. Suppose $Y$ consists of $n$ non-intersecting $\hat{L}$-diffusions $h$-transformed by $\hat{h}_n$, with transition semigroup $\hat{P}_{t}^{n,\hat{h}_n}$, and $X$ is a system of $n+1$ $L$-diffusions reflected off $Y$ started according to the distribution $\Lambda^{\hat{h}_{n}}_{n,n+1}(x,\cdot)$ for some $x\in\mathring{W}^{n+1}(I)$. Then $X$ is distributed as a diffusion process with semigroup $P_t^{n+1,h_{n+1}}$ started from $x$, where $h_{n+1}=\Lambda_{n,n+1}\Pi_{n,n+1}\hat{h}_{n}$.
\end{thm}
\begin{proof}
By Proposition \ref{dynamics1} and the discussion above, the process (X,Y) evolves according to the Markov semigroup $Q_t^{n_1,n_2,\hat{h}_{n_2}}$. Then, an application of the Rogers-Pitman Markov functions criterion in \cite{RogersPitman} with the function $\phi(x,y)=x$ and the intertwining (\ref{MasterIntertwining}) gives that, under the initial law $\Lambda^{\hat{h}_{n}}_{n,n+1}(x,\cdot)$ for $(X,Y)$, $\left(X(t);t\ge 0\right)$ is a Markov process with semigroup $P_t^{n+1,h_{n+1}}$ started from $x$, in particular a diffusion.
\end{proof}

The statement and proof of the result for $W^{n,n}$ and $W^{n+1,n}$ is completely analogous.

 Finally, the intertwining relation (\ref{MasterIntertwining}) also allows us to start the two-level process $(X,Y)$ from a degenerate point, in particular the system of reflecting $SDEs$ when some of the $Y$ coordinates coincide, as long as starting the process with semigroup $P_t^{n_2,h_{n_2}}$ from such a degenerate point is valid. Suppose $\big(\mu_t^{n_2,h_{n_2}}\big)_{t>0}$ is an entrance law for $P_t^{n_2,h_{n_2}}$, namely for $t,s>0$,
\begin{align*}
\mu_s^{n_2,h_{n_2}}P_t^{n_2,h_{n_2}}=\mu_{t+s}^{n_2,h_{n_2}} ,
\end{align*}
then we have the following corollary, which is obtained immediately by applying $\mu_t^{n_2,h_{n_2}}$ to both sides of (\ref{MasterIntertwining}):
\begin{cor}\label{EntranceLawCorollary}
Under the assumptions above, if $\left(\mu_s^{n_2,h_{n_2}}\right)_{s>0}$ is an entrance law for the process with semigroup $P_t^{n_2,h_{n_2}}$ then $\left(\mu_s^{n_2,h_{n_2}}\Lambda^{\hat{h}_{n_1}}_{n_1,n_2}\right)_{s>0}$
forms an entrance law for process $(X,Y)$ with semigroup $Q_t^{n_1,n_2,\hat{h}_{n_1}}$.
\end{cor}

Hence, the statement of Theorem \ref{MasterDynamics} generalizes, so that if $X$ is a system of $L$-diffusions reflected off $Y$ started according to an entrance law, then $X$ is again itself distributed as a Markov process.

The entrance laws that we will be concerned with in this paper will correspond to starting the process with semigroup $P_t^{n_2,h_{n_2}}$ from a single point $(x,\cdots,x)$ for some $x\in I$. These will be given by so called time dependent \textit{biorthogonal ensembles}, namely measures of the form,
\begin{align}\label{biorthogonalensemble}
\det\left(f_i(t,x_j)\right)^{n_2}_{i,j=1}\det\left(g_i(t,x_j)\right)^{n_2}_{i,j=1} \ .
\end{align}
Under some further assumptions on the Taylor expansion of the one dimensional transition density $p_t(x,y)$ they will be given by so called \textit{polynomial ensembles}, where one of the determinant factors is the Vandermonde determinant,
\begin{align}\label{polynomialensemble}
\det\left(\phi_i(t,x_j)\right)^{n_2}_{i,j=1}\det\left(x_j^{i-1}\right)^{n_2}_{i,j=1} \ .
\end{align}
A detailed discussion is given in the Appendix.

\section{Applications and examples}

Applying the theory developed in the previous section we will now show how some of the known examples of diffusions in Gelfand-Tsetlin patterns fit into this framework and construct new processes of this kind. In particular we will treat all the diffusions associated with Random Matrix eigenvalues, a model related to Plancherel growth that involves a wall, examples coming from Sturm-Liouville semigroups and finally point out the connection to strong stationary times and superpositions and decimations of Random Matrix ensembles. 

First, recall that the space of Gelfand-Tsetlin patterns of depth $N$ denoted by $\mathbb{GT}(N)$ is defined to be,
\begin{align*}
\left\{\left(x^{(1)},\cdots,x^{(N)}\right):x^{(n)}\in W^{n}, \ x^{(n)}\prec x^{(n+1)}\right\},
\end{align*}
and also the space of symplectic Gelfand-Tsetlin patterns of depth $N$ denoted by $\mathbb{GT}_{\textbf{s}}(N)$ is given by,
\begin{align*}
\left\{\left(x^{(1)},\hat{x}^{(1)}\cdots,x^{(N)}, \hat{x}^{(N)}\right):x^{(n)},\hat{x}^{(n)}\in W^{n}, \ x^{(n)}\prec \hat{x}^{(n)}\prec x^{(n+1)} \right\}.
\end{align*}
Please note the minor discrepancy in the definition of $\mathbb{GT}(N)$ with the notation used for $W^{n,n+1}$: here for two consecutive levels $x^{(n)}\in W^n, x^{(n+1)}\in W^{n+1}$ in the Gelfand-Tsetlin pattern the pair $(x^{(n+1)},x^{(n)})\in W^{n,n+1}$ and not the other way round.
\subsection{Concatenating two-level processes} 
We will describe the construction for $\mathbb{GT}$, with the extension to $\mathbb{GT}_{\textbf{s}}$ being analogous. Let us fix an interval $I$ with endpoints $l<r$ and let $L_n$ for $n=1,\cdots, N$ be a sequence of diffusion process generators in $I$ (satisfying (\ref{bc1l}) and (\ref{bc1r})) given by,
\begin{align}
L_n=a_n(x)\frac{d^2}{dx^2}+b_n(x)\frac{d}{dx}.
\end{align}
We will moreover denote their transition densities with respect to Lebesgue measure by $p_t^n(\cdot,\cdot)$.

We want to consider a process $\left(\mathbb{X}(t);t \ge 0\right)=\left(\left(\mathbb{X}^{(1)}(t),\cdots,\mathbb{X}^{(N)}(t)\right);t \ge 0\right)$ taking values in $\mathbb{GT}(N)$ so that, for each $2 \le n \le N ,\ \mathbb{X}^{(n)}$ consists of $n$ independent $L_n$ diffusions reflected off the paths of $\mathbb{X}^{(n-1)}$. More precisely we consider the following system of reflecting $SDEs$, with $1 \le i \le n \le N$, initialized in $\mathbb{GT}(N)$ and stopped at the stopping time $\tau_{\mathbb{GT}(N)}$ to be defined below,
\begin{align}\label{GelfandTsetlinSDEs}
d\mathbb{X}_i^{(n)}(t)=\sqrt{2a_n\left(\mathbb{X}_i^{(n)}(t)\right)}d\beta_i^{(n)}(t)+b_n\left(\mathbb{X}_i^{(n)}(t)\right)dt+dK_i^{(n),-}-dK_i^{(n),+},
\end{align}
driven by an array $\left(\beta_i^{(n)}(t);t\ge 0, 1 \le i \le n \le N\right)$ of $\frac{N(N+1)}{2}$ independent standard Brownian motions. The positive finite variation processes $K_i^{(n),-}$ and $K_i^{(n),+}$ are such that $K_i^{(n),-}$ increases only when $\mathbb{X}_i^{(n)}=\mathbb{X}_{i-1}^{(n-1)}$, $K_i^{(n),+}$ increases only when $\mathbb{X}_i^{(n)}=\mathbb{X}_{i}^{(n-1)}$ with $K_1^{(N),-}$ increasing when $\mathbb{X}_1^{(N)}=l$ and $K_N^{(N),+}$  increasing when $\mathbb{X}_N^{(N)}=r$, so that $\mathbb{X}=\left(\mathbb{X}^{(1)},\cdots,\mathbb{X}^{(N)}\right)$ stays in $\mathbb{GT}(N)$ forever. The stopping $\tau_{\mathbb{GT}(N)}$ is given by,
\begin{align*}
\tau_{\mathbb{GT}(N)}=\inf \big\{ t \ge 0: \exists \ (n,i,j) \ 2 \le n \le N-1, 1 \le i < j\le n \textnormal{ s.t. } \mathbb{X}_i^{(n)}(t)=\mathbb{X}_j^{(n)}(t) \big\}.
\end{align*}
Stopping at $\tau_{\mathbb{GT}(N)}$ takes care of the problematic possibility of two of the time dependent barriers coming together. It will turn out that $\tau_{\mathbb{GT}(N)}=\infty$ almost surely under certain initial conditions of interest to us given in Proposition \ref{multilevelproposition} below; this will be the case since then each level $\mathbb{X}^{(n)}$ will evolve according to a Doob's $h$-transform and thus consisting of non-intersecting paths.
That the system of reflecting $SDEs$ (\ref{GelfandTsetlinSDEs}) above is well-posed, under a Yamada-Watanabe condition on the coefficients $\left(\sqrt{a_n},b_n\right)$ for $1\le n \le N$, follows (inductively) from Proposition \ref{wellposedness}.

We would like Theorem \ref{MasterDynamics} to be applicable to each pair $(\mathbb{X}^{(n-1)},\mathbb{X}^{(n)})$, with $X=\mathbb{X}^{(n)}$ and $Y=\mathbb{X}^{(n-1)}$. To this end, for $n=2,\cdots,N$, suppose that $\mathbb{X}^{(n-1)}$ is distributed according to the following $h$-transformed Karlin-McGregor semigroup by the strictly positive in $\mathring{W}^{n-1}$ eigenfunction $g_{n-1}$ with eigenvalue $e^{c_{n-1}t}$,
\begin{align*}
e^{-c_{n-1}t}\frac{g_{n-1}(y_1,\cdots,y_{n-1})}{g_{n-1}(x_1,\cdots,x_{n-1})}\det\left(\widehat{p}_t^{n}(x_i,y_j)\right)_{i,j=1}^{n-1} ,
\end{align*}
where $\widehat{p}^{n}_t(\cdot,\cdot)$ denotes the transition density associated with the dual $\widehat{L}_n$ (killed at an exit of regular absorbing boundary point) of $L_n$. We furthermore, denote by $\widehat{m}^{n}(\cdot)$ the density with respect to Lebesgue measure of the speed measure of $\widehat{L}_n$. Then, Theorem \ref{MasterDynamics} gives that under a special initial condition (stated therein) for the joint dynamics of $(\mathbb{X}^{(n-1)},\mathbb{X}^{(n)})$, with $X=\mathbb{X}^{(n)}$ and $Y=\mathbb{X}^{(n-1)}$, the projection on $\mathbb{X}^{(n)}$ is distributed as the $G_{n-1}$ $h$-transform of $n$ independent $L_n$ diffusions, thus consisting of non-intersecting paths, where $G_{n-1}$ is given by,
\begin{align}\label{eigenGelfand}
G_{n-1}(x_1,\cdots,x_{n})=\int_{W^{n-1,n}(x)}^{}\prod_{i=1}^{n-1}\widehat{m^n}(y_i)g_{n-1}(y_1,\cdots,y_{n-1})dy_1\cdots dy_{n-1}.
\end{align}
Consistency then demands, by comparing $(\mathbb{X}^{(n-1)},\mathbb{X}^{(n)})$ and $(\mathbb{X}^{(n)},\mathbb{X}^{(n+1)})$, the following condition between the transition kernels (which is also sufficient as we see below for the construction of a consistent process $(\mathbb{X}^{(1)},\cdots,\mathbb{X}^{(N)})$), for $t>0,x,y \in \mathring{W}^n$,
\begin{align}\label{consistencyGelfand}
e^{-c_{n-1}t}\frac{G_{n-1}(y_1,\cdots,y_{n})}{G_{n-1}(x_1,\cdots,x_{n})}\det\left(p_t^n(x_i,y_j)\right)_{i,j=1}^{n}=e^{-c_nt}\frac{g_n(y_1,\cdots,y_{n})}{g_n(x_1,\cdots,x_{n})}\det\left(\widehat{p}_t^{n+1}(x_i,y_j)\right)_{i,j=1}^{n}.
\end{align}
Denote the semigroup associated with these densities by $\left(\mathfrak{P}^{(n)}(t);t >0\right)$ and also define the Markov kernels $\mathfrak{L}_{n-1}^n(x,dy)$ for $x \in \mathring{W}^n$ by,
\begin{align*}
\mathfrak{L}_{n-1}^n(x,dy)=\frac{\prod_{i=1}^{n-1}\widehat{m^n}(y_i)g_{n-1}(y_1,\cdots,y_{n-1})}{G_{n-1}(x_1,\cdots,x_{n})}\textbf{1}\left(y \in W^{n-1,n}(x)\right)dy_1\cdots dy_{n-1}.
\end{align*}

Then, by inductively applying Theorem \ref{MasterDynamics}, we easily see the following Proposition holds:

\begin{prop}\label{multilevelproposition}
Assume $(\textbf{R})$ and $(\textbf{BC+})$ hold for the $L_n$-diffusion and $(\textbf{YW})$ holds for the pairs of $(L_n,\hat{L}_n)$-diffusions for $2\le n \le N$. Moreover, suppose that there exist functions $g_n$ and $G_n$ so that the consistency relations (\ref{eigenGelfand}) and (\ref{consistencyGelfand}) hold. Let $\nu_N(dx)$ be a measure supported in $\mathring{W}^N$. Consider the process $\left(\mathbb{X}(t);t \ge 0\right)=\left(\left(\mathbb{X}^{(1)}(t),\cdots,\mathbb{X}^{(N)}(t)\right);t \ge 0\right)$ in $\mathbb{GT}(N)$ satisfying the $SDEs$ (\ref{GelfandTsetlinSDEs}) and initialized according to,
\begin{align}\label{Gibbsinitial}
\nu_N(dx^{(N)})\mathfrak{L}_{N-1}^N(x^{(N)},dx^{(N-1)}) \cdots \mathfrak{L}_{1}^2(x^{(2)},dx^{(1)}).
\end{align}
Then $\tau_{\mathbb{GT}(N)}=\infty$ almost surely, $\left(\mathbb{X}^{(n)}(t);t \ge 0\right)$ for $1 \le n \le N$ evolves according to $\mathfrak{P}^{(n)}(t)$ and for fixed $T>0$ the law of $\left(\mathbb{X}^{(1)}(T),\cdots,\mathbb{X}^{(N)}(T)\right)$ is given by,
\begin{align}\label{Gibbsevolved}
\left(\nu_N\mathfrak{P}^{(N)}_T\right)(dx^{(N)})\mathfrak{L}_{N-1}^N(x^{(N)},dx^{(N-1)}) \cdots \mathfrak{L}_{1}^2(x^{(2)},dx^{(1)}).
\end{align}
\end{prop}

\begin{proof}
For $n=2$ this is the statement of Theorem \ref{MasterDynamics}. Assume that the proposition is proven for $n=N-1$. Observe that, an initial condition of the form (\ref{Gibbsinitial}) in $\mathbb{GT}(N)$ gives rise to an initial condition of the same form in $\mathbb{GT}(N-1)$:
\begin{align*}
\tilde{\nu}_{N-1}(dx^{(N-1)})\mathfrak{L}_{N-2}^{N-1}(x^{(N-1)},dx^{(N-2)}) \cdots \mathfrak{L}_{1}^2(x^{(2)},dx^{(1)}),\\
\tilde{\nu}_{N-1}(dx^{(N-1)})=\int_{\mathring{W}^N}^{}\nu_N(dx^{(N)})\mathfrak{L}_{N-1}^N(x^{(N)},dx^{(N-1)}).
\end{align*}
Then, by the inductive hypothesis $\left(\mathbb{X}^{(N-1)}(t);t \ge 0\right)$ evolves according to $\mathfrak{P}^{(N-1)}(t)$, with the joint evolution of $(\mathbb{X}^{(N-1)},\mathbb{X}^{(N)})$, by (\ref{eigenGelfand}) and (\ref{consistencyGelfand}) with $n=N-1$, as in Theorem \ref{MasterDynamics}, with $X=\mathbb{X}^{(N)}$ and $Y=\mathbb{X}^{(N-1)}$ and with initial condition $\nu_N(dx^{(N)})\mathfrak{L}_{N-1}^N(x^{(N)},dx^{(N-1)})$. We thus obtain that  $\left(\mathbb{X}^{(N)}(t);t \ge 0\right)$ evolves according to $\mathfrak{P}^{(N)}(t)$ and for fixed $T$ the conditional distribution of $\mathbb{X}^{(N-1)}(T)$ given $\mathbb{X}^{(N)}(T)$ is $\mathfrak{L}_{N-1}^N\left(\mathbb{X}^{(N)}(T),dx^{(N-1)}\right)$. This, along with the inductive hypothesis on the law of $\mathbb{GT}(N-1)$ at time $T$ yields (\ref{Gibbsevolved}). The fact that $\tau_{\mathbb{GT}(N)}=\infty$ is also clear since each $\left(\mathbb{X}^{(n)}(t);t \ge 0\right)$ is governed by a Doob transformed Karlin-McGregor semigroup.
\end{proof}

Similarly, the result above holds by replacing $\nu_N(dx^{(N)})$ by an entrance law $\left(\nu_t^{(N)}(dx^{(N)})\right)_{t\ge 0}$ for $\mathfrak{P}^{(N)}(t)$, in which case $\left(\nu_N\mathfrak{P}^{(N)}_T\right)(dx^{(N)})$ gets replaced by $\nu_T^{(N)}(dx^{(N)})$.

The consistency relations (\ref{eigenGelfand}) and (\ref{consistencyGelfand}) and the implications for which choices of $L_1,\cdots, L_N$ to make will not be studied here. These questions are worth further investigation and will be addressed in future work.

\subsection{Brownian motions in full space}\label{MultilevelDBM}
The process considered here was first constructed by Warren in \cite{Warren}. Suppose in our setup of the previous section we take as the $L$-diffusion a standard Brownian motion with generator $\frac{1}{2}\frac{d^2}{dx^2}$, speed measure with density $m(x)=2$ and scale function $s(x)=x$. Then, its conjugate diffusion with generator $\hat{L}$ from the results of the previous section is again a standard Brownian motion, so that in particular $P_t^n=\hat{P}_t^n$. Recall that the Vandermonde determinant $h_n(x)=\prod_{1 \le i<j \le n}^{}(x_j-x_i)$ is a positive harmonic function for $P_t^n$ (see for example \cite{Warren} or by iteration from the results here). Moreover, the $h$-transformed semigroup $P_t^{n,h_n}$ is exactly the semigroup of $n$ particle Dyson Brownian motion.

\begin{prop}\label{DysonProposition}
Let $x \in \mathring{W}^{n+1}(\mathbb{R})$ and consider a process $(X,Y)\in W^{n,n+1}(\mathbb{R})$ started from the distribution  $\left(\delta_x,\frac{n!h_n(y)}{h_{n+1}(x)}\mathbf{1}(y\prec x)dy\right)$ with the $Y$ particles evolving as $n$ particle Dyson Brownian motion and the $X$ particles as $n+1$ standard Brownian motions reflected off the $Y$ particles. Then, the $X$ particles are distributed as $n+1$ Dyson Brownian motion started from $x$.
\end{prop}
\begin{proof}
We apply Theorem (\ref{MasterDynamics}) with the $L$-diffusion being a standard Brownian motion. Observe that, $(\mathbf{R})$, $(\mathbf{BC+})$ and $(\mathbf{YW})$ are easily seen to be satisfied. Finally, as recalled above the Vandermonde determinant $h_n(x)=\prod_{1 \le i<j \le n}^{}(x_j-x_i)$ is a positive harmonic function for $n$ independent Brownian motions killed when they intersect and the semigroup $P_t^{n,h_n}$ is the one associated to $n$ particle Dyson Brownian motion.
\end{proof}

In fact, we can start the process from the boundary of $W^{n,n+1}(\mathbb{R})$ via an entrance law as described in the previous section. To be more concrete, an entrance law for $P_t^{n+1,h_{n+1}}$ describing the process starting from the origin, which can be obtained via a limiting procedure detailed in the Appendix is the following:
\begin{align*}
\mu_t^{n+1,h_{n+1}}(dx)=C_{n+1}t^{-(n+1)^2/2}\exp\big(-\frac{1}{2t}\sum_{i=1}^{n+1}x_i^2\big)h_{n+1}^2(x)dx.
\end{align*}
Thus, from the previous section's results 
\begin{align*}
\nu_t^{n,n+1,h_{n+1}}(dx,dy)=\mu_t^{n+1,h_{n+1}}(dx) \frac{n!h_n(y)}{h_{n+1}(x)}\mathbf{1}(y\prec x)dy,
\end{align*} 
forms an entrance law for the semigroup associated to the two-level process in Proposition \ref{DysonProposition}. Hence, we obtain the following:
\begin{prop}\label{superpositionref1}
 Consider a Markovian process $(X,Y)\in W^{n,n+1}(\mathbb{R})$ initialized according to the entrance law $\nu_t^{n,n+1,h_{n+1}}(dx,dy)$ with the $Y$ particles evolving as $n$ particle Dyson Brownian motion and the $X$ particles as $n+1$ standard Brownian motions reflected off the $Y$ particles. Then, the $X$ particles are distributed as $n+1$ Dyson Brownian motion started from the origin.
\end{prop}

It can be seen that we are in the setting of Proposition \ref{multilevelproposition} with the $L_k\equiv L$-diffusion a standard Brownian motion and the functions $g_k, G_k$ being up to a multiplicative constant equal to the Vandermonde determinant $\prod_{1 \le i<j \le k}^{}(x_j-x_i)$. Thus, we can concatenate these two-level processes to build a process $\left(\mathbb{X}^{n}(t);t \ge 0\right)=(X_i^{(k)}(t);t \ge 0, 1 \le i \le k \le n)$ taking values in $\mathbb{GT}(n)$ recovering Proposition 6 of \cite{Warren}. Being more concrete, the dynamics of $\mathbb{X}^{n}(t)$ are as follows: level $k$ of this process consists of $k$ independent standard Brownian motions reflected off the paths of level $k-1$. Then, from Proposition \ref{multilevelproposition} we get:
\begin{prop}
If $\mathbb{X}^{n}$ is started from the origin then the $k^{th}$ level process $X^{(k)}$ is distributed as $k$ particle Dyson Brownian motion started from the origin.
\end{prop}

\paragraph{Connection to Hermitian Brownian motion} We now point out the well known connection to the minor process of a Hermitian valued Brownian motion.It is a well known fact that the eigenvalues of minors of Hermitian matrices interlace. In particular, for any $n\times n$ Hermitian valued diffusion the eigenvalues of the $k\times k$ minor $\left(\lambda^{(k)}(t);t\ge 0\right)$ and of the $(k-1)\times (k-1)$ minor $\left(\lambda^{(k-1)}(t);t\ge0\right)$ interlace: $\left(\lambda^{(k)}_1(t)\le \lambda_2^{(k-1)}(t)\le \cdots \le \lambda_k^{(k)}(t);t\ge 0\right)$. Now, let $\left(H(t);t\ge 0\right)$ be an $n\times n$ Hermitian valued Brownian motion. Then $\left(\lambda^{(k)}(t);t\ge0\right)$ evolves as $k$ particle Dyson Brownian motion. Also for any \textit{fixed} time $T$ the vector $(\lambda^{(1)}(T),\cdots,\lambda^{(n)}(T))$ has the same distribution as $\mathbb{X}(T)$, namely it is uniform on the space of $\mathbb{GT}(n)$ with bottom level $\lambda^{(n)}(T)$. However the evolution of these processes is different, in fact the interaction between two levels of the minor process $\left(\lambda^{(k-1)}(t),\lambda^{(k)}(t);t\ge0\right)$ is quite complicated involving long range interactions and not the local reflection as in our case as shown in \cite{Adler}. In fact, the evolution of $\left(\lambda^{(k-1)}(t),\lambda^{(k)}(t),\lambda^{(k+1)}(t);t\ge 0\right)$ is not even Markovian at least for some initial conditions (again see \cite{Adler}).

\subsection{Brownian motions in half line and BES(3)}
The process we will consider here, taking values in a symplectic Gelfand-Tsetlin pattern, was first constructed by Cerenzia in \cite{Cerenzia} as the diffusive scaling limit of the symplectic Plancherel growth model. It is built from reflecting and killed Brownian motions in the half line. We begin in the simplest possible setting:

\begin{prop}\label{Bes3}
 Consider a process $(X,Y)\in W^{1,1}([0,\infty))$ started according to the distribution $(\delta_x,\mathbf{1}_{[0,x]}dy)$ for $x>0$ with the $Y$ particle evolving as a reflecting Brownian motion in $[0,\infty)$ and the $X$ particle as a Brownian motion in $(0,\infty)$ reflected upwards off the $Y$ particle. Then, the $X$ particle is distributed as a $BES(3)$ process (Bessel process of dimension 3) started from $x$.
\end{prop}

\begin{proof}
Take as the $L$-diffusion a Brownian motion absorbed when it reaches $0$ and let $P_t^1$ be the semigroup of Brownian motion \textit{killed} (not absorbed) at 0. Then, its dual diffusion $\hat{L}$ is a reflecting Brownian motion in the positive half line and let $\hat{P}_t^1$ be the semigroup it gives rise to. Observe that, $(\mathbf{R})$, $(\mathbf{BC+})$ and $(\mathbf{YW})$ are easily seen to be satisfied.  Letting, $\hat{h}_{1,1}(x)=1$ which is clearly a positive harmonic function for $\hat{L}$, we get that $h_{1,1}(x)=\int_{0}^{x}1dx=x$. Now, note that $P_t^{1,h_{1,1}}$ is exactly the semigroup of a $BES(3)$ process. As is well known, a Bessel process of dimension 3 is a Brownian motion conditioned to stay in $(0,\infty)$ by an $h$-transform with the function $x$. Then, from the analogue of Theorem \ref{MasterDynamics} in $W^{n,n}$ we obtain the statement.
\end{proof}

Now we move to the next stage of $2$ particles evolving as reflecting Brownian motions being reflected off a $BES(3)$ process. 
\begin{prop}
 Consider a process $(X,Y)\in W^{1,2}([0,\infty))$ started according to the following distribution $\left(\delta_{(x_1,x_2)},\frac{2y}{x_2^2-x_1^2}1_{[x_1,x_2]}dy\right)$ for $x_1<x_2$ with the $Y$ particle evolving as a BES(3) process and the $X$ particles as reflecting Brownian motions in $[0,\infty)$ reflected off the $Y$ particles. Then, the $X$ particles are distributed as two non-intersecting reflecting Brownian motions started from $(x_1,x_2)$.
\end{prop}

\begin{proof}
We apply Theorem \ref{MasterDynamics}. We take as the $L$-diffusion a reflecting Brownian motion. Write $P_t^{2}$ for the Karlin-McGregor semigroup associated to 2 reflecting Brownian motions killed when they intersect. Note that, $(\mathbf{R})$, $(\mathbf{BC+})$ and $(\mathbf{YW})$ are clearly satisfied. Observe that with $\hat{h}_{1,2}(x)=x$, which is a positive harmonic function for a Brownian motion killed at 0, we have:
\begin{align*}
h_{1,2}(x_1,x_2)=\int_{x_1}^{x_2}xdx=\frac{1}{2}(x_2^2-x_1^2).
\end{align*}
Finally note that, $P_t^{2,h_{1,2}}$ is exactly the semigroup of $2$ non-intersecting reflecting Brownian motions in $[0,\infty)$.
\end{proof}

These relations can be iterated to $n$ and $n$ and also $n$ and $n+1$ particles. Define the functions:
\begin{align*}
\hat{h}_{n,n}(x)&=\prod_{1 \le i < j \le n}^{}(x_j^2-x_i^2),\\
\hat{h}_{n,n+1}(x)&=\prod_{1 \le i < j \le n}^{}(x_j^2-x_i^2)\prod_{i=1}^{n}x_i.
\end{align*}
Also, consider the positive kernels $\Lambda_{n_1,n_2}$, defined in (\ref{PreIntertwiningKernel}), with $\hat{m}\equiv 2$. Then, an easy calculation (after writing these functions as determinants) gives that up to a constant $h_{n,n}=\Lambda_{n,n}\hat{h}_{n,n}$ is equal to $\hat{h}_{n,n+1}$ and $h_{n,n+1}=\Lambda_{n,n+1}\hat{h}_{n,n+1}$ is equal to $\hat{h}_{n+1,n+1}$. Finally, let $\Lambda^{\hat{h}_{n,n}}_{n,n}$ and $\Lambda^{\hat{h}_{n,n+1}}_{n,n+1}$ denote the corresponding normalized Markov kernels.

\begin{prop}
 Consider a process $(X,Y)\in W^{n,n}([0,\infty))$ started according to the distribution $(\delta_x,\Lambda^{\hat{h}_{n,n}}_{n,n}(x,\cdot))$ for $x\in\mathring{W}^{n}([0,\infty))$ with the $Y$ particles evolving as $n$ reflecting Brownian motions conditioned not to intersect in $[0,\infty)$ and the $X$ particles as $n$ Brownian motion in $(0,\infty)$ reflected off the $Y$ particles. Then, the $X$ particles are distributed as $n$ $BES(3)$ processes conditioned never to intersect started from $x$.
\end{prop}

\begin{proof}
We take as the $L$-diffusion a Brownian motion absorbed at $0$. Then, the $\hat{L}$-diffusion is a reflecting Brownian motion. As before, $(\mathbf{R})$, $(\mathbf{BC+})$ and $(\mathbf{YW})$ are clearly satisfied. Note that, $\hat{h}_{n,n}$ is a harmonic function for $n$ reflecting Brownian motions killed when they intersect. Moreover, note that $P_t^{n,h_{n,n}}$ is exactly the semigroup of $n$ non-intersecting $BES(3)$ processes (note that the $n$ particle Karlin-McGregor semigroup $P_t^n$ is that of $n$ killed at zero Brownian motions). The statement follows from the analogue of Theorem \ref{MasterDynamics} in $W^{n,n}$.
\end{proof}

\begin{prop}
 Consider a process $(X,Y)\in W^{n,n+1}([0,\infty))$ started according to the following distribution $\left(\delta_x,\Lambda^{\hat{h}_{n,n+1}}_{n,n+1}(x,\cdot)\right)$ for $x \in\mathring{W}^{n+1}([0,\infty))$ with the $Y$ particles evolving as $n$ BES(3) processes conditioned not to intersect and the $X$ particles as $n+1$ reflecting Brownian motions in $[0,\infty)$ reflected off the $Y$ particles. Then, the $X$ particles are distributed as $n+1$ non-intersecting reflecting Brownian motions started from $x$.
\end{prop}

\begin{proof}
We take as the $L$-diffusion a reflecting Brownian motion. Then, the $\hat{L}$-diffusion is a Brownian motion absorbed at 0. As before, the assumptions $(\mathbf{R})$, $(\mathbf{BC+})$ and $(\mathbf{YW})$ are clearly satisfied. Note that, $\hat{h}_{n,n+1}$ is harmonic for the corresponding Karlin-McGregor semigroup $\hat{P}_t^n$, associated with $n$ Brownian motions killed at zero and when they intersect. Moreover, note that the semigroup $\hat{P}_t^{n, \hat{h}_{n,n+1}}$, namely the semigroup $\hat{P}_t^{n}$ $h$-transformed by $\hat{h}_{n,n+1}$, gives the semigroup of the process $Y$. Finally, observe that $P_t^{n+1,h_{n,n+1}}$ is exactly the semigroup of $n+1$ non-intersecting reflecting Brownian motions. The statement follows from Theorem \ref{MasterDynamics}.
\end{proof}

\begin{rmk}
The semigroups considered above are also the semigroups of $n$ Brownian motions conditioned to stay in a Weyl Chamber of type $B$ and type $D$ (after we disregard the sign of the last coordinate) respectively (see for example \cite{JonesO'Connell} where such a study was undertaken). 
\end{rmk}

We can in fact start these processes from the origin, by using the following explicit entrance law (see for example \cite{Cerenzia} or the Appendix for the general recipe) for $P_t^{n,h_{n,n}}$ and $P_t^{n,h_{n-1,n}}$ issued from zero,
\begin{align*}
\mu_t^{n,h_{n,n}}(dx)&=C_{n,n}'t^{-n(n+\frac{1}{2})}\exp\bigg(-\frac{1}{2t}\sum_{i=1}^{n}x_i^2\bigg)h^2_{n,n}(x)dx,\\
\mu_t^{n,h_{n-1,n}}(dx)&=C'_{n-1,n}t^{-n(n-\frac{1}{2})}\exp\bigg(-\frac{1}{2t}\sum_{i=1}^{n}x_i^2\bigg)h_{n-1,n}^2(x)dx.
\end{align*}

Concatenating these two-level processes, we construct a process $\left(\mathbb{X}_s^{(n)}(t);t \ge 0\right)=(X^{(1)}(t)\prec \hat{X}^{(1)}(t)\prec \cdots \prec X^{(n)}(t)\prec \hat{X}^{(n)}(t);t\ge 0)$ in $\mathbb{GT}_{\textbf{s}}(n)$ with dynamics as follows: Firstly, $X_1^{(1)}$ is a Brownian motion reflecting at the origin. Then, for each $k$, the $k$ particles corresponding to $\hat{X}^{(k)}$ perform independent Brownian motions reflecting off the $X^{(k)}$ particles to maintain interlacing. Finally, for $k\ge 2$ the $k$ particles corresponding to $X^{(k)}$ reflect off $\hat{X}^{(k-1)}$ and also in the case of $X^{(k)}_1$ reflecting at the origin.

 Then, the symplectic analogue of Proposition \ref{multilevelproposition} (which is again proven in the same way by consistently patching together two-level processes) implies the following, recovering the results of Section 2.3 of \cite{Cerenzia}:

\begin{prop}\label{CerenziaGelfand}
If $\mathbb{X}_s^{n}$ is started from the origin then the projections onto $X^{(k)}$ and $\hat{X}^{(k)}$ are distributed as $k$ non-intersecting reflecting Brownian motions and $k$ non-intersecting $BES(3)$ processes respectively started from the origin.
\end{prop}

\subsection{Brownian motions in an interval}
Let $I=[0,\pi]$ for concreteness and let the $L$-diffusion be a reflecting Brownian motion in $I$. Then its dual, the $\hat{L}$-diffusion is a Brownian motion absorbed at $0$ or $\pi$. It will be shown in Corollary \ref{minimal}, that the minimal positive eigenfunction, is given up to a (signed) constant factor by,
\begin{align}\label{killedinterval}
\hat{h}_n(x)=\det(\sin(kx_j))_{k,j=1}^n.
\end{align}
This is the eigenfunction that corresponds to conditioning these Brownian motions to stay in the interval $(0,\pi)$ and not intersect forever. Also, observe that up to a constant factor $\hat{h}_n$ is given by (see the notes \cite{Conrey}, \cite{Meckes} and Remark \ref{RemarkClassicalGroups} below for the connection to classical compact groups),
\begin{align*}
\prod_{i=1}^{n}\sin(x_i)\prod_{1 \le i < j \le n}^{}\big(\cos(x_i)-\cos(x_j)\big).
\end{align*}

Now, via the iterative procedure of producing eigenfunctions, namely by taking $\Lambda_{n,n+1}\hat{h}_n$, where $\Lambda_{n,n+1}$ is defined in (\ref{PreIntertwiningKernel}), we obtain that up to a (signed) constant factor,
\begin{align}\label{reflecting intervale}
h_{n+1}(x)=\det(\cos((k-1)x_j))_{k,j=1}^{n+1},
\end{align}
is a strictly positive eigenfunction for $P_t^{n+1}$. In fact, it is the minimal positive eigenfunction (again this follows from Corollary \ref{minimal}) of $P_t^{n+1}$ and it corresponds to conditioning these reflected Brownian motions in the interval to not intersect. This is also (see \cite{Conrey},\cite{Meckes} and Remark \ref{RemarkClassicalGroups}) given up to a constant factor by,
\begin{align*}
\prod_{1 \le i < j \le n+1}^{}\big(\cos(x_i)-\cos(x_j)\big).
\end{align*}
Define the Markov kernel:
\begin{align*}
(\Lambda^{\hat{h}_n}_{n,n+1}f)(x)=\frac{n!}{h_{n+1}(x)}\int_{W^{n,n+1}(x)}^{}\hat{h}_n(y)f(x,y)dy.
\end{align*}
Then we have the following result:
\begin{prop}
Let $x \in \mathring{W}^{n+1}([0,\pi])$. Consider a process $(X,Y)\in W^{n,n+1}([0,\pi])$ started at $\left(\delta_x,\Lambda^{\hat{h}_n}_{n,n+1}(x,\cdot)\right)$ with the $Y$ particles evolving as n Brownian motions 
conditioned to stay in $(0,\pi)$ and conditioned to not intersect and the $X$ particles as $n+1$ reflecting Brownian motions in $[0,\pi]$ reflected off the $Y$ particles. Then the $X$ particles are distributed as $n+1$ non-intersecting Brownian motions reflected at the boundaries of $[0,\pi]$ started from $x$.
\end{prop}
\begin{proof}
Take as the $L$-diffusion a reflecting Brownian motion in $[0,\pi]$. The $\hat{L}$-diffusion is a Brownian motion absorbed at $0$ or $\pi$. Observe that, the assumptions $(\mathbf{R})$, $(\mathbf{BC+})$ and $(\mathbf{YW})$ are satisfied. Moreover, as noted above $\hat{h}_n$ is the ground state for $n$ Brownian motions killed when they hit $0$ or $\pi$ or when they intersect. The statement of the proposition then follows from Theorem \ref{MasterDynamics}.
\end{proof}

\begin{rmk}
The dual relation, in the following sense is also true: If we reflect $n$ Brownian motions between $n+1$ reflecting Brownian motions in $[0,\pi]$ conditioned not to intersect then we obtain $n$ Brownian motions conditioned to stay in $(0,\pi)$ and conditioned not to intersect. This is obtained by noting that up to a constant factor $\hat{h}_n$ defined in (\ref{killedinterval}) is given by $\Lambda_{n+1,n}h_{n+1}$, with $h_{n+1}$ as in (\ref{reflecting intervale}).
\end{rmk}

\begin{rmk}\label{RemarkClassicalGroups}
The processes studied above are related to the eigenvalue evolutions of Brownian motions on $SO(2(n+1))$ (reflecting Brownian motions in $[0,\pi]$) and $USp(2n)$ (conditioned Brownian motions in $[0,\pi]$) respectively (see e.g. \cite{PauwelsRogers} for skew product decompositions of Brownian motions on manifolds of matrices).
\end{rmk}

\begin{rmk}
It is also possible to build the following interlacing processes with equal number of particles. Consider as the $Y$ process $n$ Brownian motions in $[0,\pi)$ reflecting at $0$ and conditioned to stay away from $\pi$ and not to intersect. In our framework $\hat{L}=\frac{1}{2}\frac{d^2}{dx^2}$ with Neumann boundary condition at $0$ and Dirichlet at $\pi$. Then the minimal eigenfunction corresponding to this conditioning is given up to a sign by,
\begin{align*}
\det\bigg(\cos\bigg(\big(k-\frac{1}{2}\big)y_j\bigg)\bigg)_{k,j=1}^n.
\end{align*}
Now let $X$ be $n$ Brownian motions in $(0,\pi]$ reflecting at $\pi$ and reflected off the $Y$ particles. Then the projection onto the $X$ process (assuming the two levels $(X,Y)$ are started appropriately) evolves as $n$ Brownian motions in $(0,\pi]$ reflecting at $\pi$ and conditioned to stay away from $0$ and not to intersect. These processes are related to the eigenvalues of Brownian motions on $SO(2n+1)$ and $SO^-(2n+1)$ respectively.
\end{rmk}

\subsection{Brownian motions with drifts}\label{DriftingBrownianSection}
The processes considered here were first introduced by Ferrari and Frings in \cite{FerrariFrings} (there only the \textit{fixed time} picture was studied, namely no statement was made about the distribution of the projections on single levels as processes). They form a generalization of the process studied in the first subsection.

\subsubsection{Hermitian Brownian with drifts}
We begin by a brief study of the matrix valued process first. Let $\left(Y_t;t\ge 0\right)=\left(B_t;t \ge 0\right)$ be an $n\times n$ Hermitian Brownian motion. We seek to add a matrix of \textit{drifts} and study the resulting eigenvalue process. For simplicity let $M$ be a diagonal $n\times n$ Hermitian matrix with distinct ordered eigenvalues $\mu_1<\cdots<\mu_n$ and consider the Hermitian valued process $\left(Y_t^M;t\ge 0\right)=\left(B_t+tM;t\ge0\right)$. 

Then a computation that starts by applying Girsanov's theorem, using unitary invariance of Hermitian Brownian motion, integrating over $\mathbb{U}(n)$, the group of $n\times n $ unitary matrices, and then computing that integral using the classical Harish Chandra-Itzykson-Zuber (HCIZ) formula gives that the eigenvalues $(\lambda_1^M(t),\cdots,\lambda_n^M(t);t\ge0)$ of $\left(Y_t^M;t \ge0\right)$ form  a diffusion process with explicit transition density given by,
\begin{align*}
s_t^{n,M}(\lambda,\lambda')=\exp\big(-\frac{1}{2}\sum_{i=1}^{n}\mu_i^2t\big)\frac{\det\big(\exp(\mu_j\lambda'_i))\big)_{i,j=1}^n}{\det\big(\exp(\mu_j\lambda_i))\big)_{i,j=1}^n}\det\big(\phi_t(\lambda_i,\lambda_j')\big)_{i,j=1}^n,
\end{align*}
where $\phi_t$ is the standard heat kernel. For a proof of this fact, which uses the theory of Markov functions, see for example \cite{Chin}.

Observe that, $s_t^{n,M}$ is exactly the transition density of $n$ Brownian motions with drifts $\mu_1<\cdots<\mu_n$ conditioned to never intersect as studied in \cite{BBO}. More generally, if we look at the $k\times k$ minor of $\left(Y_t^M;t \ge 0\right)$ then its eigenvalues evolve as $k$ Brownian motions with drifts $\mu_1<\cdots<\mu_k$ conditioned to never intersect.

\begin{rmk}
These processes also appear in the recent work of Ipsen and Schomerus \cite{IpsenSchomerus} as the finite time Lyapunov exponents of "Isotropic Brownian motions".
\end{rmk}

Now, write $\mu^{(k)}$ for $(\mu_1,\cdots,\mu_k)$ and $P_t^{n,\mu^{(n)}}$ for the semigroup that arises from $s_t^{n,M}$. Then, $u_t^{n,\mu^{(n)}}(d\lambda)$ defined by,
\begin{align*}
u_t^{n,\mu^{(n)}}(d\lambda)=const_{n,t}\det(e^{-(\lambda_i-t\mu_j)^2/2t})^n_{i,j=1} \frac{\prod_{1 \le i < j \le n}^{}(\lambda_j-\lambda_i)}{\prod_{1 \le i < j \le n}^{}(\mu^{(n)}_j-\mu^{(n)}_i)}d\lambda,
\end{align*}
forms an entrance law for $P_t^{n,\mu^{(n)}}$ starting from the origin (see for example \cite{FerrariFrings} or the Appendix).

\subsubsection{Interlacing construction with drifting Brownian motions with reflection}

Now moving on to Warren's process with drifts (as referred to in \cite{FerrariFrings}). We seek to build $n+1$ Brownian motions with drifts $\mu_1<\cdots<\mu_{n+1}$ conditioned to never intersect by reflecting off $n$ Brownian motions with drifts $\mu_{1}<\cdots<\mu_{n}$ conditioned to never intersect $n+1$ independent Brownian motions each with drift $\mu_{n+1}$. We  prove the following: 
\begin{prop}
 Consider a Markov process $(X,Y)\in W^{n,n+1}(\mathbb{R})$ started from the origin with the $Y$ particles evolving as $n$ Brownian motions with drifts $\mu_1<\cdots<\mu_n$ conditioned to never intersect and the $X$ particles as $n+1$ Brownian motions all with drift $\mu_{n+1}$ reflected off the $Y$ particles. Then, the $X$ particles are distributed as $n+1$ Brownian motions with drifts $\mu_1<\cdots<\mu_{n+1}$ conditioned to never intersect started from the origin.
\end{prop}

\begin{proof}
Let the $L$-diffusion be a Brownian motion with drift $\mu_{n+1}$, namely with generator $L=\frac{1}{2}\frac{d^2}{dx^2}+\mu_{n+1}\frac{d}{dx}$. Then, its dual diffusion $\hat{L}=\frac{1}{2}\frac{d^2}{dx^2}-\mu_{n+1}\frac{d}{dx}$ has speed measure $\hat{m}(x)=2e^{-2\mu_{n+1}x}$. Note that, the assumptions $(\mathbf{R})$, $(\mathbf{BC+})$ and $(\mathbf{YW})$ are easily seen to be satisfied. Let $P_t^{n+1,\mu_{n+1}}$ and $\hat{P}_t^{n,\mu_{n+1}}$ denote the corresponding Karlin-McGregor semigroups. Consider the (not yet normalized) positive kernel $\Lambda^{\mu_{n+1}}_{n,n+1}$ given by,
\begin{align*}
(\Lambda^{\mu_{n+1}}_{n,n+1}f) (x)=\int_{W^{n,n+1}(x)}^{}f(x,y)\prod_{i=1}^{n}2e^{-2\mu_{n+1}y_i}dy_i.
\end{align*}
and define the function
\begin{align*}
\hat{h}_n^{\mu_{n+1},\mu^{(n)}}(y)=\prod_{i=1}^{n}e^{\mu_{n+1}y_i}\det(e^{\mu_iy_j} )_{ i,j=1}^{ n}.
\end{align*}
Note that, $\hat{h}_n^{\mu_{n+1},\mu^{(n)}}$ is a strictly positive eigenfunction for $\hat{P}_t^{n,\mu_{n+1}}$. Moreover, the $h$-transform of $\hat{P}_t^{n,\mu_{n+1}}$ with $\hat{h}_n^{\mu_{n+1},\mu^{(n)}}$ is exactly the semigroup $P_t^{n,\mu^{(n)}}$ of $n$ Brownian motions with drifts $(\mu_1,\cdots,\mu_n)$ conditioned to never intersect. By integrating the determinant we get,
\begin{align*}
(\Lambda^{\mu_{n+1}}_{n,n+1}\hat{h}_n^{\mu_{n+1},\mu^{(n)}}) (x)=\frac{2^n}{\prod_{i=1}^{n}(\mu_{n+1}-\mu_i)}\det(e^{(\mu_{i}-\mu_{n+1})x_j})_{i,j=1}^{n+1},
\end{align*}
and note that the $h$-transform of $P_t^{n+1,\mu_{n+1}}$ by $\Lambda^{\mu_{n+1}}_{n,n+1}\hat{h}_n^{\mu_{n+1},\mu^{(n)}}$ is $P_t^{n+1,\mu^{(n+1)}}$. 
Finally, defining the entrance law for the two-level process started from the origin by $\nu_t^{n,n+1,\mu_{n+1},\mu^{(n)}}=u_t^{n+1,\mu^{(n+1)}}\Lambda^{\mu_{n+1},\mu^{(n)}}_{n,n+1}$,
we obtain the statement of the proposition from Theorem \ref{MasterDynamics} (see also discussion after Corollary \ref{EntranceLawCorollary}).
\end{proof}

\begin{rmk}
A 'positive temperature' version of the proposition above appears as Proposition 9.1 in \cite{Toda}.
\end{rmk}

We can then iteratively apply the result above to concatenate two-level processes and build a process:
\begin{align*}
\left(\mathbb{X}_{(\mu_1,\cdots,\mu_n)}(t);t\ge 0\right)=\left(X^{(1)}_{\mu_1}(t)\prec X^{(2)}_{\mu_2}(t)\prec\cdots \prec X^{(n)}_{\mu_n}(t);t\ge0\right),
\end{align*}
in $\mathbb{GT}(n)$ as in Proposition \ref{multilevelproposition} whose joint dynamics are given as follows (this was also described in \cite{FerrariFrings}): Level $k$ consists of $k$ copies of independent Brownian motions all with drifts $\mu_k$ reflected off the paths of level $k-1$. Then, from Proposition \ref{multilevelproposition} one obtains:
\begin{prop}
Assume $\mu_1 < \mu_2< \cdots < \mu_n$. Consider the process $\left(\mathbb{X}_{(\mu_1,\cdots,\mu_n)}(t);t\ge 0\right)$ defined above started from the origin. Then, the projection on $X^{(k)}_{\mu_k}$ is distributed as $k$ Brownian motions with drifts $\mu_1<\cdots<\mu_k$ conditioned to never intersect, issueing from the origin.
\end{prop}

\begin{rmk}
Note that, the multilevel process whose construction is described above via the hard reflection dynamics and the minors of the Hermitian valued process $\left(Y_t^M;t \ge 0\right)$ coincide on each fixed level $k$ (as single level processes, this is what we have proven here) and also at fixed times (this is already part of the results of \cite{FerrariFrings}). However, they do not have the same law as processes. Finally, for the fixed time correlation kernel of this Gelfand-Tsetlin valued process see Theorem 1 of \cite{FerrariFrings}.
\end{rmk}

\subsection{Geometric Brownian motions and quantum Calogero-Sutherland}

A geometric Brownian motion of unit diffusivity and drift parameter $\alpha$ is given by the $SDE$,
\begin{align*}
ds(t)=s(t)dW(t)+\alpha s(t)dt ,
\end{align*}
which can be solved explicitly to give that,
\begin{align*}
s(t)=s(0)\exp\left(W(t)+\left(\alpha-\frac{1}{2}\right)t\right).
\end{align*}
We will assume that $s(0)>0$, so that the process lives in $(0,\infty)$. Its generator is given by,
\begin{align*}
L^{\alpha}=\frac{1}{2}x^2\frac{d^2}{dx^2}+\alpha x \frac{d}{dx} ,
\end{align*}
with both $0$ and $\infty$ being natural boundaries. With $h_n(x)=\prod_{1\le i< j \le n}^{}(x_j-x_i)$ denoting the Vandermonde determinant it can be easily verified (although it also follows by recursively applying the results below) that $h_n$ is a positive eigenfunction of $n$ independent geometric Brownian motions, namely that with,
\begin{align*}
L_n^{\alpha}=\sum_{i=1}^{n}\frac{1}{2}x_i^2\partial_{x_i}^2+\alpha\sum_{i=1}^{n}x_i\partial_{x_i} ,
\end{align*}
we have,
\begin{align*}
L_n^{\alpha}h_n=\frac{n(n-1)}{2}\left(\frac{n-2}{3}+\alpha\right)h_n=c_{n,\alpha}h_n.
\end{align*}
The quantum Calogero-Sutherland Hamiltonian $\mathcal{H}^{\theta}_{CS}$ (see \cite{Calogero}, \cite{Sutherland}) is given by,
\begin{align*}
\mathcal{H}^{\theta}_{CS}=\frac{1}{2}\sum_{i=1}^{n}\left(x_i\partial_{x_i}\right)^2+\theta \sum_{i=1}^{n}\sum_{j\ne i}^{} \frac{x_i^2}{x_i-x_j}\partial_{x_i} .
\end{align*}
Its relation to geometric Brownian motions lies in the following simple observation. For $\theta=1$ this quantum Hamiltonian coincides with the infinitesimal generator of $n$ independent geometric Brownian motions with drift parameter $\frac{1}{2}$ $h$-transformed by the Vandermonde determinant namely,
\begin{align*}
\mathcal{H}^{1}_{CS}=h_n^{-1}\circ L^{\frac{1}{2}}_{n} \circ h_n-c_{n,\frac{1}{2}}.
\end{align*}
We now show how one can construct a $\mathbb{GT}(n)$ valued process so that the $k^{th}$ level consists of $k$ geometric Brownian motions with drift parameter $n-k+\frac{1}{2}$ $h$-transformed by the Vandermonde determinant. The key ingredient is the following: 

\begin{prop}\label{GeometricProp}
Consider a process $(X,Y)\in W^{n,n+1}((0,\infty))$ started according to the following distribution $(\delta_x,\frac{n!h_n(y)}{h_{n+1}(x)}\mathbf{1}(y\prec x)dy)$ for $x\in \mathring{W}^{n+1}((0,\infty))$ with the $Y$ particles evolving as $n$ non-intersecting geometric Brownian motions with drift parameter $\alpha+1$ conditioned to not intersect via an $h$-transform by $h_n$ and the $X$ particles evolving as $n+1$ geometric Brownian motions with drift parameter $\alpha$ being reflected off the $Y$ particles. Then, the $X$ particles are distributed as $n+1$ non-intersecting geometric Brownian motions with drift parameter $\alpha$ conditioned to not intersect via an $h$-transform by $h_{n+1}$, started form $x\in \mathring{W}^{n+1}((0,\infty))$.
\end{prop}

\begin{proof}
Taking as the $L$-diffusion $L^{\alpha}$, and note that its speed measure is given by $m^{\alpha}(x)=2x^{2\alpha-2}$, the conjugate diffusion is $\widehat{L^{\alpha}}=L^{1-\alpha}$. Observe that, the assumptions $(\mathbf{R})$, $(\mathbf{BC+})$ and $(\mathbf{YW})$ are clearly satisfied.

First, note that an easy calculation gives that the $h$-transform of $\widehat{L^{\alpha}}$ by $\widehat{m^{\alpha}}^{-1}$ is an $L^{\alpha+1}$-diffusion, namely a geometric Brownian motion with drift parameter $\alpha+1$. Hence, an $h$-transform of $n$ $\widehat{L^{\alpha}}$-diffusions by the eigenfunction $\prod_{i=1}^{n}\widehat{m^{\alpha}}^{-1}(y_i)h_n(y)$ gives $n$ non-intersecting geometric Brownian motions with drift parameter $\alpha+1$ conditioned to not intersect via an $h$-transform by $h_n$. The statement of the proposition is then obtained from an application of Theorem \ref{MasterDynamics}.
\end{proof}

\begin{rmk}
Observe that, under an application of the exponential map the results of Section \ref{DriftingBrownianSection}, give a generalization of Proposition \ref{GeometricProp} above.
\end{rmk}

Using the proposition above it is straightforward, and we will not elaborate on, how to iterate to build the $\mathbb{GT}(n)$ valued process with the correct drift parameters on each level.

\begin{rmk}
The following geometric Brownian motion,
\begin{align*}
ds(t)=\sqrt{2}s(t)dW(t)-(u+u'+v+v') s(t)dt,
\end{align*}
also arises as a continuum scaling limit after we scale space by $1/N$ and send $N$ to infinity of the bilateral birth and death chain with birth rates $(x-u)(x-u')$ and death rates $(x+v)(x+v')$ considered by Borodin and Olshanski in \cite{BorodinOlshanski}.
\end{rmk}

\subsection{Squared Bessel processes and LUE matrix diffusions}
In this subsection we will first construct a process taking values in $\mathbb{GT}$ being the analogue of the Brownian motion model for squared Bessel processes and having close connections to the LUE matrix valued diffusion. We also build a process in $\mathbb{GT}_\textbf{s}$ generalizing the construction of Cerenzia (after a "squaring" transformation of the state space) for all dimensions $d\ge2$. We begin with a definition:

\begin{defn}
The squared Bessel process of dimension $d$, abbreviated from now on as $BESQ(d)$ process, is the one dimensional diffusion with generator in $(0,\infty)$,
\begin{align*}
L^{(d)}=2x\frac{d^2}{dx^2}+d\frac{d}{dx}.
\end{align*}
 The origin is an entrance boundary for $d\ge 2$, a regular boundary point for $0<d<2$ and an exit one for $d\le 0$. Define the index $\nu(d)=\frac{d}{2}-1$. The density of the speed measure of $L^{(d)}$ is $m_{\nu}(y)=c_{\nu}y^{\nu}$ and its scale function $s_{\nu}(x)=\bar{c}_{\nu}x^{-\nu}$, $\nu \ne 0$ and $s_0(x)=logx$. Then from the results of the previous section its conjugate, the $\widehat{L^{(d)}}$ diffusion, is a $BESQ(2-d)$ process with the dual boundary condition. Moreover, the following relation will be key, see \cite{SurveyBessel}: A Doob $h$-transform of a $BESQ(2-d)$ process by its scale function $x^{\nu+1}$ gives a $BESQ(d+2)$ process.
\end{defn}

 Note that, condition $(\mathbf{BC+})$ only holds for dimensions $d \in (-\infty,0] \cup [2,\infty)$; this is because for $0<d<2$, the origin is a regular boundary point and the diffusion coefficient degenerates (these values of the parameters will not be considered here). We use the following notation throughout, for $d \in (-\infty,0] \cup [2,\infty)$:  we write $P_t^{n,(d)}$ for the Karlin-McGregor semigroup of $n$ $BESQ(d)$ processes killed when they intersect or when they hit the origin, in case $d \le 0$.

We start in the simplest setting of $W^{1,1}$ and consider the situation of a single $BESQ(2-d)$ process being reflected upwards off a $BESQ(d)$ process: 
\begin{prop}\label{BESQ11}
 Let $d\ge 2$. Consider a process $(X,Y)\in W^{1,1}([0,\infty))$ started according to the distribution $(\delta_x,\frac{(\nu+1)y^{\nu}}{x^{\nu+1}}1_{[0,x]}dy)$ for $x>0$ with the $Y$ particle evolving as a $BESQ(d)$ process and the $X$ particle as a $BESQ(2-d)$ process in $(0,\infty)$ reflected off the $Y$ particle. Then, the $X$ particle is distributed as a $BESQ(d+2)$ process started from $x$.
\end{prop}

\begin{proof}
We take as the $L$-diffusion a $BESQ(2-d)$ process. Then, the $\hat{L}$-diffusion is a $BESQ(d)$ process. Note that, the assumptions $(\mathbf{R})$, $(\mathbf{BC+})$ and $(\mathbf{YW})$ are satisfied. Since $\hat{h}_{1,1}^{(d)}(x)=1$ is invariant for $BESQ(d)$, the following is invariant for $BESQ(2-d)$,
\begin{align*}
h^{(d)}_{1,1}(x)=\int_{0}^{x}c_{\nu}y^{\nu}dy=\frac{c_{\nu}}{\nu+1}x^{\nu+1}.
\end{align*}
Then, as already remarked above, see \cite{SurveyBessel}, the $h$-transformed process with semigroup $P_t^{1,(2-d),h^{(d)}_{1,1}}$ is exactly a $BESQ(d+2)$ process. The analogue of Theorem \ref{MasterDynamics} in $W^{n,n}$ gives the statement of the proposition.
\end{proof}

We expect that the restriction to $d\ge 2$ is not necessary for the result to hold (it should be true for $d>0$). In fact, Corollary \ref{Bes3}, corresponds to $d=1$, after we perform the transformation $x \mapsto \sqrt{x}$, which in particular maps $BESQ(1)$ and $BESQ(3)$ to reflecting Brownian motion and $BES(3)$ respectively. 

We now move on to an arbitrary number of particles. Define the functions,
\begin{align*}
\hat{h}^{(d)}_{n,n}(x)&=\prod_{1 \le i < j \le n}^{}(x_j-x_i)=\det\left(x_i^{j-1}\right)_{i,j=1}^n \ , \\
\hat{h}_{n,n+1}^{(d)}(x)&=\prod_{1 \le i < j \le n}^{}(x_j-x_i)\prod_{i=1}^{n}x_i^{\nu+1}=\det\left(x_i^{j+\nu}\right)_{i,j=1}^n \ .
\end{align*}
Moreover, let $\Lambda_{n-1,n}$ and $\Lambda_{n,n}$ be the following positive kernels, defined as in (\ref{PreIntertwiningKernel}), where we recall that $m_{\nu(d)}(\cdot)$ is the speed measure density with respect to Lebesgue measure of a $BESQ(d)$ process:
\begin{align*}
(\Lambda_{n-1,n}f)(x) &= \int_{W^{n-1,n}(x)}^{}\prod_{i=1}^{n-1}m_{\nu(2-d)}(y_i)f(x,y)dy,\\
(\Lambda_{n,n}f)(x) &= \int_{W^{n,n}(x)}^{}\prod_{i=1}^{n_1}m_{\nu(d)}(y_i)f(x,y)dy.
\end{align*}
 An easy calculation gives that $h^{(d)}_{n-1,n}(x)=c_{n-1,n}(\nu)(\Lambda_{n-1,n}\Pi_{n-1,n}\hat{h}_{n-1,n}^{(d)})(x)$ is equal to $\hat{h}^{(d)}_{n,n}(x)$ and $h_{n,n}^{(d)}(x)=c_{n,n}(\nu)(\Lambda_{n,n}\Pi_{n,n}\hat{h}_{n,n}^{(d)})(x)$ is equal $\hat{h}_{n,n+1}^{(d)}(x)$, where $c_{n-1,n}(\nu), c_{n,n}(\nu)$ are explicit constants whose exact values are not important in what follows. Then we have:

\begin{prop}\label{BesqProp1}
Let $d\ge 2$. Consider a process $(X,Y)\in W^{n,n+1}([0,\infty))$ started according to the following distribution $(\delta_x,\frac{n!\prod_{1 \le i < j \le n}^{}(y_j-y_i)}{\prod_{1 \le i < j \le n+1}^{}(x_j-x_i)}\mathbf{1}(y\prec x)dy)$ for $x\in \mathring{W}^{n+1}([0,\infty))$ with the $Y$ particles evolving as $n$ non-intersecting $BESQ(d+2)$ processes and the $X$ particles evolving as $n+1$ $BESQ(d)$ processes being reflected off the $Y$ particles. Then, the $X$ particles are distributed as $n+1$ non-intersecting $BESQ(d)$ processes started form $x\in \mathring{W}^{n+1}([0,\infty))$.
\end{prop}

\begin{proof}
Take as the $L$-diffusion a $BESQ(d)$ process. Then, the $\hat{L}$-diffusion is a $BESQ(2-d)$ process. Note that, the assumptions $(\mathbf{R})$, $(\mathbf{BC+})$ and $(\mathbf{YW})$ are satisfied. We use the positive harmonic function $\hat{h}_{n,n+1}^{(d)}(x)$ for the semigroup $P_t^{n,(2-d)}$ of $n$ independent $BESQ(2-d)$ processes killed when they hit 0 or when they intersect, which transforms them into $n$ non-intersecting $BESQ(d+2)$ processes.  
Finally observe that, $P_t^{n+1,(d),h^{(d)}_{n,n+1}}$ is exactly the semigroup of $n+1$ $BESQ(d)$ processes conditioned to never intersect (see e.g. \cite{O Connell}). Then, Theorem \ref{MasterDynamics} gives the statement of the proposition.
\end{proof}

\begin{prop}\label{BesqProp2}
Let $d\ge 2$. Consider a process $(X,Y)\in W^{n,n}([0,\infty))$ started according to the following distribution $\big(\delta_x,\frac{c_{n,n}(\nu)\hat{h}_{n,n}^{(d)}(y)\prod_{i=1}^{n}m_{\nu(d)}(y_i)}{h_{n,n}^{(d)}(x)}\mathbf{1}(y\prec x)dy\big)$ for $x \in \mathring{W}^{n}([0,\infty))$ with the $Y$ particles evolving as $n$ non-intersecting $BESQ(d)$ processes and the $X$ particles evolving as $n$ $BESQ(2-d)$ processes being reflected off the $Y$ particles. Then, the $X$ particles are distributed as $n$ non-intersecting $BESQ(d+2)$ processes started form $x\in \mathring{W}^{n}([0,\infty))$.
\end{prop}

\begin{proof}
Take as the $L$-diffusion a $BESQ(2-d)$ process. Then, the $\hat{L}$-diffusion is a $BESQ(d)$ process. Note that, the assumptions $(\mathbf{R})$, $(\mathbf{BC+})$ and $(\mathbf{YW})$ are satisfied. We use the positive harmonic function $\hat{h}^{(d)}_{n,n}(x)$ for the semigroup $P_t^{n,(d)}$ of $n$ independent $BESQ(d)$ processes killed when they intersect. Furthermore note that, $P_t^{n,(2-d),h_{n,n}^{(d)}}$ is the semigroup of $n$  $BESQ(d+2)$ processes conditioned to never intersect (the transformation by $h_{n,n}^{(d)}$ corresponds to transforming the $BESQ(2-d)$ processes to $BESQ(d+2)$ and then conditioning these to never intersect). Then, the analogue of Theorem \ref{MasterDynamics} in $W^{n,n}$ gives the statement.
\end{proof}

It is possible to start both of these processes from the origin via the following explicit entrance law for $n$ non-intersecting $BESQ(d)$ processes (see for example \cite{O Connell}),
\begin{align*}
\mu_t^{n,(d)}(dx)=C_{n,d}t^{-n(n+\nu)}\prod_{1 \le i < j \le n}^{}(x_j-x_i)^2\prod_{i=1}^{n}x_i^{\nu}e^{-\frac{1}{2t}x_i}dx.
\end{align*}
Defining the two entrance laws,
\begin{align*}
\nu_t^{n,n,\hat{h}_{n,n}^{(d)}}(dx,dy)&=\mu_t^{n,(d+2)}(dx) \frac{c_{n,n}(\nu)\hat{h}_{n,n}^{(d)}(y)\prod_{i=1}^{n}m_{\nu(d)}(y_i)}{h_{n,n}^{(d)}(x)}\mathbf{1}(y\prec x) dy,\\
\nu_t^{n,n+1,\hat{h}^{(d)}_{n,n+1}}(dx,dy)&=\mu_t^{n+1,(d)}(dx) \frac{n!\prod_{1 \le i < j \le n}^{}(y_j-y_i)}{\prod_{1 \le i < j \le n+1}^{}(x_j-x_i)}\mathbf{1}(y\prec x)dy,
\end{align*}
for the processes with semigroups corresponding to the pair $(X,Y)$ described in Propositions \ref{BesqProp2} and \ref{BesqProp1} respectively, we immediately arrive at the following proposition in analogy to the case of Dyson's Brownian motion:

\begin{prop} \label{superpositionref2} \textbf{(a)}Let $d\ge 2$. Consider a process $(X,Y)\in W^{n,n+1}([0,\infty))$ started according to the entrance law $\nu_t^{n,n+1,\hat{h}^{(d)}_{n,n+1}}(dx,dy)$ with the $Y$ particles evolving as $n$ non-intersecting $BESQ(d+2)$ processes and the $X$ particles evolving as $n+1$ $BESQ(d)$ processes being reflected off the $Y$ particles. Then, the $X$ particles are distributed as $n+1$ non-intersecting $BESQ(d)$ processes issueing from the origin.\\
\textbf{(b)} Let $d\ge 2$. Consider a process $(X,Y)\in W^{n,n}([0,\infty))$ started according to the entrance law $\nu_t^{n,n,\hat{h}_{n,n}^{(d)}}(dx,dy)$ with the $Y$ particles which evolve as $n$ non-intersecting $BESQ(d)$ processes and the $X$ particles evolving as $n$ $BESQ(2-d)$ processes being reflected off the $Y$ particles. Then, the $X$ particles are distributed as $n$ non-intersecting $BESQ(d+2)$ processes issueing from the origin.
\end{prop}

Making use of the proposition above we build two processes in Gelfand-Tsetlin patterns. First, the process in $\mathbb{GT}(n)$. To do this, we make repeated use of part \textbf{(a)} of Proposition \ref{superpositionref2} to consistently concatenate two-level processes. Note the fact that the dimension $d$, of the $BESQ(d)$ processes, decreases by $2$ at each stage that we increase the number of particles. So we fix $n$ the depth of the Gelfand-Tsetlin pattern and $d^*$ the dimension of the $BESQ$ processes at the bottom of the pattern. Then, we build a consistent process,
\begin{align*}
\left(\mathbb{X}^{n,(d^*)}(t);t\ge 0\right)=(X_i^{(k)}(t);t\ge 0, 1 \le i\le k \le n) ,
\end{align*}
 taking values in $\mathbb{GT}(n)$ with the joint dynamics described as follows: $X_1^{(1)}$ evolves as a $BESQ(d^*+2(n-1))$ process. Moreover, for $k\ge 2$ particles at level $k$ evolve as $k$ independent $BESQ(d^*+2(n-k))$ processes reflecting off the $(k-1)$ particles at the $(k-1)^{th}$ level to maintain the interlacing. Hence, from Proposition \ref{multilevelproposition} (see discussion following it regarding the entrance laws) we obtain:
\begin{prop}
Let $d \ge 2$. If $\mathbb{X}^{n,{(d^*)}}$ is started from the origin according to the entrance law then the projection onto the $k^{th}$ level process $X^{(k)}$ is distributed as $k$ $BESQ(d^*+2(n-k))$ processes conditioned to never intersect.
\end{prop}

By making alternating use of parts \textbf{(a)} and \textbf{(b)} of Proposition \ref{superpositionref2} we construct a consistent process
\begin{align*}
\left(\mathbb{X}^{n,(d)}(t);t \ge 0\right)=(X^{(1)}(t)\prec \hat{X}^{(1)}(t)\prec \cdots \prec X^{(n)}(t)\prec \hat{X}^{(n)}(t);t\ge 0)
\end{align*}
in $\mathbb{GT}_\textbf{s}(n)$, for which Proposition \ref{CerenziaGelfand} can be viewed as the $d=1$ case, and whose joint dynamics are given as follows: $X_1^{(1)}$ evolves as a $BESQ(d)$ process. Then, for any $k$, the $k$ particles corresponding to $\hat{X}^{(k)}$ evolve as $k$ independent $BESQ(2-d)$ processes reflecting off the particles corresponding to $X^{(k)}$ in order for the interlacing to be maintained. Moreover, for $k\ge 2$ the $k$ particles corresponding to $X^{(k)}$ evolve as $k$ independent $BESQ(d)$ processes reflecting off the particles corresponding to $\hat{X}^{(k-1)}$ in order to maintain the interlacing.

 Then, it is a consequence of the symplectic analogue of Proposition \ref{multilevelproposition} (involving an entrance law, see the discussion following Proposition \ref{multilevelproposition}) that:

\begin{prop} \label{symplecticBESQ}
Let $d\ge 2$. If $\mathbb{X}^{n,{(d)}}$ is started from the origin then the projections onto $X^{(k)}$ and $\hat{X}^{(k)}$ are distributed as $k$ non-intersecting $BESQ(d)$ and $k$ non-intersecting $BESQ(d+2)$ processes respectively started from the origin.
\end{prop}

\paragraph{Connection to Wishart processes} We now spell out the connection between the processes constructed above and matrix valued diffusion processes by first considering the connection to $\mathbb{X}^{n,(d^*)}$, for $d^*$ even. Let $d^*=2$ for simplicity.

 Take $\left(A(t);t\ge 0\right)$ to be an $n\times n$ complex Brownian matrix and consider $\left(H(t);t\ge 0\right)=\left(A(t)A(t)^*;t \ge 0\right)$. This is called the Wishart process and was first studied in the real symmetric case by Marie-France Bru in \cite{Wishart}, see also \cite{Demni} for a detailed study in the Hermitian setting and some of its properties. Then, it is well known (first proven in \cite{O Connell}), we have that $\left(\lambda^{(k)}(t);t\ge0\right)$, the eigenvalues of the $k\times k$ minor of $\left(H(t);t \ge0 \right)$, evolve as $k$ non-colliding $BESQ(2(n-k+1))$ processes. These eigenvalues then interlace with $\left(\lambda^{(k-1)}(t);t\ge0\right)$ which evolve as $k-1$ non-colliding $BESQ(2(n-k+1)+1)$ processes with the \textit{fixed} time $T$ conditional density of $\lambda^{(k-1)}(T)$ given $\lambda^{(k)}(T)$ on $W^{k-1,k}(\lambda^{(k)}(T))$ being  $\Lambda^{\hat{h}^{(d)}_{k-1,k}}_{k-1,k}\left(\lambda^{(k)}(T),\cdot\right)$ (see Section 3 of \cite{FerrariFrings}, Section 3.3 of \cite{ForresterNagao}). Inductively (since for fixed $T$, $\lambda^{(n-k)}(T)$ is a Markov chain in $k$ see Section 4 of \cite{FerrariFrings}) this gives that the distribution at \textit{fixed} times $T$ of the vector $(\lambda^{(1)}(T),\cdots,\lambda^{(n)}(T))$ is uniform over the space of $\mathbb{GT}(n)$ with bottom level $\lambda^{(n)}(T)$. Moreover, by making use of this coincidence along \textit{space-like paths} one can write down the dynamical correlation kernel (along space-like paths) of the process we constructed from Theorem 1.3 of \cite{FerrariFringsPartial}.

\begin{rmk}
Although $\mathbb{X}^{n,(2)}$ and the minor process described in the preceding paragraph on single levels or at fixed times coincide, the interaction between consecutive levels of the minor process should be different from local hard reflection, although the dynamics of consecutive levels of the $LUE$ process have not been studied yet (as far as we know).
\end{rmk}

We now describe the random matrix model that parallels $\mathbb{X}_s^{n,{(d)}}$ for $d$ even. Start with a row vector $\left(A^{(d)}(t);t \ge 0\right)$ of $d/2$ independent standard complex Brownian motions, then $\left(X^{(d)}(t);t \ge 0\right)=\left(A^{(d)}(t)A^{(d)}(t)^*;t \ge 0\right)$ evolves as a one dimensional $BESQ(d)$ diffusion (this is really just the definition of a $BESQ(d)$ process). Now, add another independent complex Brownian motion to make $\left(A^{(d)}(t);t\ge 0\right)$ a row vector of length $d/2+1$. Then, $\left(X^{(d)}(t);t \ge 0\right)=\left(A^{(d)}(t)A^{(d)}(t)^*;t \ge 0\right)$ evolves as a $BESQ(d+2)$ process interlacing with the aforementioned $BESQ(d)$. At fixed times, the fact that the conditional distribution of the $BESQ(d)$ process given the position $x$ of the $BESQ(d+2)$ process is proportional to $y^{\frac{d}{2}-1} 1_{[0,x]}$ follows from the conditional laws in \cite{DiekerWarren} (see also \cite{Defosseux}) and will be spelled out in a few sentences. Now, make $\left(A^{(d)}(t);t \ge 0\right)$ a $2 \times \left(\frac{d}{2}+1\right)$ matrix by adding a row of $d/2+1$ independent complex Brownian motions, the eigenvalues of $\left(X^{(d)}(t);t\ge 0\right)=\left(A^{(d)}(t)A^{(d)}(t)^*;t \ge 0\right)$ evolve as $2$ $BESQ(d)$ processes which interlace with the $BESQ(d+2)$. We can continue this construction indefinitely by adding columns and rows successively of independent complex Brownian motions. As before, this eigenvalue process will coincide with $\mathbb{X}_s^{n,{(d)}}$ on single levels as stochastic processes but also at \textit{fixed} times as distributions of whole interlacing arrays. We elaborate a bit on this fixed time coincidence. For simplicity, let $T=1$. Let $A$ be an $n\times k$ matrix of independent standard complex normal random variables. Let $A'$ be the $n \times (k+1)$ matrix obtained from $A$ by adding to it a column of independent standard complex normal random variables. Let $\lambda$ be the $n$ eigenvalues of $AA^*$ and $\lambda'$ be the $n$ eigenvalues of $A'(A')^*$. We want the conditional density $\rho_{\lambda|\lambda'}(\lambda)$, of $\lambda$ given $\lambda'$, with respect to Lebesgue measure. From \cite{DiekerWarren} (see also \cite{Defosseux}) the conditional density $\rho_{\lambda'|\lambda}(\lambda)$ is given by,
 \begin{align*}
\rho_{\lambda'|\lambda}(\lambda)=\frac{\prod_{1\le i < j \le n}^{}(\lambda'_j-\lambda'_i)}{\prod_{1\le i < j \le n}^{}(\lambda_j-\lambda_i)}e^{-\sum_{i=1}^{n}(\lambda'_i-\lambda_i)}\textbf{1}(\lambda\prec \lambda')   \ .
 \end{align*}
Hence, by Bayes' rule, and recalling the law of the $LUE$ ensemble, we have,
\begin{align*}
\rho_{\lambda|\lambda'}(\lambda)=\left[\frac{\rho_{\lambda}}{\rho_{\lambda'}}\rho_{\lambda'|\lambda}\right](\lambda)=\frac{\prod_{1\le i < j \le n}^{}(\lambda_j-\lambda_i)\prod_{i=1}^{n}\lambda_i^{\frac{d}{2}-1}}{\prod_{1\le i < j \le n}^{}(\lambda'_j-\lambda'_i)\prod_{i=1}^{n}\lambda_i'^{\frac{d}{2}}}\textbf{1}(\lambda\prec \lambda') \ .
\end{align*}
Similarly to the case of $\mathbb{GT}$, by induction this gives fixed time coincidence of the two $\mathbb{GT}_\textbf{s}$ valued processes.

\subsection{Diffusions associated with orthogonal polynomials}

Here, we consider three diffusions in Gelfand-Tsetlin patterns associated with the classical orthogonal polynomials, Hermite, Laguerre and Jacobi. Although the one dimensional diffusion processes these are built from, the Ornstein-Uhlenbeck, the Laguerre and Jacobi are special cases of Sturm-Liouville diffusions with discrete spectrum, which we will consider in the next subsection, they are arguably the most interesting examples, with close connections to random matrices and so we consider them separately (for the classification of one dimensional diffusion operators with polynomial eigenfunctions see \cite{Mazet} and for a nice exposition Section 2.7 of \cite{MarkovDiffusion}). One of the common features of the Karlin-McGregor semigroups associated with them is that they all have the Vandermonde determinant as their ground state (this follows from Corollary \ref{minimal}). At the end of this subsection we describe the connection to eigenvalue processes of minors of matrix diffusions.

\begin{defn}
The Ornstein-Uhlenbeck (OU) diffusion process in $I=\mathbb{R}$ has generator and SDE description,
\begin{align*}
L_{OU}&=\frac{1}{2}\frac{d^2}{dx^2}-x\frac{d}{dx},\\
dX(t)&=dB(t)-X(t)dt,
\end{align*} 
with $m_{OU}(x)=e^{-x^2}$ and $-\infty$ and $\infty$ both natural boundaries. Its conjugate diffusion process $\hat{L}_{OU}$ has generator and SDE description,
\begin{align*}
\hat{L}_{OU}&=\frac{1}{2}\frac{d^2}{dx^2}+x\frac{d}{dx},\\
d\hat{X}(t)&=dB(t)+\hat{X}(t)dt,
\end{align*} 
and again $-\infty$ and $\infty$ are both natural boundaries and note the drift away from the origin.
\end{defn}

\begin{defn}
The Laguerre $Lag(\alpha)$ diffusion process in $I=[0,\infty)$ has generator and SDE description,
\begin{align*}
L_{Lag(\alpha)}&=2x\frac{d^2}{dx^2}+(\alpha-2x)\frac{d}{dx},\\ dX(t)&=2\sqrt{X(t)}dB(t)+(\alpha-2X(t))dt,
\end{align*} 
with $m_{Lag(\alpha)}(x)=x^{\alpha/2}e^{-x}$ and $\infty$ being natural and for $\alpha \ge 2$ the point $0$ is an entrance boundary. We will only be concerned with such values of $\alpha$ here.
\end{defn}

\begin{defn}
The Jacobi diffusion process $Jac(\beta,\gamma)$ in $I=[0,1]$ has generator and SDE description,
\begin{align*}
L_{Jac(\beta,\gamma)}&=2x(1-x)\frac{d^2}{dx^2}+2(\beta-(\beta+\gamma)x)\frac{d}{dx},\\
dX(t)&=2\sqrt{X(t)(1-X(t))}dB(t)+2(\beta-(\beta+\gamma)X(t))dt,
\end{align*}
with $m_{Jac(\beta,\gamma)}(x)=x^{\beta-1}(1-x)^{\gamma-1}$ and $0$ and $1$ being entrance for $\beta, \gamma \ge 1$. We will only be concerned with such values of $\beta$ and $\gamma$ in this section.
\end{defn}

The restriction of parameters $\alpha, \beta, \gamma$ for $Lag(\alpha)$ and $Jac(\beta,\gamma)$ is so that $(\mathbf{BC}+)$ is satisfied (for a certain range of the parameters the points $0$ and/or $1$ are regular boundaries in which case  $(\mathbf{BC}+)$ is no longer satisfied due to the fact that the diffusion coefficients degenerate at the boundary points).

We are interested in the construction of a process in $\mathbb{GT}(N)$, so that in particular at each stage the number of particles increases by one. We start in the simplest setting of $W^{1,2}$ and in particular the Ornstein-Uhlenbeck case to explain some subtleties. We will then treat all cases uniformly.

Consider a two-level process $(X,Y)$ with the $X$ particles evolving as two $OU$ processes being reflected off the $Y$ particle which evolves as an $\hat{L}_{OU}$ diffusion. Then, since this is an honest Markov process, Theorem \ref{MasterDynamics} (whose conditions are easily seen to be satisfied) gives that if started appropriately, the projection on the $X$ particles is Markovian with semigroup $P_t^{2,OU,\bar{h}_2}$. Here, $P_t^{2,OU,\bar{h}_2}$ is the Doob $h$-transformed semigroup of two independent $OU$ processes killed when they intersect by the harmonic function $\bar{h}_2$:
\begin{align*}
\bar{h}_2(x_1,x_2)=\int_{x_1}^{x_2}\hat{m}_{OU}(y)dy=s_{OU}(x_2)-s_{OU}(x_1),
\end{align*}
where $s_{OU}(x)=e^{x^2}F(x)$ is the scale function of the $OU$ process and $F(x)=e^{-x^2}\int_{0}^{x}e^{y^2}dy$ is the Dawson function. We note that, although this process is built from two $OU$ processes being kept apart (more precisely this diffusion lives in $\mathring{W}^2$), it is \textit{not two independent $OU$ processes conditioned to never intersect}.

However, we can initially $h$-transform the $\hat{L}_{OU}$ process to make it an $OU$ process with the $h$-transform given by $\hat{h}_1(x)=\hat{m}^{-1}_{OU}(x)$ with eigenvalue $-1$. Now, note that:
\begin{align*}
h_2(x_1,x_2)=\int_{x_1}^{x_2}\hat{m}_{OU}(y)\hat{m}^{-1}_{OU}(y)dy=(x_2-x_1).
\end{align*}
This, as we see later in Corollary \ref{minimal} is the ground state of the semigroup associated to two independent $OU$ processes killed when they intersect.
Thus, if we consider a two-level process $(X,Y)$ with the $X$ particles evolving as 2 $OU$ processes reflected off a single $OU$ process, we get from Theorem \ref{MasterDynamics} that the projection on the $X$ particles is distributed as two independent $OU$ processes conditioned to never intersect via a Doob $h$-transform by $h_2$.

 Similarly, an easy calculation gives that we can $h$-transform the $\hat{L}_{Lag(\alpha)}$-diffusion to make it a $Lag(\alpha+2)$ with the $h$-transform being $\hat{m}^{-1}_{Lag(\alpha)}(x)$ with eigenvalue $-2$ and $h$-transform with $\hat{m}^{-1}_{Jac(\beta,\gamma)}(x)$ with eigenvalue $-2(\beta+\gamma)$ the $\hat{L}_{Jac(\beta,\gamma)}$-diffusion to make it a $Jac(\beta+1,\gamma+1)$ to obtain the analogous result. Furthermore, this generalizes to arbitrary $n$. First, let 
\begin{align*}
h_{n+1}(x)&=\frac{1}{n!}\prod_{1 \le i < j \le n+1}^{}(x_j-x_i)
\end{align*}
denote the Vandermonde determinant. By Corollary \ref{minimal}, $h_{n+1}$ is the ground state of the semigroup associated to $n+1$ independent copies of an $OU$ or $Lag(\alpha)$ or $Jac(\beta,\gamma)$ diffusion killed when they intersect.

\begin{prop} 
Assume the constants $\alpha, \beta, \gamma$ satisfy $\alpha \ge 2, \beta \ge 1, \gamma \ge 1$. Let $(X,Y)$ be a two-level diffusion process in $W^{n,n+1}(I^{\circ})$ started according to the distribution $(\delta_x,\frac{n!\prod_{1 \le i < j \le n}^{}(y_j-y_i)}{\prod_{1 \le i < j \le n+1}^{}(x_j-x_i)}\mathbf{1}(y\prec x)dy)$, where $x\in \mathring{W}^{n+1}(I)$, and $X$ and $Y$ evolving as follows:\\
\textbf{OU}: $X$ as $n+1$ independent $OU$ processes reflected off $Y$ which evolves as $n$ $OU$ processes conditioned to never intersect via a Doob $h$-transform by $h_n$, \\
\textbf{Lag}: $X$ as $n+1$ independent $Lag(\alpha)$ processes reflected off  $Y$ which evolves as $n$ $Lag(\alpha+2)$ processes conditioned to never intersect via a Doob $h$-transform by $h_n$,\\
\textbf{Jac}: $X$ as $n+1$ independent $Jac(\beta,\gamma)$ processes reflected off $Y$ which evolves as $n$ $Jac(\beta+1,\gamma+1)$ processes conditioned to never intersect via a Doob $h$-transform by $h_n$.\\
Then, the $X$ particles are distributed as,\\
\textbf{OU}: $n+1$ $OU$ processes conditioned to never intersect via a Doob $h$-transform by $h_{n+1}$,\\
\textbf{Lag}: $n+1$ $Lag(\alpha)$ processes conditioned to never intersect via a Doob $h$-transform by $h_{n+1}$,\\
\textbf{Jac}: $n+1$ $Jac(\beta,\gamma)$ processes conditioned to never intersect via a Doob $h$-transform by $h_{n+1}$,\\
started from $x$.
\end{prop}

\begin{proof}
We take as the $L$-diffusion an $OU$ or $Lag(\alpha)$ or $Jac(\beta,\gamma)$ diffusion respectively. Note that, the assumptions $(\mathbf{R})$, $(\mathbf{BC+})$ and $(\mathbf{YW})$ are satisfied for $Lag(\alpha)$ for $\alpha \ge 2$ and for $Jac(\beta,\gamma)$ for $\beta \ge 1, \gamma \ge 1$ (also these assumptions are clearly satisfied for an $OU$ process). Furthermore, observe that with,
\begin{align*}
\hat{h}_n(x)&=\prod_{i=1}^{n}\hat{m}^{-1}(x)\prod_{1 \le i < j \le n}^{}(x_j-x_i),
\end{align*}
we have:
\begin{align*}
h_{n+1}(x)&=(\Lambda_{n,n+1}\Pi_{n,n+1}\hat{h}_n)(x)=\frac{1}{n!}\prod_{1 \le i < j \le n+1}^{}(x_j-x_i).
\end{align*}
Moreover, note that the semigroup of $n$ independent copies of an $\hat{L}$-diffusion (namely either an $\hat{L}_{OU}$ or $\hat{L}_{Lag(\alpha)}$ or $\hat{L}_{Jac(\beta,\gamma)}$ diffusion) killed when they intersect $h$-transformed by $\hat{h}_n$ is exactly the semigroup corresponding to $Y$ in the statement of the proposition. Finally, making use of Theorem \ref{MasterDynamics} we obtain the required statement.
\end{proof}

It is rather easy to see how to iterate this construction to obtain a consistent process in a Gelfand-Tsetlin pattern. To be precise, let us fix $N$ the depth of the pattern and constants $\alpha \ge 2$, $\beta \ge 1$ and $\gamma \ge 1$ that will be the parameters of the processes at the bottom row. Then, in the Ornstein-Uhlenbeck case level $k$ evolves as $k$ independent $OU$ processes reflected off the paths at level $k-1$. In the Laguerre case level $k$ evolves as $k$ independent $Lag(\alpha+2(N-k))$ processes reflected off the particles at level $k-1$. Finally, in the Jacobi case level $k$ evolves as $k$ independent $Jac(\beta+(N-k),\gamma+(N-k))$ processes reflected off the particles at level $(k-1)$. The result giving the distribution of the projection on each level (under certain initial conditions) is completely analogous to previous sections and we omit the statement.

\begin{rmk}
In the $Laguerre$ case we can build in a completely analogous way a process in $\mathbb{GT}_\textbf{s}$ in analogy to the $BESQ(d)$ case of Proposition \ref{symplecticBESQ}. In the Jacobi case (with $\beta,\gamma \ge 1$) we can build a process $(X,Y)\in W^{n,n}((0,1))$ started from the origin (according to the entrance law) with the $Y$ particles evolving as $n$ non-intersecting $Jac(\beta,\gamma+1)$ and the $X$ particles as $n$ $Jac(1-\beta,\gamma)$ in $(0,1)$ reflected off the $Y$ particles. Then, the $X$ particles are distributed as $n$ non-intersecting $Jac(\beta+1,\gamma)$ processes started from the origin. 
\end{rmk}

\paragraph{Connection to random matrices}
We now make the connection to the eigenvalues of matrix valued diffusion processes associated with orthogonal polynomials. The relation for the Ornstein-Uhlenbeck process and $Lag(d)$ processes we constructed is the same as for Brownian motions and $BESQ(d)$ processes. The only difference being, that we replace the complex Brownian motions by complex Ornstein-Uhlenbeck processes in the matrix valued diffusions (the only difference being, that this introduces a restoring $-x$ drift in both the matrix valued diffusion processes and the $SDEs$ for the eigenvalues). 

We now turn to the Jacobi minor process. First, following Doumerc's PhD thesis \cite{Doumerc} (see in particular Section 9.4.3 therein) we construct the matrix Jacobi diffusion as follows. Let $\left(U(t),t\ge0\right)$ be a Brownian motion on $\mathbb{U}(N)$, the manifold of $N \times N$ unitary matrices and let $p+q=N$. Let $n$ be such that $n\le p,q$ and consider $\left(H(t),t\ge0\right)$ the projection onto the first $n$ rows and $p$ columns of $\left(U(t),t\ge 0\right)$. Then $\left(J^{p,q}(t),t\ge0\right)=\left(H(t)H(t)^*,t\ge 0\right)$ is defined to be the $n \times n$ matrix Jacobi diffusion (with parameters $p,q$). Its eigenvalues evolve as $n$ non-colliding $Jac(p-(n-1),q-(n-1))$ diffusions. Its $k\times k$ minor is built by projecting onto the first $k$ rows of $\left(U(t),t\ge0 \right)$ and it has eigenvalues $\left(\lambda^{(k)}(t),t\ge 0\right)$ that evolve as $k$ non-colliding $Jac(p-(n-1)+n-k,q-(n-1)+n-k)$. For fixed times $T$, if $\left(U(t),t\ge0\right)$ is started according to Haar measure, the distribution of $\lambda^{(k-1)}(T)$ given $\lambda^{(k)}(T)$ on $W^{k-1,k}(\lambda^{(k)}(T))$ being $\Lambda_{k-1,k}^{\hat{h}_{k-1}}\left(\lambda^{(k)}(T),\cdot\right)$ see e.g. \cite{ForresterNagao}. For the connection to the process in $W^{n,n}$ described in the remark, we could have projected on the first $n$ rows and $p+1$ columns of $\left(U(t),t\ge 0\right)$ and denoting that by $\left(H(t)',t\ge 0\right)$, then $\left(J^{p+1,q-1}(t),t\ge 0\right)=\left(H(t)'(H(t)')^*,t\ge 0\right)$ has eigenvalues evolving as $n$ non-colliding $Jac(p-(n-1)+1,q-(n-1)-1)$ and those interlace with the eigenvalues of $\left(J^{p,q}(t),t\ge 0\right)$. 

\begin{rmk}
Non-colliding Jacobi diffusions have also appeared in the work of Gorin \cite{Gorin} as the scaling limits of some natural Markov chains on the Gelfand-Tsetlin graph in relation to the harmonic analysis of the infinite unitary group $\mathbb{U}(\infty)$.
\end{rmk}

\subsection{Diffusions with discrete spectrum}
\subsubsection{Spectral expansion and ground state of the Karlin-McGregor semigroup}
In this subsection, we show how the diffusions associated with the classical orthogonal polynomials and the Brownian motions in an interval are special cases of a wider class of one dimensional diffusion processes with explicitly known minimal eigenfunctions for the Karlin-McGregor semigroups associated with them. We start by considering the diffusion process generator $L$ with \textit{discrete spectrum} $0\ge -\lambda_1>-\lambda_2>\cdots$ (the absence of natural boundaries is sufficient for this, see for example Theorem 3.1 of \cite{McKean}) with speed measure $m$ and transition density given by $p_t(x,dy)=q_t(x,y)m(dy)$ where,
\begin{align*}
L\phi_k(x)&=-\lambda_k\phi_k(x),\\
q_t(x,y)&=\sum_{k=1}^{\infty}e^{-\lambda_kt}\phi_k(x)\phi_k(y).
\end{align*}
The eigenfunctions $\{\phi_k\}_{k\ge 1}$ form an orthonormal basis of $L^2(I,m(dx))$ and the expansion $\sum_{k=1}^{\infty}e^{-\lambda_kt}\phi_k(x)\phi_k(y)$ converges uniformly on compact squares in $I^{\circ}\times I^{\circ}$. Furthermore, the Karlin-McGregor semigroup transition density with respect to $\prod_{i=1}^{n}m(dy_i)$ is given by,
\begin{align*}
\det(q_t(x_i,y_j))_{i,j=1}^n.
\end{align*}
 We now obtain an analogous spectral expansion for this. We start by expanding the determinant to get,
\begin{align*}
\det(q_t(x_i,y_j))_{i,j=1}^n&=\sum_{\sigma \in \mathfrak{S}_n}^{}sign(\sigma)\prod_{i=1}^{n}q_t(x_i,y_{\sigma(i)})\\ &=\sum_{k_1,\cdots,k_n}^{}\prod_{i=1}^{n}\phi_{k_i}(x_i)e^{-\lambda_{k_i}t}\sum_{\sigma \in \mathfrak{S}_n}^{}sign(\sigma)\prod_{i=1}^{n}\phi_{k_i}(y_{\sigma(i)})\\
&=\sum_{k_1,\cdots,k_n}^{}\prod_{i=1}^{n}\phi_{k_i}(x_i)e^{-\lambda_{k_i}t}\det(\phi_{k_i}(y_j))_{i,j=1}^n. 
\end{align*}
Write  $\phi_{\mathsf{k}}(y)$  for $\det(\phi_{k_i}(y_j))_{i,j=1}^n$ for an n-tuple $\mathsf{k}=(k_1,\cdots,k_n)$ and also $\lambda_{\mathsf{k}}$ for $(\lambda_{k_1},\cdots, \lambda_{k_n})$ and note that we can restrict to $k_1,\cdots, k_n$ distinct otherwise the determinant vanishes. In fact we can restrict to $k_1,\cdots,k_n$ ordered by replacing $k_1,\cdots,k_n$ by $k_{\tau(1)},\cdots,k_{\tau(n)}$ and summing over $\tau \in \mathfrak{S}_n$ to obtain, with $|\lambda_{\mathsf{k}}|=\sum_{i=1}^{n}\lambda_{k_i}$:
\begin{align}
\det(q_t(x_i,y_j))_{i,j=1}^n=\sum_{1\le k_1 < \cdots <k_n}^{}e^{-|\lambda_{\mathsf{k}}|t}\phi_{\mathsf{k}}(x)\phi_{\mathsf{k}}(y).
\end{align}
The expansion is converging uniformly on compacts in $W^n(I^\circ)\times W^n(I^\circ)$ for $t>0$. Now, denoting by $T$ the lifetime of the process we obtain, for $x=(x_1,\cdots,x_n) \in \mathring{W}^n(I)$, the following spectral expansion that converges uniformly on compacts in $x\in W^n(I^\circ)$,
\begin{align}
\mathbb{P}_x(T>t)=\left[P_t^n\textbf{1}\right](x)=\sum_{1\le k_1 < \cdots <k_n}^{}e^{-|\lambda_{\mathsf{k}}|t}\phi_{\mathsf{k}}(x)\langle \phi_{\mathsf{k}},\textbf{1}\rangle_{W_{}^{n}(m)}
\end{align}
where we used the notation:
\begin{align*}
\langle f,g\rangle_{W_{}^{n}(m)}=\int_{W^n(I^\circ)}^{}f(x_1,\cdots,x_n)g(x_1,\cdots,x_n)\prod_{i=1}^{n}m(x_i)dx_i.
\end{align*}
So, as $t \to \infty$ by the fact that the eigenvalues are distinct and ordered the leading exponential term is forced to be $k_i=i$ and thus:
\begin{align*}
\mathbb{P}_x(T>t)= \langle \phi_{(1,\cdots,n)},\textbf{1}\rangle_{W_{}^{n}(m)} \times \ e^{-\sum_{i=1}^{n}\lambda_it}\det(\phi_i(x_j))_{i,j=1}^n +o\left(e^{-\sum_{i=1}^{n}\lambda_it}\right), \ \text{as} \ t \to \infty .
\end{align*} 
Hence, we can state the following corollary.
\begin{cor}\label{minimal} The function,
\begin{align}\label{GroundState}
h_n(x)=\det(\phi_i(x_j))_{i,j=1}^n
\end{align}
is the ground state of $P_t^n$.
\end{cor}

The above argument proves that $h_n(x)\ge 0$ but in fact the positivity is strict, $h_n(x)> 0$ for all $x \in  W^n(I^\circ)$ which can be seen as follows. We have the eigenfunction relation, by the Andreif (or generalized Cauchy-Binet) identity:
\begin{align*}
\int_{W^n(I^\circ)}^{}\det(q_t(x_i,y_j))_{i,j=1}^n\det(\phi_i(y_j))_{i,j=1}^n\prod_{i=1}^{n}m(y_i)dy_i=e^{-\sum_{i=1}^{n}\lambda_it}\det(\phi_i(x_j))_{i,j=1}^n.
\end{align*}
Assume that $\det(\phi_i(x_j))_{i,j=1}^n=0$ for some $x \in  W^n(I^\circ)$. Then, by the \textit{strict} positivity of $\det(q_t(x_i,y_j))_{i,j}^n\prod_{i=1}^{n}m(y_i)>0$ and continuity of $h_n(x)$ (see Theorem 4 of \cite{KarlinMcGregor}, also Problem 6 and its solution on pages 158-159 of \cite{ItoMckean}), the determinant $\det(\phi_i(y_j))_{i,j=1}^n$ must necessarily vanish everywhere in $W^n(I^\circ)$. Hence, we can write  for all $x \in I^{\circ}$ $\phi_n(x)=\sum_{i=1}^{n-1}a_i \phi_i(x)$ for some constants $a_i$. However, this contradicts the orthonormality of the eigenfunctions and so $h_n(x)> 0$ for all $x \in  W^n(I^\circ)$.

A different way to see that $h_n(x)$ is strictly positive (up to a constant) in $\mathring{W}^n(I)$ is the well known fact (see paragraph immediately after Theorem 6.2 of Chapter 1 on page 36 of \cite{Karlin}) that the eigenfunctions coming from Sturm-Liouville operators form a Complete $T$-system ($CT$-system) or $Chebyshev$ system namely $\forall n \ge 1$,
\begin{align*}
h_{n}(x)=\det(\phi_i(x_j))_{i,j=1}^{n}>0, \ x\in \mathring{W}^n(I).
\end{align*}

\begin{rmk}
In fact a $CT$-system requires that the determinant does not vanish in $W^n(I)$ so w.l.o.g multiplying by -1 if needed we can assume it is positive.
\end{rmk}

For the orthogonal polynomial diffusions and Brownian motions in an interval taking the $\phi_j$'s to be the Hermite, Laguerre, Jacobi polynomials (which via row and column operations give the Vandermonde determinant) and trigonometric functions (of increasing frequencies) we obtain the minimal eigenfunction.

Following this discussion, we can thus define the \textit{conditioned semigroup} with transition kernel $p_t^{n,h_n}$ with respect to Lebesgue measure in $W^n(I^\circ)$ as follows,
\begin{align*}
p_t^{n,h_n}(x,y)=e^{\sum_{i=1}^{n}\lambda_it}\frac{\det(\phi_i(y_j))_{i,j=1}^n}{\det(\phi_i(x_j))_{i,j=1}^n}\det(p_t(x_i,y_j))_{i,j=1}^n.
\end{align*}
\subsubsection{Conditioning diffusions for non-intersection through local interactions}
Now, a natural question arising is the following. When is it possible to obtain $n$ conservative (by that we mean in case $l$ or $r$ can be reached then they are forced to be regular reflecting) $L$-diffusions \textit{conditioned via the minimal positive eigenfunction} to never intersect through the hard reflection interactions we have been studying in this work? We are able to provide an answer in Proposition \ref{conditioningviareflection} below under a certain assumption that we now explain. 

First, note that $L$ being conservative implies $\phi_1=1$. Furthermore, assuming that the $\phi_k\in C^{n-1}(I^\circ)$ for $1\le k\le n$ and denoting by $\phi_k^{(j)}$ their $j^{th}$ derivative we define the Wronskian $W(\phi_1,\cdots,\phi_n)(x)$ of $\phi_1,\cdots,\phi_n$ by,
\begin{align*}
W(\phi_1,\cdots,\phi_n)(x)=\det\big(\phi_i^{(j-1)}(x)\big)_{i,j=1}^n.
\end{align*}

Then, we say that $\{\phi_j\}^n_{j=1}$ form a (positive) Extended Complete $T$-system or $ECT$-system if for all $1\le k \le n$, 
\begin{align*}
W(\phi_1,\cdots,\phi_k)(x)>0, \ \forall x \in I^\circ .
\end{align*}

This is a stronger property, in particular implying that $\{\phi_j\}^n_{j=1}$ form a $CT$-system (see Theorem 2.3 of Chapter 2 of \cite{Karlin}). Assuming that the eigenfunctions in question $\{\phi_j\}^n_{j=1}$ form a (positive) $ECT$-system then since $\phi_1=1$,
\begin{align*}
W(\phi_2^{(1)},\cdots,\phi_n^{(1)})(x)>0, \ \forall x \in I^\circ,
\end{align*}
and hence,
\begin{align}\label{DiscreteSpectrumEigenfunction}
\hat{h}_{n-1}(x):=\det(\mathcal{D}_{\hat{m}}\phi_{i+1}(x_j))_{i,j=1}^{n-1}>0, \ x\in \mathring{W}^{n-1}(I).
\end{align}
We then have the following positive answer for the question we stated previously:

\begin{prop}\label{conditioningviareflection}
Under the conditions of Theorem \ref{MasterDynamics}, furthermore assume that the generator $L$ has discrete spectrum and its first $n$ eigenfunctions $\{\phi_j\}^n_{j=1}$ form an ECT-system. Now assume that the $X$ particles consist of $n$ independent $L$-diffusions reflected off the $Y$ particles which evolve as an $n-1$ dimensional diffusion with semigroup $P_t^{n-1,\hat{h}_{n-1}}$, where $\hat{h}_{n-1}$ is defined in (\ref{DiscreteSpectrumEigenfunction}). Then, the $X$ particles (if the two-level process is started appropriately) are distributed as $n$ independent $L$-diffusions conditioned to never intersect with semigroup $P_t^{n,h_n}$, where $h_n$ is defined by (\ref{GroundState}).
\end{prop}

\begin{proof}
Making use of the relations $\mathcal{D}_{\hat{m}}=\mathcal{D}_{s}$ and $\mathcal{D}_{\hat{s}}=\mathcal{D}_{m}$ between the diffusion process generator $L$ and its dual we obtain,
\begin{align*}
\hat{L}\mathcal{D}_{\hat{m}}\phi_{i}=\mathcal{D}_{\hat{m}}\mathcal{D}_{\hat{s}}\mathcal{D}_{\hat{m}}\phi_i=\mathcal{D}_{\hat{m}}\mathcal{D}_{m}\mathcal{D}_{s}\phi_i=-\lambda_i\mathcal{D}_{\hat{m}}\phi_i.
\end{align*} 
Thus, $\left(e^{\lambda_it}\mathcal{D}_{\hat{m}}\phi_{i}(\hat{X}(t));t\ge 0\right)$ for each $1\le i \le n$ is a local martingale. By virtue of boundedness (since we assume that the $L$-diffusion is conservative we have $\underset{x\to l,r}{\lim}\mathcal{D}_{\hat{m}}\phi_i(x)=\underset{x\to l,r}{\lim}\mathcal{D}_{s}\phi_i(x)=0$) it is in fact a true martingale and so for $1\le i \le n$,
\begin{align}\label{truemartingale}
\hat{P}_t^1\mathcal{D}_{\hat{m}}\phi_{i}=e^{-\lambda_it}\mathcal{D}_{\hat{m}}\phi_{i}.
\end{align}
Then, by the well-known Andreif (or generalized Cauchy-Binet) identity we obtain,
\begin{align*}
\hat{P}_t^{n-1}\hat{h}_{n-1}=e^{-\sum_{i=1}^{n-1}\lambda_{i+1}t}\hat{h}_{n-1}
\end{align*}
and thus $\hat{h}_{n-1}$ is a strictly positive eigenfunction for $\hat{P}_t^{n-1}$. Finally, by performing a simple integration we see that,
\begin{align*}
(\Lambda_{n-1,n}\Pi_{n-1,n}\hat{h}_{n-1})(x)=const_n h_n(x), \ x \in W^n(I).
\end{align*} 
Using Theorem \ref{MasterDynamics} we obtain the statement of the proposition.
\end{proof}

Obviously the diffusions associated with orthogonal polynomials and Brownian motions in an interval fall under this framework.

\subsection{Eigenfunctions via intertwining} \label{subsectioneigen}

In this short subsection we point out that all eigenfunctions for $n$ copies of a diffusion process with generator $L$ in $W^n$ (not necessarily diffusions with discrete spectrum e.g. Brownian motions or $BESQ(d)$ processes) that are obtained by iteration of the intertwining kernels considered in this work, or equivalently from building a process in a Gelfand-Tsetlin pattern, are of the form,
\begin{align}\label{eigenfunctionrepresentation1}
\mathfrak{H}_n(x_1,\cdots,x_n)=\det\left(h_i^{(n)}(x_j)\right)^n_{i,j=1},
\end{align} 
for functions $\left(h_1^{(n)},\cdots,h_n^{(n)}\right)$ (not necessarily the eigenfunctions of a one dimensional diffusion operator) given by,
\begin{align}\label{eigenfunctionrepresentation2}
h_i^{(n)}(x)=w^{(n)}_1(x)\int_{c}^{x}w^{(n)}_2(\xi_1)\int_{c}^{\xi_1}w^{(n)}_3(\xi_2)\cdots \int_{c}^{\xi_{i-2}}w^{(n)}_i(\xi_{i-1})d\xi_{i-1}\cdots d\xi_1,
\end{align}
for some weights $w_i^{(n)}(x)>0$ and $c\in I^\circ$. An easy consequence of the representation above (see e.g. Theorem 1.1 of Chapter 6 of \cite{Karlin}) and assuming $w_i^{(n)}\in C^{n-i}(l,r)$ ($n-i$ times continuously differentiable) is that the Wronskian $W\left(h_1^{(n)},\cdots,h_n^{(n)}\right)$ is given by for $x \in I^{\circ}$,
\begin{align}
W\left(h_1^{(n)},\cdots,h_n^{(n)}\right)(x)=\left[w_1^{(n)}(x)\right]^n\left[w_2^{(n)}(x)\right]^{n-1}\cdots\left[w_n^{(n)}(x)\right],
\end{align}
so that in particular $W\left(h_1^{(n)},\cdots,h_n^{(n)}\right)(x)>0$.

We shall restrict to the case of $\mathbb{GT}(n)$ (where the number of particles on each level increases by $1$) for simplicity and prove claims (\ref{eigenfunctionrepresentation1}) and (\ref{eigenfunctionrepresentation2}) by induction. For $n=1$ there is nothing to prove. We conclude by stating and proving the inductive step as a precise proposition:

\begin{prop}
Assume that the input, strictly positive, eigenfunction $\mathfrak{H}_{n-1}$ for $n-1$ copies of a one dimensional diffusion process is of the form (\ref{eigenfunctionrepresentation1}) and (\ref{eigenfunctionrepresentation2}). Then, the eigenfunction $\mathfrak{H}_{n}$ built from the intertwining relation of Karlin-McGregor semigroups (\ref{KMintertwining}) for $n$ copies of its dual diffusion has the same form (\ref{eigenfunctionrepresentation1}) and (\ref{eigenfunctionrepresentation2}), with the weights $\{w_i^{(n)}\}_{i=1}^n$ satisfying an explicit recursion in terms of the $\{w_i^{(n-1)}\}_{i=1}^{n-1}$.
\end{prop}

\begin{proof}
In order to obtain a strictly positive eigenfunction for $n$ copies of an $L$-diffusion, we can in fact start more generally with $n$ copies of an $L$-diffusion $h$-transformed by a one dimensional strictly positive eigenfunction $h$ (denoting by $L^h$ such a diffusion process where we assume that $L^h$ satisfies the boundary conditions of Section 2 in order for the intertwining (\ref{KMintertwining}) to hold). It is then clear that:
\begin{align}
\mathfrak{H}_n(x_1,\cdots,x_n)=\prod_{i=1}^{n}h(x_i)(\Lambda_{n-1,n}\mathfrak{H}_{n-1})(x_1,\cdots,x_n),
\end{align}
where now $\mathfrak{H}_{n-1}(x_1,\cdots,x_{n-1})$ is a strictly positive eigenfunction of $n-1$ copies of an $\widehat{L^h}$ diffusion and which by our hypothesis is given by,
\begin{align}
\mathfrak{H}_{n-1}(x_1,\cdots,x_{n-1})=\det\left(h_i^{(n-1)}(x_j)\right)^{n-1}_{i,j=1},
\end{align}
for some functions $\left(h_1^{(n-1)},\cdots,h_{n-1}^{(n-1)}\right)$ with a representation as in (\ref{eigenfunctionrepresentation2}) for some weights $\{w_{i}^{(n-1)}\}_{i\le n-1}$. A simple integration now gives,
\begin{align*}
h^{(n)}_1(x)&=h(x),\\
h_i^{(n)}(x)&=h(x)\int_{c}^{x}\widehat{m^h}(y)h_{i-1}^{(n-1)}(y)dy, \ \  \textnormal{for} \ i \ge 2,
\end{align*}
where $\widehat{m^h}(x)=h^{-2}(x)s'(x)$ is the density of the speed measure of a $\widehat{L^h}$ diffusion. We thus obtain the following recursive representation for the weights $\{w_i^{(n)}\}_{i\le n}$,
\begin{align}
w_1^{(n)}(x)&=h(x),\\
w_2^{(n)}(x)&=h^{-2}(x)s'(x)w_1^{(n-1)}(x),\\
w_i^{(n)}(x)&=w_{i-1}^{(n-1)}(x), \ \  \textnormal{for} \ i \ge 3.
\end{align}
\end{proof}

\subsection{Connection to superpositions and decimations}
For particular entrance laws, the joint law of X and Y at a fixed time can be interpreted in terms of superpositions/decimations
of random matrix ensembles (see e.g. \cite{ForresterRains}). For example, in the context of Proposition \ref{superpositionref1}, the joint law of X and Y at time 1 agrees with the joint law of the odd (respectively even) 
eigenvalues in a superposition of two independent samples from the $GOE_{n+1}$ and $GOE_n$ ensembles, consistent with the fact that 
in such a superposition, the odd (respectively even) eigenvalues are distributed according to the $GUE_{n+1}$ (respectively $GUE_n$) 
ensembles, see Theorem 5.2 in \cite{ForresterRains}.  In the BESQ/Laguerre case, our Proposition \ref{superpositionref2} is similarly related to recent work on GOE singular 
values by Bornemann and La Croix \cite{BornemannLaCroix} and Bornemann and 
Forrester \cite{BornemannForrester}.

\subsection{Connection to strong stationary duals}

Strong stationary duality (SSD) first introduced by Diaconis and Fill \cite{DiaconisFill} in the discrete state space setting is a fundamental notion in the study of strong stationary times which are a key tool in understanding mixing times of Markov Chains. More recently, Fill and Lyzinski \cite{FillLyzinski} developed an analogous theory for diffusion processes in compact intervals. Given a conservative diffusion $\mathcal{G}$ one associates to it a SSD $\mathcal{G}^*$ such that the two semigroups are intertwined (see Definition 3.1 there). In Theorem 3.4 therein the form of the dual generator is derived and as already indicated in Remark 5.4 in the same paper this is exactly the dual diffusion $\hat{\mathcal{G}}$ $h$-transformed by its scale function. 

In our framework, considering a two-level process in $W^{1,1}$ with $L=\hat{\mathcal{G}}$ and so $\hat{L}=\mathcal{G}$ and using the positive harmonic function $\hat{h}_1\equiv 1$, the distribution of the projection on the $X$ particle (under certain initial conditions) coincides with the SSD $\mathcal{G}^*$ diffusion.
 Hence this provides a coupling of a diffusion $\mathcal{G}$ and its strong stationary dual $\mathcal{G}^*$ respecting the intertwining between $\mathcal{G}$ and $\mathcal{G}^*$.

\section{Edge particle systems}

In this section we will study the autonomous particle systems at either edge of the Gelfand-Tsetlin pattern valued processes we have constructed. In the figure below, the particles we will be concerned with are denoted in $\bullet$.
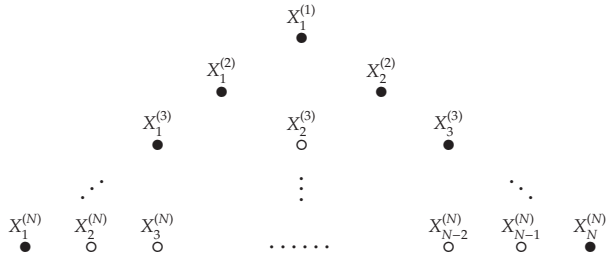
\begin{center}
\begin{tabular}{ c c c c c c c c c }
 &  &  &&$\overset{X^{(1)}_1}{\bullet}$ &&&&\\ 
 &  &  & $\overset{X^{(2)}_1}{\bullet} $&&$\overset{X^{(2)}_2}{\bullet} $&&&  \\ 
     &  &$\overset{X^{(3)}_1}{\bullet} $   &&$\overset{X^{(3)}_2}{\circ} $ &&$\overset{X^{(3)}_3}{\bullet} $&&  \\ 
  &$\iddots$ & & &$\vdots$& &&$\ddots$&\\
$\overset{X^{(N)}_1}{\bullet} $ &  $\overset{X^{(N)}_2}{\circ} $  &$\overset{X^{(N)}_3}{\circ} $&&$\cdots\cdots$&& $\overset{X^{(N)}_{N-2}}{\circ} $& $\overset{X^{(N)}_{N-1}}{\circ} $&$\overset{X^{(N)}_N}{\bullet} $
\end{tabular}
\end{center}
Our goal is to derive determinantal expressions for their transition densities. Such expressions were derived by Schutz for TASEP in \cite{Schutz} and later Warren \cite{Warren} for Brownian motions. See also Johansson's work in \cite{Johansson}, for an analogous formula for a Markov chain related to the Meixner ensemble and finally Dieker and Warren's investigation in \cite{DiekerWarren2}, for formulae in the discrete setting based on the RSK correspondence. These so called Schutz-type formulae were the starting points for the recent complete solution of TASEP in \cite{KPZfixedpoint} which led to the KPZ fixed point and also for the recent progress \cite{JohanssonPercolation} in the study of the two time joint distribution in Brownian directed percolation. For a detailed investigation of the Brownian motion model the reader is referred to the book \cite{ReflectedBrownianKPZ}.

We will mainly restrict ourselves to the consideration of Brownian motions, $BESQ(d)$ processes and the diffusions associated with orthogonal polynomials. In a little bit more generality we will assume that the interacting diffusions have generators of the form,
\begin{align*}
L=a(x)\frac{d^2}{dx^2}+b(x)\frac{d}{dx},
\end{align*}
with,
\begin{align*}
a(x)=a_0+a_1x+a_2x^2 \ \ \ b(x)=b_0+b_1x.
\end{align*}
We will also make the following \textbf{standing assumption} in this section. We restrict to the case of the boundaries of the state space $I$ being either \textit{natural} or \textit{entrance} thus the state space is an open interval $(l,r)$. Under these assumptions the transition densities will be smooth in $(l,r)$ in both the backwards and forwards variables (possibly blowing up as we approach $l$ or $r$ see e.g \cite{Stroock} and for a detailed study of the transition densities of the Wright-Fisher diffusion see \cite{WrightFisher}). This covers all the processes we built that relate to minor processes of matrix diffusions. This interacting particle system can also be seen as the solution to the following system of $SDE$'s with one-sided collisions with $(x_1^1 \le \cdots \le x_n^n)$,
\begin{align}\label{edgesystem1}
X_1^{(1)}(t)&=x^1_1+\int_{0}^{t}\sqrt{2a(X_1^{(1)}(s))}d\gamma_1^1(s)+\int_{0}^{t}b^{(1)}(X_1^{(1)}(s))ds,\nonumber\\
&\vdots\nonumber\\
X_m^{(m)}(t)&=x_m^m+\int_{0}^{t}\sqrt{2a(X_m^{(m)}(s))}d\gamma_m^m(s)+\int_{0}^{t}b^{(m)}(X_m^{(m)}(s))ds+K_m^{m,-}(t),\\
&\vdots\nonumber\\
X_n^{(n)}(t)&=x_n^n+\int_{0}^{t}\sqrt{2a(X_n^{(n)}(s))}d\gamma_n^n(s)+\int_{0}^{t}b^{(n)}(X_n^{(n)}(s))ds+K_n^{n,-}(t).\nonumber
\end{align}
where $\gamma_i^i$ are independent standard Brownian motions and $K_i^{i,-}$ are positive finite variation processes with the measure $dK_i^{i,-}$ supported on $\left\{t:X_i^{(i)}(t)=X_{i-1}^{(i-1)}(t)\right\}$ and
\begin{align*}
b^{(k)}(x)=b(x)+(n-k)a'(x)=b_0+(n-k)a_1+(b_1+2(n-k)a_2)x.
\end{align*}
That these $SDE$'s are well-posed, so that in particular the solution is Markov, follows from the same arguments as in Section \ref{SectionWellposedness}. Note that, a quadratic diffusion coefficient $a(\cdot)$ and linear drift $b(\cdot)$ satisfy $(\mathbf{YW})$. See the following figure for a description of the interaction. The arrows indicate the direction of the 'pushing force' (with magnitude the finite variation process $K$) applied when collisions occur between the particles so that the ordering is maintained.
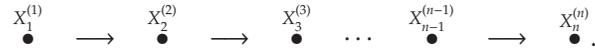
\begin{center}
\begin{tabular}{ c c c c c c c c c }
 $\overset{X^{(1)}_1}{\bullet}$&$\longrightarrow$ &$\overset{X^{(2)}_2}{\bullet}$&$\longrightarrow$ &$\overset{X^{(3)}_3}{\bullet}$ &$\cdots$ &$\overset{X^{(n-1)}_{n-1}}{\bullet}$&$\longrightarrow$&$\overset{X^{(n)}_{n}}{\bullet}$.
 \end{tabular}
\end{center}
Note that our assumption that the boundary points are either \textit{entrance} or \textit{natural} does not always allow for an \textit{infinite} such particle system,in particular think of the $BESQ(d)$ case where $d$ drops down by $2$ each time we add a particle. Denote by $p_t^{(k)}(x,y)$ the transition kernel associated with  the $L^{(k)}$-diffusion with generator,
\begin{align*}
L^{(k)}=a(x)\frac{d^2}{dx^2}+b^{(k)}(x)\frac{d}{dx}.
\end{align*}
Defining,
\begin{align*}
\mathcal{S}_t^{(k),j}(x,x')=\begin{cases}
\int_{l}^{x'}\frac{(x'-z)^{j-1}}{(j-1)!}p_t^{(k)}(x,z)dz \ \ j\ge 1\\
\partial_{x'}^{-j}p_t^{(k)}(x,x') \ \ j \le 0
\end{cases},
\end{align*}
and with $x=(x_1,\cdots,x_n)$, $x'=(x_1',\cdots,x_n')$,
\begin{align}
s_t(x,x')=\det\left(\mathcal{S}_t^{(i),i-j}(x_i,x_j')\right)^n_{i,j=1},
\end{align}
we arrive at the following proposition.
\begin{prop}\label{PropositionEdgeParticle}
Assume that the diffusion and drift coefficients of the generators $L^{(k)}$ are of the form $a(x)=a_0+a_1x+a_2x^2$ and $b^{(k)}(x)=b_0+(n-k)a_1+(b_1+2(n-k)a_2)x$ and moreover assume that the boundaries of the state space are either natural or entrance for the $L^{(k)}$-diffusion; in particular this implies certain constraints on the constants $a_0, a_1, a_2, b_0, b_1$. Then, the process $(X_1^{(1)}(t),\cdots,X_n^{(n)}(t))$ satisfying the $SDEs$ (\ref{edgesystem1}), in which $X^{(k)}_k$ is an $L^{(k)}$-diffusion reflected off $X^{(k-1)}_{k-1}$, has transition densities $s_t(x,x')$.
\end{prop}
\begin{proof}
First, we make the following crucial observation. Define the constant $c_{k,n}=2(n-k-1)a_2+b_1$ and note that the $L^{(k)}$-diffusion is the $h$-transform of the conjugate $\widehat{L^{(k+1)}}$ with $\widehat{m^{(k+1)}}^{-1}(x)$ with eigenvalue $c_{k,n}$, so that $L^{(k)}=\left(\widehat{L^{(k+1)}}\right)^{*}-c_{k,n}$ which is again a bona fide diffusion process generator (with $\mathsf{L}^{*}$ denoting the formal adjoint of $\mathsf{L}$ with respect to Lebesgue measure). Thus, making use of (\ref{conjtrans}) and (\ref{symmetrizing}) we obtain the following relation between the transition densities,
\begin{align}
p_t^{(k)}(x,z)=-e^{c_{k,n}t}\int_{l}^{z}\partial_xp_t^{(k+1)}(x,w)dw, \label{onesidedtransitionrelation}\\
\partial_z^jp_t^{(k)}(x,z)=-e^{c_{k,n}t}\partial_z^{j-1}\partial_xp_t^{(k+1)}(x,z).\nonumber
\end{align}

Now, let $f:{W}^n(I^\circ)\mapsto \mathbb{R}$ be continuous with compact support. Then, we have the following $t=0$ boundary condition,
\begin{align}\label{timezeroedge}
\lim_{t \to 0}\int_{W^n(I^\circ)}^{}s_t(x,x')f(x')dx'=f(x),
\end{align}
which formally can easily be seen to hold since the transition densities along the main diagonal approximate delta functions and all other contributions vanish. We spell this out now. Let $\epsilon>0$ and suppose $f$ is zero in a $2\epsilon$ neighbourhood of $\partial W^n(I^\circ)$. We consider a contribution to the Leibniz expansion of the determinant coming from a permutation $\rho$ that is not the identity. Hence there exist $i<j$ so that $\rho(i)>i$ and $\rho(j)\le i$ and note that the factors $\mathcal{S}_t^{(i),i-\rho(i)}\left(x_i,x_{\rho(i)}'\right)$ and $\mathcal{S}_t^{(j),j-\rho(j)}\left(x_j,x_{\rho(j)}'\right)$ are contained in the contribution corresponding to $\rho$. Since $j-\rho(j)>0$ and $i-\rho(i)<0$ observe that on the set $\left\{x_{\rho(i)}'-x_i>\epsilon\right\}\cup\left\{x_{\rho(j)}'-x_j<-\epsilon\right\}$ at least one of these factors and so the whole contribution as $t \downarrow 0$
vanishes uniformly. On the other hand on the complement of this set we have $x_{\rho(i)}'\le x_i+\epsilon\le x_j+\epsilon\le x'_{\rho(j)}+2\epsilon$. Since $\rho(j)<\rho(i)$ so that $x'_{\rho(j)}\le x'_{\rho(i)}$ we thus obtain that if $x'$ is in the complement of $\left\{x_{\rho(i)}'-x_i>\epsilon\right\}\cup\left\{x_{\rho(j)}'-x_j<-\epsilon\right\}$ it also belongs to some $2\epsilon$ neighbourhood of $\partial W^n(I^\circ)$ and hence outside the support of $f$. (\ref{timezeroedge}) then follows.

Now by multilinearity of the determinant the equation in $(0,\infty)\times \mathring{W}^n(I)\times \mathring{W}^n(I)$,
\begin{align*}
\partial_ts_t(x,x')=\sum_{i=1}^{n}L^{(k)}_{x_i}s_t(x,x'),
\end{align*}
 is satisfied since we have  $\partial_t\mathcal{S}_t^{(k),j}(x,x')=L_x^{(k)}\mathcal{S}_t^{(k),j}(x,x')$ for all $k$. Here, $L^{(k)}_{x_i}$ is simply a copy of the differential operator $L^{(k)}$ acting in the $x_i$ variable.
 
Moreover, for the Neumann/reflecting boundary conditions we need to check the following conditions $\partial_{x_i}s_t(x,x')|_{x_i=x_{i-1}}=0$ for  $i=2,\cdots,n$.\\
This follows from,
\begin{align*}
\partial_{x_i}\mathcal{S}_t^{(i),i-j}(x_i,x_j')|_{x_i=x_{i-1}}=-e^{-c_{i-1,n}t}\mathcal{S}_t^{(i-1),i-1-j}(x_{i-1},x_j').
\end{align*}
This is true because of the following observations. For $j \le -1$
\begin{align*}
 \partial_z^{-j}p_t^{(i-1)}(x,z)=-e^{c_{i-1,n}t}\partial_z^{-j-1}\partial_xp_t^{(i)}(x,z).
\end{align*}
For $j \ge 1$
\begin{align*}
 &\int_{l}^{x'}\frac{(x'-z)^{j-1}}{(j-1)!}p_t^{(i-1)}(x,z)dz=-e^{c_{i-1,n}t}\partial_x\int_{l}^{x'}\frac{(x'-z)^{j-1}}{(j-1)!}\int_{l}^{z}p_t^{(i)}(x,w)dwdz\\
 &=-e^{c_{i-1,n}t}\partial_x \bigg[\bigg[-\frac{(x'-z)^j}{j!}\int_{l}^{z}p_t^{(k)}(x,w)dw\bigg]_{l}^{x'}-\int_{l}^{x'}-\frac{(x'-z)^j}{j!}p_t^{(i)}(x,z)dz\bigg]\\
 &=-e^{c_{i-1,n}t}\partial_x\int_{l}^{x'}\frac{(x'-z)^{j}}{j!}p_t^{(i)}(x,z)dz.
\end{align*}
Hence $\mathcal{S}_t^{(i-1),j}(x,x')=-e^{c_{i-1,n}t}\partial_x\mathcal{S}_t^{(i),j+1}(x,x')$ and thus
\begin{align*}
\partial_{x_i}s_t(x,x')|_{x_i=x_{i-1}}=0,
\end{align*}
 for $i=2,\cdots,n$.\\
 Define for $f$ as in the first paragraph,
 \begin{align*}
 F(t,x)=\int_{W^n(I^\circ)}^{}s_t(x,x')f(x')dx'.
 \end{align*}
 Let $\textbf{S}_x$ denote the law of $(X_1^{(1)},\cdots,X_n^{(n)})$ started from $x=(x_1,\cdots,x_n) \in W^n$. Fixing $T,\epsilon$ and applying Ito's formula to the process $(F(T+\epsilon-t,x), t \le T)$ we obtain that it is a local martingale and by virtue of boundedness indeed a true martingale. Hence,
 \begin{align*}
 F(T+\epsilon,x)&=\textbf{S}_x\left[F\left(\epsilon,\left(X_1^{(1)}(T),\cdots,X_n^{(n)}(T)\right)\right)\right].
 \end{align*}
  Now letting $\epsilon \downarrow 0$ we obtain,
 \begin{align*}
 F(T,x)=\textbf{S}_x\left[f\left(X_1^{(1)}(T),\cdots,X_n^{(n)}(T)\right)\right].
 \end{align*}
 The result follows since the process spends zero Lebesgue time on the boundary so that in particular such $f$ determine its distribution.
\end{proof}
In the standard Brownian motion case with $p_t^{(k)}$ the heat kernel this recovers Proposition 8 from \cite{Warren}.

Now, we consider the interacting particle system at the other edge of the pattern with the $i^{th}$ particle getting reflected downwards from the $i-1^{th}$, namely with $(x_1^1\ge \cdots\ge x_1^n)$ this is given by the following system of $SDEs$ with reflection,
\begin{align}\label{edgesystem2}
X_1^{(1)}(t)&=x^1_1+\int_{0}^{t}\sqrt{2a(X_1^{(1)}(s))}d\gamma_1^1(s)+\int_{0}^{t}b^{(1)}(X_1^{(1)}(s))ds,\nonumber\\
&\vdots\nonumber\\
X_1^{(m)}(t)&=x_1^m+\int_{0}^{t}\sqrt{2a(X_1^{(m)}(s))}d\gamma_1^m(s)+\int_{0}^{t}b^{(m)}(X_1^{(m)}(s))ds-K_1^{m,+}(t),\\
&\vdots\nonumber\\
X_1^{(n)}(t)&=x_1^n+\int_{0}^{t}\sqrt{2a(X_1^{(n)}(s))}d\gamma_1^n(s)+\int_{0}^{t}b^{(n)}(X_1^{(n)}(s))ds-K_1^{n,+}(t),\nonumber
\end{align}
where $\gamma_1^i$ are independent standard Brownian motions and $K_1^{i,+}$ are positive finite variation processes with the measure $dK_1^{i,+}$ supported on $\left\{t:X_i^{(i)}(t)=X_{i-1}^{(i-1)}(t)\right\}$. Again see the figure below,
\begin{center}
\begin{tabular}{ c c c c c c c c c }
 $\overset{X^{(n)}_1}{\bullet}$&$\longleftarrow$ &$\overset{X^{(n-1)}_1}{\bullet}$&$\longleftarrow$ &$\overset{X^{(n-2)}_1}{\bullet}$ &$\cdots$ &$\overset{X^{(2)}_1}{\bullet}$&$\longleftarrow$&$\overset{X^{(1)}_1}{\bullet}$.
 \end{tabular}
\end{center}
Define,
\begin{align*}
\bar{\mathcal{S}}_t^{(k),j}(x,x')=\begin{cases}
-\int_{x'}^{r}\frac{(x'-z)^{j-1}}{(j-1)!}p_t^{(k)}(x,z)dz \ \ j\ge 1\\
\partial_
{x'}^{-j}p_t^{(k)}(x,x') \ \ j \le 0
\end{cases},
\end{align*}
Then letting, with $x=(x_1,\cdots,x_n)$, $x'=(x_1',\cdots,x_n')$,
\begin{align}
\bar{s}_t(x,x')=\det(\bar{\mathcal{S}}_t^{(i),i-j}(x_i,x_j'))^n_{i,j=1},
\end{align} 
we arrive at the following proposition.
\begin{prop}
Assume that the diffusion and drift coefficients of the generators $L^{(k)}$ are of the form $a(x)=a_0+a_1x+a_2x^2$ and $b^{(k)}(x)=b_0+(n-k)a_1+(b_1+2(n-k)a_2)x$ and moreover assume that the boundaries of the state space are either natural or entrance for the $L^{(k)}$-diffusion. Then, the process $(X_1^{(1)}(t),\cdots,X_1^{(n)}(t))$ satisfying the $SDEs$ (\ref{edgesystem2}), in which $X^{(k)}_1$ is an $L^{(k)}$-diffusion reflected off $X^{(k-1)}_{1}$, has transition densities $\bar{s}_t(x,x')$.
\end{prop}
\begin{proof}
The key observation in this setting is the following relation between the transition kernels:
\begin{align*}
p_t^{(k)}(x,z)=e^{c_{k,n}t}\int_{z}^{r}\partial_xp_t^{(k+1)}(x,w)dw.
\end{align*}
This is immediate from (\ref{onesidedtransitionrelation}) since each diffusion process in this section is an honest Markov process.

Then, checking the parabolic equation with the correct spatial boundary conditions is as before. Now the $t=0$ boundary condition, again follows from the fact that all contributions from off diagonal terms in the determinant have at least one term vanishing uniformly in this new domain $(x_1 \ge \cdots \ge x_n) $.
\end{proof}
Via a simple integration, we obtain the following formulae for the distributions of the leftmost and rightmost particles in the Gelfand-Tsetlin pattern,
\begin{cor}
\begin{align*}
\mathbb{P}_{x^{(0)}}(X_n^{(n)}(t)\le z)&=\det\big(\mathcal{S}_t^{(i),i-j+1}(x^{(0)}_i,z)\big)_{i,j=1}^n ,\\
\mathbb{P}_{\bar{x}^{(0)}}(X_1^{(n)}(t)\ge z)&=\det\big(-\bar{\mathcal{S}}_t^{(i),i-j+1}(\bar{x}^{(0)}_i,z)\big)_{i,j=1}^n,
\end{align*}
where $x^{(0)}=(x^{(0)}_1\le \cdots \le x^{(0)}_n)$ and $\bar{x}^{(0)}=(\bar{x}^{(0)}_1\ge \cdots \ge \bar{x}^{(0)}_n)$.
\end{cor}
 For $p_t^{(k)}$ the heat kernel and $x^{(0)}=(0, \cdots ,0)$ this recovers a formula from \cite{Warren}. In the $BESQ(d)$ case and $t=1$ the above give expressions for the largest and smallest eigenvalues for the $LUE$ ensemble. We obtain the analogous expressions in the Jacobi case as $t\to \infty$ since the $JUE$ is the invariant measure of non-intersecting Jacobi processes.

\section{Well-posedness and transition densities for SDEs with reflection}

\subsection{Well-posedness of reflecting SDEs}\label{SectionWellposedness}
We will prove well-posedness (existence and uniqueness) for the systems of reflecting $SDEs$ (\ref{System1SDEs}), (\ref{System2}), (\ref{GelfandTsetlinSDEs}), (\ref{edgesystem1}) and (\ref{edgesystem2}) considered in this work. It will be more convenient, although essentially equivalent for our purposes, to consider reflecting $SDEs$ for $X$ in the time dependent domains (or between barriers) given by $Y$ i.e. in the form of (\ref{systemofreflectingLdiffusions}). More precisely we will consider $SDEs$ with reflection for a single particle $X$ in the time dependent domain $[\mathsf{Y}^-,\mathsf{Y}^+]$ where $\mathsf{Y}^-$ is the lower time dependent boundary and  $\mathsf{Y}^+$ is the upper time dependent boundary. This covers all the cases of interest to us by taking $\mathsf{Y}^-=Y_{i-1}$ and $\mathsf{Y}^+=Y_{i}$ with the possibility $\mathsf{Y}^-\equiv l$ and/or $\mathsf{Y}^+\equiv r$.

We will first obtain weak existence, for coefficients $\sigma(x)=\sqrt{2a(x)},b(x)$ continuous and of at most linear growth, the precise statement to found in Proposition \ref{WeakExistenceProp} below. We begin by recalling the definition and some properties of the Skorokhod problem in a time dependent domain. We will use the following notation, $\mathbb{R}_+=[0,\infty)$. Suppose we are given continuous functions $z,\mathsf{Y}^-,\mathsf{Y}^+ \in C\left(\mathbb{R}_+;\mathbb{R}\right)$ such that $\forall T\ge 0$,
\begin{align*}
 \underset{t\le T}{\inf}\left(\mathsf{Y}^+(t)-\mathsf{Y}^-(t)\right)>0,
\end{align*} 
a condition to be removed shortly by a stopping argument. We then say that the pair $(x,k)\in C\left(\mathbb{R}_+;\mathbb{R}\right) \times C\left(\mathbb{R}_+;\mathbb{R}\right)$ is a solution to the Skorokhod problem for $\left(z,\mathsf{Y}^-,\mathsf{Y}^+\right)$ if for every $t\ge 0$ we have $x(t)=z(t)+k(t)\in [\mathsf{Y}^-(t),\mathsf{Y}^+(t)]$ and $k(t)=k^-(t)-k^+(t)$ where $k^+$ and $k^-$ are non decreasing, in particular bounded variation functions, such that $\forall t \ge 0$ ,
\begin{align*}
\int_{0}^{t}\textbf{1}\left(z(s)>\mathsf{Y}^-(s)\right)dk^-(s)=0 \ \textnormal{ and }\ \int_{0}^{t}\textbf{1}\left(z(s)<\mathsf{Y}^+(s)\right)dk^+(s)=0.
\end{align*}
Observe that the constraining terms $k^+$ and $k^-$ only increase on the boundaries of the time dependent domain, namely at $\mathsf{Y}^+$ and $\mathsf{Y}^-$ respectively. Now, consider the \textit{solution} map denoted by $\mathcal{S}$,
\begin{align*}
\mathcal{S}:C\left(\mathbb{R}_+;\mathbb{R}\right) \times C\left(\mathbb{R}_+;\mathbb{R}\right)\times C\left(\mathbb{R}_+;\mathbb{R}\right) \to C\left(\mathbb{R}_+;\mathbb{R}\right) \times C\left(\mathbb{R}_+;\mathbb{R}\right)
\end{align*}
given by,
\begin{align*}
\mathcal{S}:\left(z,\mathsf{Y}^-,\mathsf{Y}^+\right)\mapsto \left(x,k\right).
\end{align*}
Then the key fact is that the map $\mathcal{S}$ is Lipschitz continuous in the supremum norm and there exists a unique solution to the Skorokhod problem, see for example Proposition 2.3 and Corollary 2.4 of \cite{SDER} (also Theorem 2.6 of \cite{RamananBurdzySkorokhod}). Below we will sometimes abuse notation and write $x=\mathcal{S}\left(z,\mathsf{Y}^-,\mathsf{Y}^+\right)$ just for the $x$-component of the solution $(x,k)$.

Now suppose $\sigma:\mathbb{R} \to \mathbb{R}$ and $b:\mathbb{R}\to \mathbb{R}$ are Lipschitz continuous functions. Then by a classical argument based on Picard iteration, see for example Theorem 3.3 of \cite{SDER}, we obtain that there exists a unique strong solution to the $SDER$ ($SDE$ with reflection) for $\mathsf{Y}^-(0) \le X(0) \le \mathsf{Y}^+(0)$,
\begin{align*}
X(t )=X(0)+\int_{0}^{t}\sigma\left(X(s)\right)d\beta(s)+\int_{0}^{t  }b\left(X(s)\right)ds+K^-(t )-K^+(t ),
\end{align*}
where $\beta$ is a standard Brownian motion and $\left(K^+(t);t \ge 0\right)$ and $\left(K^-(t );t \ge 0\right)$ are non decreasing processes that increase only when $X(t)=\mathsf{Y}^+(t)$ and $X(t)=\mathsf{Y}^-(t)$ respectively so that for all $t \ge0$ we have $X(t ) \in [\mathsf{Y}^-(t),\mathsf{Y}^+(t)]$. Here, by strong solution we mean that on the filtered probability space $\left(\Omega,\mathcal{F},\{\mathcal{F}_t\},\mathbb{P}\right)$ on which $\left(X,K,\beta\right)$ is defined, the process $(X,K)$ is adapted with respect to the filtration $\mathcal{F}_t^{\beta}$ generated by the Brownian motion $\beta$. Equivalently $(X,K)$ where $K=K^+-K^-$ solves the Skorokhod problem for $(z,\mathsf{Y}^-,\mathsf{Y}^+)$ where,
\begin{align*}
z\left(\cdot\right)\overset{def}{=}X(0)+\int_{0}^{\cdot}\sigma\left(X(s)\right)d\beta(s)+\int_{0}^{\cdot}b\left(X(s)\right)ds.
\end{align*}
We write $\mathfrak{s}_L^R$ for the corresponding measurable solution map on path space, namely so that $X=\mathfrak{s}_L^R\left(\beta;\mathsf{Y}^-,\mathsf{Y}^+\right)$.

Now, suppose $\sigma:\mathbb{R} \to \mathbb{R} $ and $b:\mathbb{R}\to \mathbb{R}$ are merely continuous and of at most linear growth, namely:
\begin{align*}
|\sigma(x)|,|b(x)| \le C\left(1+|x|\right),
\end{align*}
for some constant $C$. We will abbreviate this assumption by $(\textbf{CLG})$. Then, we can still obtain weak existence using the following rather standard argument. Take $\sigma^{(n)}:\mathbb{R} \to \mathbb{R} $ and $b^{(n)}:\mathbb{R}\to \mathbb{R}$ to be Lipschitz, converging uniformly to $\sigma$ and $b$ and satisfying a uniform linear growth condition. More precisely:
\begin{align}
&\sigma^{(n)}\overset{\textnormal{unif}}{\longrightarrow}\sigma, \ b^{(n)}\overset{\textnormal{unif}}{\longrightarrow}b\nonumber,\\
&|\sigma^{(n)}(x)|,|b^{(n)}(x)|\le \tilde{C}\left(1+|x|\right)\label{uniformlineargrowth},
\end{align}
for some constant $\tilde{C}$ that is independent of $n$. For example, we could take the mollification $\sigma^{(n)}=\phi_n*\sigma$, with $\phi_n(x)=n\phi(nx)$ where $\phi$ is a smooth bump function: $\phi \in C^{\infty}, \phi \ge 0, \int \phi=1$ and $\textnormal{ supp}(\phi) \subset [-1,1]$. Then, if $|\sigma(x)| \le C\left(1+|x|\right)$ we easily get $|(\phi_n*\sigma)(x)|\le \left(2+|x|\right)$ uniformly in $n$. Let $\left(X^{(n)},K^{(n)}\right)$ be the corresponding strong solution to the $SDER$ above with coefficients $\sigma^{(n)}$ and $b^{(n)}$. Then the laws of,
\begin{align*}
X(0)+\int_{0}^{\cdot}\sigma^{(n)}\left(X^{(n)}(s)\right)d\beta(s)+\int_{0}^{\cdot}b^{(n)}\left(X^{(n)}(s)\right)ds,
\end{align*}
are easily seen to be tight by applying Aldous' tightness criterion (see for example Chapter 16 of \cite{Kallenberg} or Chapter 3 of \cite{EthierKurtz}) using the uniformity in $n$ of the linear growth condition (\ref{uniformlineargrowth}). Hence, from the Lipschitz continuity of $\mathcal{S}$ we obtain that the laws of $\left(X^{(n)},K^{(n)}\right)$ are tight as well. 

Thus, we can choose a subsequence $\left(n_i;i \ge 1\right)$ such that the laws of $\left(X^{(n_i)},K^{(n_i)}\right)$ converge weakly to some $\left(X,K\right)$. Using the Skorokhod representation theorem we can upgrade this to joint almost sure convergence on a new probability space $\left(\tilde{\Omega},\tilde{\mathcal{F}},\{\tilde{\mathcal{F}}_t\},\tilde{\mathbb{P}}\right)$. More precisely, we can define processes $\left(\tilde{X}^{(i)},\tilde{K}^{(i)}\right)_{i\ge 1},\left(\tilde{X},\tilde{K}\right)$ on $\left(\tilde{\Omega},\tilde{\mathcal{F}},\{\tilde{\mathcal{F}}_t\},\tilde{\mathbb{P}}\right)$ so that:
\begin{align*}
\left(\tilde{X}^{(i)},\tilde{K}^{(i)}\right)\overset{\textnormal{d}}{=}\left(X^{(n_i)},K^{(n_i)}\right), \ \left(\tilde{X},\tilde{K}\right)\overset{\textnormal{d}}{=}\left(X,K\right), \ \left(\tilde{X}^{(i)},\tilde{K}^{(i)}\right)\overset{\textnormal{a.s.}}{\longrightarrow}(\tilde{X},\tilde{K}).
\end{align*}
Now, the stochastic processes:
\begin{align*}
M_n(t)=\tilde{X}^{(n)}(t)-\tilde{X}^{(n)}(0)-\int_{0}^{t}b^{(n)}\left(\tilde{X}^{(n)}(s)\right)ds-\left(\tilde{K}^{(n)}\right)^-(t)+\left(\tilde{K}^{(n)}\right)^+(t)
\end{align*}
are martingales with quadratic variation:
\begin{align*}
\langle M_n, M_n\rangle(t)=\int_{0}^{t}\left(\sigma^{(n)}\left(\tilde{X}^{(n)}(s)\right)\right)^2ds.
\end{align*}
By the following convergences:
\begin{align*}
\sigma^{(n)}\overset{\textnormal{unif}}{\longrightarrow}\sigma, \ b^{(n)}\overset{\textnormal{unif}}{\longrightarrow}b, \ \left(\tilde{X}^{(i)},\tilde{K}^{(i)}\right)\overset{\textnormal{a.s.}}{\longrightarrow}(\tilde{X},\tilde{K})
\end{align*}
we obtain that $M_n \overset{\textnormal{a.s.}}{\longrightarrow}M$ where,
\begin{align*}
M(t)=\tilde{X}(t)-\tilde{X}(0)-\int_{0}^{t}b\left(\tilde{X}(s)\right)ds-\tilde{K}^-(t)+\tilde{K}^+(t)
\end{align*}
is a martingale with quadratic variation given by:
\begin{align*}
\langle M, M\rangle(t)=\int_{0}^{t}\sigma^2\left(\tilde{X}(s)\right)ds.
\end{align*}
Then, by the martingale representation theorem there exists a standard Brownian motion $\tilde{\beta}$, that is defined on a possibly enlarged probability space, so that $M(t)=\int_{0}^{t}\sigma\left(\tilde{X}(s)\right)d\tilde{\beta}(s)$ and thus:
\begin{align*}
\tilde{X}(t)=\tilde{X}(0)+\int_{0}^{t}\sigma\left(\tilde{X}(s)\right)d\tilde{\beta}(s)+\int_{0}^{t }b\left(\tilde{X}(s)\right)ds+\tilde{K}^-(t )-\tilde{K}^+(t),
\end{align*}
where again the non decreasing processes $\left(\tilde{K}^+(t);t \ge 0\right)$ and $\left(\tilde{K}^-(t);t \ge 0\right)$ increase only when $\tilde{X}(t)=\mathsf{Y}^+(t)$ and $\tilde{X}(t)=\mathsf{Y}^-(t)$ respectively so that $\tilde{X}(t) \in [\mathsf{Y}^-(t),\mathsf{Y}^+(t)] \ \forall t \ge0$. Hence, we have obtained the existence of a weak solution to the $SDER$ for $\sigma$ and $b$ continuous and of at most linear growth.

We now remove the condition that $\mathsf{Y}^-,\mathsf{Y}^+$ never collide by stopping the process at the first time $\tau=\inf\{t\ge0:\mathsf{Y}^-(t)=\mathsf{Y}^+(t)\}$ that they do. First we note that, there exists an extension to the Skorokhod problem and to $SDER$, allowing for reflecting barriers $\mathsf{Y}^-,\mathsf{Y}^+$ that come together, see \cite{RamananBurdzySkorokhod}, \cite{SDER} for the detailed definition. Both results used in the previous argument, namely the Lipschitz continuity of the solution map, which we still denote by $\mathcal{S}$, and existence and uniqueness of strong solutions to $SDER$ extend to this setting, see e.g. Theorem 2.6, also Corollary 2.4 and Theorem 3.3 in \cite{SDER}. The difference of the extended problem to the classical one described at the beginning, being that $k=k^--k^+$ is allowed to have infinite variation. However, as proven in Proposition 2.3 and Corollary 2.4 in \cite{RamananBurdzySkorokhod} (see also Remark 2.2 in \cite{SDER}) the unique solution to the extended Skorokhod problem coincides with the one of the classical one in $[0,T]$ while $\underset{t\le T}{\inf}\left(\mathsf{Y}^+(t)-\mathsf{Y}^-(t)\right)>0$. Thus, by the previous considerations, for any $T<\tau$, we still have a weak solution to the $SDER$ above, with bounded variation local terms $K$; the final statement more precisely given as:

\begin{prop}\label{WeakExistenceProp}
Assume $\mathsf{Y}^-,\mathsf{Y}^+$ are continuous functions such that $\mathsf{Y}^-(t)\le \mathsf{Y}^+(t), \forall t \ge 0$ and let $\tau=\inf\{t\ge0:\mathsf{Y}^-(t)=\mathsf{Y}^+(t)\}$. Assume (\textbf{CLG}), namely that $\sigma(\cdot),b(\cdot)$ are continuous functions satisfying an at most linear growth condition, for some positive constant $C$:
\begin{align*}
|\sigma(x)|,|b(x)| \le C\left(1+|x|\right).
\end{align*}
Then, there exists a filtered probability space $\left(\Omega,\mathcal{F},\{\mathcal{F}_t\},\mathbb{P}\right)$ on which firstly an adapted Brownian motion $\beta$ is defined (not necessarily generating the filtration). Moreover, for $\mathsf{Y}^-(0)\le X(0) \le \mathsf{Y}^+(0)$ the adapted process $(X,K)$ satisfies:
\begin{align}\label{prototypeSDER}
X(t\wedge \tau)=X(0)+\int_{0}^{t\wedge  \tau}\sigma\left(X(s)\right)d\beta(s)+\int_{0}^{t \wedge  \tau}b\left(X(s)\right)ds+K^-(t \wedge  \tau)-K^+(t \wedge  \tau),
\end{align}
such that for all $t \ge0$ we have $X(t\wedge \tau) \in [\mathsf{Y}^-(t\wedge \tau),\mathsf{Y}^+(t \wedge \tau)]$ and for any $T< \tau$ the non decreasing processes $\left(K^+(t);t \le T\right)$ and $\left(K^-(t );t \le T\right)$ increase only when $X(t)=\mathsf{Y}^+(t)$ and $X(t)=\mathsf{Y}^-(t)$ respectively.
\end{prop}

We will now be concerned with pathwise uniqueness. Due to the intrinsic one-dimensionality of the problem we can fortunately apply a simple Yamada-Watanabe type argument. For the convenience of the reader we now recall assumption $(\mathbf{YW})$, defined in Section 2: Let $I$ be an interval with endpoints $l<r$ and suppose $\rho$ is a non-decreasing function from $(0,\infty)$ to itself such that $\int_{0^+}^{}\frac{dx}{\rho(x)}=\infty$. Consider, the following condition on functions $a:I\to \mathbb{R}_+$ and $b: I\to \mathbb{R}$, where we implicitly assume that $a$ and $b$ initially defined in $I^\circ$ can be extended continuously to the boundary points $l$ and $r$ (in case these are finite),
\begin{align*}
&|\sqrt {a(x)}-\sqrt {a(y)}|^2 \le \rho(|x-y|),\\
&|b(x)-b(y)|\le C|x-y|.
\end{align*}
Moreover, we assume that $\sqrt {a(\cdot)}$ is of at most linear growth. Note that, for $b(\cdot)$ this is immediate by Lipschitz continuity. 

Also, observe that since $\rho$ is continuous at $0$ with $\rho(0)=0$ (the assumption on $\rho$ implies this) we get that $\sqrt {a(\cdot)}$ is continuous. Thus, $(\mathbf{YW})$ implies (\textbf{CLG}) and in particular the existence result above applies under $(\mathbf{YW})$. We are now ready to state and prove our well-posedness result.

\begin{prop}\label{wellposedness}
Under the $(\mathbf{YW})$ assumption the $SDER$ (\ref{prototypeSDER}) with $(\sigma,b)=(\sqrt{2a},b)$ has a pathwise unique solution.
\end{prop}
\begin{proof}
Suppose that $X$ and $\tilde{X}$ are two solutions of (\ref{prototypeSDER}) with respect to the same noise. Then the argument given at Chapter IX Corollary 3.4 of \cite{RevuzYor} shows that $L^0(X_i-\tilde{X}_i)=0$ where for a semimartingale $Z$, $L^{a}(Z)$ denotes its semimartingale local time at $a$ (see for example Section 1 Chapter VI of \cite{RevuzYor}). Hence by Tanaka's formula we get,\\

$|X(t\wedge \tau)-\tilde{X}(t\wedge \tau)|=\int_{0}^{t\wedge\tau}\textnormal{sgn}(X(s)-\tilde{X}(s))d(X(s)-\tilde{X}(s))$\\$=\int_{0}^{t \wedge \tau}\textnormal{sgn}(X(s)-\tilde{X}(s))\left(\sqrt{ 2a(X(s))}-\sqrt{ 2a(\tilde{X}(s))}\right)d\beta(s)$\\
$+\int_{0}^{t\wedge \tau}\textnormal{sgn}(X(s)-\tilde{X}(s))(b(X(s))-b (\tilde{X}(s)))ds$\\
$-\int_{0}^{t\wedge \tau}\textnormal{sgn}(X(s)-\tilde{X}(s))d(K^+(s)-\tilde{K}^+(s))+\int_{0}^{t\wedge \tau}\textnormal{sgn}(X(s)-\tilde{X}(s))d(K^-(s)-\tilde{K}^-(s))$.\\

Note that $\mathsf{Y}^- \le X,\tilde{X}\le \mathsf{Y}^+$, $dK^+$ is supported on $\{t:X(t)=\mathsf{Y}^+(t)\}$ and $d\tilde{K}^+$ is supported on $\{t:\tilde{X}(t)=\mathsf{Y}^+(t)\}$. So if $\tilde{X}<X\le \mathsf{Y}^+$ then $dK^+-d\tilde{K}^+\ge 0$ and if $X<\tilde{X}\le \mathsf{Y}^+$ then $dK^+-d\tilde{K}^+\le 0$. Hence $\int_{0}^{t\wedge \tau}\textnormal{sgn}(X(s)-\tilde{X}(s))d(K^+(s)-\tilde{K}^+(s))\ge 0$. With similar considerations $\int_{0}^{t\wedge \tau}\textnormal{sgn}(X(s)-\tilde{X}(s))d(K^-(s)-\tilde{K}^-(s))\le 0$. Taking expectations we obtain,
\begin{align*}
\mathbb{E}[|X(t\wedge \tau)-\tilde{X}(t\wedge\tau)|]\le \mathbb{E}\left[\int_{0}^{t\wedge \tau}\textnormal{sgn}(X(s)-\tilde{X}(s))(b(X(s))-b (\tilde{X}(s)))ds\right]\\ \le C\int_{0}^{t\wedge \tau} \mathbb{E} [|X(s)-\tilde{X}(s)|]ds.
\end{align*}
The statement of the proposition then follows from Gronwall's lemma.
\end{proof}

Under the pathwise uniqueness obtained in Proposition \ref{wellposedness} above, if the evolution $\left(\mathsf{Y}^-(t\wedge \tau),\mathsf{Y}^+(t\wedge \tau);t\ge 0\right)$ is Markovian, then standard arguments (see for example Section 1 of Chapter IX of \cite{RevuzYor}) imply that $\left(\mathsf{Y}^-(t\wedge \tau),\mathsf{Y}^+(t\wedge \tau),X(t \wedge \tau);t\ge 0\right)$ is Markov as well. Moreover, under this $(\mathbf{YW})$ condition we still have the solution map $X=\mathfrak{s}_L^R\left(\beta;\mathsf{Y}^-,\mathsf{Y}^+\right)$.

The reader should note that Proposition \ref*{wellposedness} covers in particular \textbf{all} the cases of Brownian motions, Ornstein-Uhlenbeck, $BESQ(d)$ , $Lag(\alpha)$ and  $Jac(\beta,\gamma)$ diffusions considered in the Applications and Examples section.

\subsection{Transition densities for SDER}\label{SectionTransitionDensities}
The aim of this section is to prove under some conditions that $q_t^{n,n+1}$ and $q_t^{n,n}$ form the transition kernels for the two-level systems of $SDEs$ (\ref{System1SDEs}) and (\ref*{System2}) in $W^{n,n+1}$ and $W^{n,n}$ respectively. For the sake of exposition we shall mainly restrict our attention to (\ref*{System1SDEs}). In the sequel, $\tau$ will denote the stopping time $T^{n,n+1}$ (or $T^{n,n}$ respectively).

Throughout this section we assume $(\textbf{R})$ and $(\textbf{BC+})$ hold for the $L$-diffusion and $(\textbf{YW})$ holds for both the $L$ and $\hat{L}$ diffusions. In particular, there exists a Markov semimartingale  $(X,Y)$ satisfying equation (\ref{System1SDEs}) (or respectively (\ref*{System2})).

To begin with we make a few simple but important observations. First, note that if the $L$-diffusion does not hit $l$ (i.e. $l$ is natural or entrance), then $X_{1}$ doesn't hit $l$ either before being driven to $l$ by $Y_1$ (in case $l$ is exit for $\hat{L}$). Similarly, it is rather obvious, since the particles are ordered, that in case $l$ is regular reflecting for the $L$-diffusion the time spent at $l$ up to time $\tau$ by the $SDEs$ (\ref{System1SDEs}) is equal to the time spent by $X_1$ at $l$. This is in turn equal to the time spent at $l$ by the excursions of $X_1$ between collisions with $Y_1$ (and before $\tau$) during which the evolution of $X_1$ coincides with the unconstrained $L$-diffusion which spends zero Lebesgue time at $l$ (e.g. see Chapter 2 paragraph 7 in \cite{BorodinSalminen}). Hence the system of reflecting $SDEs$ (\ref{System1SDEs}) spends zero Lebesgue time at either $l$ or $r$ up to time $\tau$. Since in addition to this, the noise driving the $SDEs$ is uncorrelated and the diffusion coefficients do not vanish in $I^\circ$ we get that,
 \begin{align}\label{boundaryLebesgue}
 \int_{0}^{\tau}\mathbf{1}_{\partial W^{n,n+1}(I)}\left(X(t),Y(t)\right)dt=0 \ \ \textnormal{  a.s.} \ .
 \end{align}
 
 We can now in fact relate the constraining finite variation terms $K$ to the semimartingale local times of the gaps between particles (although this will not be essential in what follows). Using the observation (\ref{boundaryLebesgue}) above and Exercise 1.16 ($3^\circ$) of Chapter VI of \cite{RevuzYor}, which states that for a positive semimartingale $Z=M+V \ge 0$ (where $M$ is the martingale part) its local time at $0$ is equal to $2\int_{0}^{\cdot}1\left(Z_s=0\right)dVs$, we get that for the $SDEs$ (\ref{System1SDEs}) the semimartingale local time of $Y_i-X_i$ at 0 up to time $\tau$ is,
\begin{align*}
2\int_{0}^{t\wedge \tau}\textbf{1}(Y_i(s)=X_i(s))dK_i^+(s)=2K_i^+(t\wedge \tau),
\end{align*}
and similarly the semimartingale local time of $X_{i+1}-Y_i$ at $0$ up to $\tau$ is,
 \begin{align*}
 2\int_{0}^{t\wedge \tau}\textbf{1}(X_{i+1}(s)=Y_i(s))dK_{i+1}^-(s)=2K_{i+1}^-(t\wedge \tau).
 \end{align*}
Now, we state a lemma corresponding to the \textit{time 0} boundary condition.

\begin{lem}\label{time0lemma}
For any $f:{W^{n,n+1}(I^\circ)}\to \mathbb{R}$ continuous with compact support we have,
\begin{align*}
\lim_{t\to 0}\int_{W^{n,n+1}(I^\circ)}^{}q_t^{n,n+1}((x,y),(x',y'))f(x',y')dx'dy'=f(x,y).
\end{align*}
\end{lem}
\begin{proof}
This follows as in the proof of Lemma 1 of \cite{Warren}. See also the beginning of the proof of Proposition \ref{PropositionEdgeParticle}.
\end{proof}

We are now ready to prove the following result on the transition densities.

 \begin{prop}\label{transititiondensities1}
Assume $(\textbf{R})$ and $(\textbf{BC+})$ hold for the $L$-diffusion and $(\textbf{YW})$ holds for both the $L$ and $\hat{L}$ diffusions. Moreover, assume that $l$ and $r$ are either natural or entrance for the $L$-diffusion. Then $q_t^{n,n+1}$ form the transition densities for the system of $SDEs$ (\ref{System1SDEs}).
 \end{prop}
\begin{proof}
Let $\textbf{Q}^{n,n+1}_{x,y}$ denote the law of the process $(X_1,Y_1,\cdots,Y_n,X_{n+1})$ satisfying the system of $SDEs$ (\ref{System1SDEs}) and starting from $(x,y)$. Define for $f$ continuous with compact support,
\begin{align*}
F^{n,n+1}(t,(x,y))&=\int_{W^{n,n+1}(I^\circ)}^{}q_t^{n,n+1}((x,y),(x',y'))f(x',y')dx'dy'.
\end{align*}
Our goal is to prove that for fixed $T>0$,
\begin{align}
F^{n,n+1}(T,(x,y))&=\textbf{Q}^{n,n+1}_{x,y}\big[f(X(T),Y(T))\textbf{1}(T<\tau)\big]\label{goaltransition1}.
\end{align}
The result then follows since from observation (\ref{boundaryLebesgue}) the only part of the distribution of  $(X(T),Y(T))$ that charges the boundary corresponds to the event $\{T\ge \tau\}$.

In what follows we shall slightly abuse notation and use the same notation for both the scalar entries and the matrices that come into the definition of $q_t^{n,n+1}$. First, note the following with $x,y \in I^\circ$,
\begin{align*}
\partial_t A_t(x,x')&=\mathcal{D}_m^{x}\mathcal{D}_s^{x}A_t(x,x')\ \ ,\ 
\ \ \partial_tB_t(x,y')=\mathcal{D}_m^{x}\mathcal{D}_s^{x}B_t(x,y'),\\
\partial_t C_t(y,x')&=\mathcal{D}_{\hat{m}}^{y}\mathcal{D}_{\hat{s}}^{y}C_t(y,x')\ \ ,\ \ 
\partial_t D_t(y,y')=\mathcal{D}_{\hat{m}}^{y}\mathcal{D}_{\hat{s}}^{y}D_t(y,y').
\end{align*}
To see the equation for $ C_t(y,x')$ note that since $\mathcal{D}_{\hat{m}}=\mathcal{D}_s$ and $\mathcal{D}_{\hat{s}}=\mathcal{D}_m$ we have, 
\begin{align*}
\partial_t C_t(y,x')=-\mathcal{D}_{s}^{y}\partial_tp_t(y,x')=-\mathcal{D}_{s}^{y}\mathcal{D}_m^{y}\mathcal{D}_s^{y}p_t(y,x')=-\mathcal{D}_{\hat{m}}^{y}\mathcal{D}_{\hat{s}}^{y}\mathcal{D}_s^{y}p_t(y,x')=\mathcal{D}_{\hat{m}}^{y}\mathcal{D}_{\hat{s}}^{y}C_t(y,x').
\end{align*}
Hence, for fixed $(x',y')\in \mathring{W}^{n,n+1}(I^\circ)$ we have,
\begin{align*}
\partial_t q_t^{n,n+1}((x,y),(x',y'))=\bigg(\sum_{i=1}^{n+1}\mathcal{D}_m^{x_i}\mathcal{D}_s^{x_i}+\sum_{i=i}^{n}\mathcal{D}_{\hat{m}}^{y_i}\mathcal{D}_{\hat{s}}^{y_i} \bigg) \ q_t^{n,n+1}((x,y),(x',y')),\\ \text{in} \ (0,\infty)\times \mathring{W}^{n,n+1}(I^\circ).
\end{align*}
Now, by definition of the entries $A_t,B_t,C_t,D_t$ we have for $x,y\in I^\circ$, 
\begin{align*}
\partial_x A_t(x,x')|_{x=y}=-\hat{m}(y)C_t(y,x'),\\
\partial_x B_t(x,y')|_{x=y}=-\hat{m}(y)D_t(y,y').
\end{align*}
Hence for fixed $(x',y')\in W^{n,n+1}(I^\circ)$ by differentiating the determinant and since two rows are equal up to multiplication by a constant we obtain,
\begin{align*}
\partial_{x_i}q_t^{n,n+1}((x,y),(x',y'))|_{x_i=y_i}&=0 \ , \ \partial_{x_i}q_t^{n,n+1}((x,y),(x',y'))|_{x_i=y_{i-1}}=0.
\end{align*}
The Dirichlet boundary conditions for $y_i=y_{i+1}$ are immediate since again two rows of the determinant are equal. Furthermore, in case $l$ or $r$ are entrance boundaries for the $L$-diffusion the Dirichlet boundary conditions for $y_1=l$ and $y_n=r$ follow from the fact that (in the limit as $y \to l,r$),
\begin{align*}
D_t(y,y')|_{y=l,r}=0, C_t(y,x')|_{y=l,r}=\mathcal{D}_s^{x}A_t(x,x')|_{x=l,r}=0.
\end{align*} 

Fix $T,\epsilon>0$. Applying Ito's formula we obtain that for each $(x',y')$ the process,
\begin{align*}
\left(\mathfrak{Q}_t(x',y'):t\in [0,T]\right)=\left(q^{n,n+1}_{T+\epsilon-t}\left((X(t),Y(t)),(x',y')\right):t\in [0,T]\right),
\end{align*}
is a local martingale. Now consider a sequence of compact intervals $J_k$ exhausting $I$ as $k\to \infty$ and write $\tau_k$ for $\inf\{t:(X(t),Y(t)) \notin J_k \}$. Note that $\textbf{1}(T<\tau\wedge \tau_k)\to \textbf{1}(T<\tau)$ as $k\to \infty$ by our boundary assumptions, more precisely by making use of the observation that $X$ does not hit $l$ or $r$ before $Y$ does. Using the optional stopping theorem (since the stopped process $\left(\mathfrak{Q}^{\tau_k}_t(x',y'):t\in [0,T]\right)$ is bounded and hence a true martingale) and then the monotone convergence theorem we obtain,
 \begin{align*}
 q_{T+\epsilon}^{n,n+1}((x,y),(x',y'))=\textbf{Q}^{n,n+1}_{x,y}\big[q_{\epsilon}^{n,n+1}((X(T),Y(T)),(x',y'))\textbf{1}(T<\tau)\big].
\end{align*}
Now multiplying by $f$ continuous with compact support, integrating with respect to $(x',y')$ and using Fubini's theorem to exchange expectation and integral we obtain,
\begin{align*}
F^{n,n+1}(T+\epsilon,(x,y))=\textbf{Q}^{n,n+1}_{x,y}\big[F^{n,n+1}(\epsilon,(X(T),Y(T))\textbf{1}(T<\tau)\big].
\end{align*}
By Lemma \ref{time0lemma}, we can let $\epsilon \downarrow 0$ to conclude,
\begin{align*}
F^{n,n+1}(T,(x,y))&=\textbf{Q}^{n,n+1}_{x,y}\big[f(X(T),Y(T))\textbf{1}(T<\tau)\big].
\end{align*}
The proposition is proven.
\end{proof}

Completely analogous arguments prove the following:
\begin{prop}\label{transititiondensities2}
Assume $(\textbf{R})$ and $(\textbf{BC+})$ hold for the $L$-diffusion and $(\textbf{YW})$ holds for both the $L$ and $\hat{L}$ diffusions. Moreover, assume that $l$ is either natural or exit and $r$ is either natural or entrance for the $L$-diffusion. Then $q_t^{n,n}$ form the transition densities for the system of $SDEs$ (\ref{System2}).
 \end{prop}
 
We note here that Propositions \ref{transititiondensities1} and \ref{transititiondensities2} apply in particular to the cases of Brownian motions with drifts, Ornstein-Uhlenbeck, $BESQ(d)$ for $d\ge2$, $Lag(\alpha)$ for $\alpha \ge 2$ and  $Jac(\beta,\gamma)$ for $\beta,\gamma \ge 1$ considered in the Applications and Examples section.

In the case $l$ and/or $r$ are regular reflecting boundary points we have the following proposition. This is where the non-degeneracy and regularity at the boundary in assumption $(\textbf{BC+})$ is used. This is technical but quite convenient since it allows for a rather streamlined rigorous argument. It presumably can be removed.

 \begin{prop}\label{transitiondensitiesreflecting}
Assume $(\textbf{R})$ and $(\textbf{BC+})$ hold for the $L$-diffusion and $(\textbf{YW})$ holds for both the $L$ and $\hat{L}$ diffusions. Moreover, assume that $l$ and/or $r$ are regular reflecting for the $L$-diffusion. Then $q_t^{n,n+1}$ form the transition densities for the system of $SDEs$ (\ref{System1SDEs}). 
\end{prop}

\begin{proof}
The strategy is the same as in Proposition \ref{transititiondensities1} above. We give the proof in the case that both $l$ and $r$ are regular reflecting for the $L$-diffusion (the other cases are analogous). First, recall that $(\textbf{BC+})$ in this case requires that $\underset{x \to l,r}{\lim} a(x)>0$ and that the limits $\underset{x \to l,r}{\lim} b(x)$, $\underset{x \to l,r}{\lim} \left(a'(x)-b(x)\right)$ exist and are finite.

Now, note that by the non-degeneracy condition $\underset{x \to l,r}{\lim}a(x)>0$ and since $\underset{x \to l,r}{\lim}b(x)$ is finite we thus obtain $\underset{x \to l,r}{\lim}s'(x)>0$. 

So for $x'\in I^{\circ}$ the relations,
\begin{align*}
\underset{x \to l,r}{\lim}\mathcal{D}^x_sA_t(x,x')=0 \textnormal{ and } \underset{x \to l,r}{\lim}\mathcal{D}^x_sB_t(x,x')=0,
\end{align*}
actually imply that for $x' \in I^{\circ}$,
\begin{align}\label{Neumannboundary}
\underset{x \to l,r}{\lim}\partial_xA_t(x,x')=0 \textnormal{ and } \underset{x \to l,r}{\lim}\partial_xB_t(x,x')=0.
\end{align}
Moreover, by rearranging the backwards equations we have for fixed $y \in I^{\circ}$ that the functions,
\begin{align*}
(t,x)\mapsto \partial_x^{2}p_t(x,y)&=\frac{\partial_tp_t(x,y)-b(x)\partial_xp_t(x,y)}{a(x)},\\
(t,x)\mapsto \partial_x^{2}\mathcal{D}^x_sp_t(x,y)&=\frac{\partial_t\mathcal{D}_s^xp_t(x,y)-\left(a'(x)-b(x)\right)\partial_x\mathcal{D}_s^xp_t(x,y)}{a(x)},\\
&=\frac{\partial_t\mathcal{D}_s^xp_t(x,y)-\left(a'(x)-b(x)\right)m(x)\partial_tp_t(x,y)}{a(x)},
\end{align*}
and more generally for $n \ge 0$ and fixed $y\in I^{\circ}$,
\begin{align*}
(t,x)\mapsto \partial_t^n\partial_x^{2}\mathcal{D}^x_sp_t(x,y)=\frac{\partial_t^{n+1}\mathcal{D}_s^xp_t(x,y)-\left(a'(x)-b(x)\right)m(x)\partial^{n+1}_tp_t(x,y)}{a(x)},
\end{align*}
can be extended continuously to $(0,\infty) \times [l,r]$ (note the closed interval $[l,r]$). This is because every function on the right hand side can be extended by the assumptions of proposition and the fact that for $y\in I^{\circ} $, $\partial_t^np_t(\cdot,y) \in Dom(L)$ (see Theorem 4.3 of \cite{McKean} for example). Thus by Whitney's extension theorem, essentially a clever reflection argument in this case (see Section 3 of \cite{ExtensionWhitney} for example), $q_t^{n,n+1}((x,y),(x',y'))$ can be extended as a $C^{1,2}$ function in $(t,(x,y))$ to the whole space. We can hence apply Ito's formula, and it is important to observe that the finite variation terms $dK^l$ and $dK^r$ at $l$ and $r$ respectively (corresponding to $X_1$ and $X_{n+1}$) vanish by the Neumann boundary conditions (\ref{Neumannboundary}), from which we deduce as before that for fixed $T>0$,
\begin{align*}
 q_{T+\epsilon}^{n,n+1}((x,y),(x',y'))=\textbf{Q}^{n,n+1}_{x,y}\big[q_{\epsilon}^{n,n+1}((X(T),Y(T)),(x',y'))\textbf{1}(T<\tau)\big].
\end{align*}
The conclusion then follows as in Proposition \ref{transititiondensities1}.
\end{proof}

Completely analogous arguments give the following:

\begin{prop}\label{transititiondensities2reflecting}
Assume $(\textbf{R})$ and $(\textbf{BC+})$ hold for the $L$-diffusion and $(\textbf{YW})$ holds for both the $L$ and $\hat{L}$ diffusions. Moreover, assume that $l$ is regular absorbing and/or $r$ is regular reflecting for the $L$-diffusion. Then $q_t^{n,n}$ form the transition densities for the system of $SDEs$ (\ref{System2}).
 \end{prop}

 These propositions cover in particular the cases of Brownian motions in the half line and in an interval considered in Sections 3.2 and 3.3 respectively.

\section{Appendix}\label{Appendix}
We collect here the proofs of some of the facts regarding conjugate diffusions that were stated and used in previous sections.

We first give the derivation of the table on the boundary behaviour of a diffusion and its conjugate. Keeping with the notation of Section 2 consider the following quantities with $x\in I^\circ$ arbitrary,
\begin{align*}
& N(l)=\int_{(l^+,x]}^{}(s(x)-s(y))M(dy)=\int_{(l^+,x]}^{}(s(x)-s(y))m(y)dy ,\\
&\Sigma(l)=\int_{(l^+,x]}^{}(M(x)-M(y))s(dy)=\int_{(l^+,x]}^{}(M(x)-M(y))s'(y)dy. 
\end{align*} 
We then have the following classification of the boundary behaviour at $l$ (see e.g. \cite{EthierKurtz}):
\begin{itemize}
\item $l$ is an entrance boundary iff $N(l)<\infty, \Sigma(l)=\infty$.
\item $l$ is a exit boundary iff $N(l)=\infty, \Sigma(l)
<\infty$.
\item $l$ is a natural boundary iff $N(l)=\infty, \Sigma(l)=\infty$.
\item $l$ is a regular boundary iff $N(l)<\infty, \Sigma(l)<\infty$.
\end{itemize}
From the relations $\hat{s}'(x)=m(x)$ and $\hat{m}(x)=s'(x)$ we obtain the following,
\begin{align*}
&\hat{N}(l)=\int_{(l^+,x]}^{}(\hat{s}(x)-\hat{s}(y))\hat{m}(y)dy=\Sigma(l),\\
&\hat{\Sigma}(l)=\int_{(l^+,x]}^{}(\hat{M}(x)-\hat{M}(y))\hat{s}'(y)dy=N(l).
\end{align*}

These relations immediately give us the table on boundary behaviour, namely: If $l$ is an entrance boundary for $X$, then it is exit for $\hat{X}$ and vice versa. If $l$ is natural for $X$, then so it is for its conjugate. If $l$ is regular for $X$, then so it is for its conjugate. In this instance as already stated in Section 2 we define the conjugate diffusion $\hat{X}$ to have boundary behaviour dual to that of $X$, namely if $l$ is reflecting for $X$ then it is absorbing for $\hat{X}$ and vice versa.

\begin{proof}[Proof of Lemma 2.1]
There is a total number of $5^2$ boundary behaviours ($5$ at $l$ and $5$ at $r$) for the $L$-diffusion (the boundary behaviour of $\hat{L}$ is completely determined from $L$ as explained above) however since the boundary conditions for an entrance and regular reflecting ($\mathcal{D}_sv=0$) and similarly for an exit and regular absorbing boundary ($\mathcal{D}_m\mathcal{D}_sv=0$) are the same we can pair them to reduce to $3^2$ cases $(\mathsf{b.c.}(l),\mathsf{b.c.}(r))$ abbreviated as follows:
\begin{align*}
(nat,nat),(ref,ref),(abs,abs),(nat,abs),(ref,abs),(abs,ref),(abs,nat),(nat,ref),(ref,nat).
\end{align*}
We now make some further reductions. Note that for $x,y \in I^{\circ}$,
\begin{align*}
\mathsf{P}_t \textbf{1}_{[l,y]}(x)=\mathsf{\hat{P}}_t \textbf{1}_{[x,r]}(y) \iff \mathsf{P}_t \textbf{1}_{[y,r]}(x)=\mathsf{\hat{P}}_t \textbf{1}_{[l,x]}(y).
\end{align*}
After swapping $x \leftrightarrow y$ this is equivalent to,
\begin{align*}
\mathsf{\hat{P}}_t \textbf{1}_{[l,y]}(x)=\mathsf{P}_t \textbf{1}_{[x,r]}(y).
\end{align*}
So we have a bijection that swaps boundary conditions with their duals $(\mathsf{b.c.}(l),\mathsf{b.c.}(r))\leftrightarrow(\widehat{\mathsf{b.c.}(l)},\widehat{\mathsf{b.c.}(r)})$. Moreover, if $\mathfrak{h}:(l,r)\to (l,r)$ is any homeomorphism such that $\mathfrak{h}(l)=r,\mathfrak{h}(r)=l$ and writing $\mathsf{H}_t$ for the semigroup associated with the $\mathfrak{h}(X)(t)$-diffusion and similarly $\hat{\mathsf{H}}_t$ for the semigroup associated with the $\mathfrak{h}(\hat{X})(t)$-diffusion we see that,
\begin{align*}
\mathsf{P}_t \textbf{1}_{[l,y]}(x)=\mathsf{\hat{P}}_t \textbf{1}_{[x,r]}(y) \ \ \forall x,y \in I^\circ \iff \mathsf{H}_t \textbf{1}_{[l,y]}(x)=\mathsf{\hat{H}}_t \textbf{1}_{[x,r]}(y) \ \ \forall x,y \in I^{\circ}.
\end{align*}
And we furthermore observe that, the boundary behaviour of the  $\mathfrak{h}(X)(t)$-diffusion at $l$ is the boundary behaviour of the $L$-diffusion at $r$ and its boundary behaviour at $r$ is that of the $L$-diffusion at $l$ and similarly for $\mathfrak{h}(\hat{X})(t)$. We thus obtain an equivalent problem where now $(\mathsf{b.c.}(l),\mathsf{b.c.}(r))\leftrightarrow(\mathsf{b.c.}(r),\mathsf{b.c.}(l))$. Putting it all together, we reduce to the following 4 cases since all others can be obtained from the transformations above,
\begin{align*}
(nat,nat),(ref,nat),(ref,ref),(ref,abs).
\end{align*}
The first case is easy since there are no boundary conditions to keep track of and is omitted. The second case is the one originally considered by Siegmund and studied extensively in the literature (see e.g. \cite{CoxRosler} for a proof). We give the proof for the last two cases.

First, assume $l$ and $r$ are regular reflecting for $X$ and so absorbing for $\hat{X}$. Let $\mathcal{R}_{\lambda}$ and $\hat{\mathcal{R}}_{\lambda}$ be the resolvent operators associated with $\mathsf{P}_t$ and $\mathsf{\hat{P}}_t$ then with $f$ being a continuous function with compact support in $I^{\circ}$ the function $u=\mathcal{R}_{\lambda}f$ solves Poisson's equation $\mathcal{D}_m\mathcal{D}_su-\lambda u=-f$ with $\mathcal{D}_su(l^+)=0, \mathcal{D}_su(r^-)=0$. Apply $\mathcal{D}_m^{-1}$ defined by $\mathcal{D}_m^{-1}f(y)=\int_{l}^{y}m(z)f(z)dz$ for $y\in I^\circ$ to obtain $\mathcal{D}_su-\lambda\mathcal{D}_m^{-1} u=-\mathcal{D}_m^{-1}f$ which can be written as,
\begin{align*}
\mathcal{D}_{\hat{m}}\mathcal{D}_{\hat{s}}\mathcal{D}_m^{-1}u-\lambda\mathcal{D}_m^{-1} u=-\mathcal{D}_m^{-1}f.
\end{align*}
So $v=\mathcal{D}_m^{-1}u$ solves Poisson's equation with $g=\mathcal{D}_m^{-1}f$,
\begin{align*}
\mathcal{D}_{\hat{m}}\mathcal{D}_{\hat{s}}v-\lambda v=-g,
\end{align*}
with the boundary conditions $\mathcal{D}_{\hat{m}}\mathcal{D}_{\hat{s}}v(l^+)=\mathcal{D}_{s}\mathcal{D}_{m}\mathcal{D}_{m}^{-1}u(l^+)=\mathcal{D}_s u(l^+)=0$ and $\mathcal{D}_{\hat{m}}\mathcal{D}_{\hat{s}}v(r^-)=0$. Now in the second case when $l$ is reflecting and $r$ absorbing we would like to check the reflecting boundary condition for $v=\mathcal{D}_m^{-1}u$ at $r$. Namely, that $(\mathcal{D}_{\hat{s}})v(r^-)=0$ and note that this is equivalent to $(\mathcal{D}_{m})v(r^-)=u(r^-)=0$. This then follows from the fact that (since $r$ is now absorbing for the $L$-diffusion) $(\mathcal{D}_m\mathcal{D}_s)u(r^-)=0$ and that $f$ is of compact support. The proof proceeds in the same way for both cases, by uniqueness of solutions to Poisson's equation (see e.g. Section 3.7 of \cite{ItoMckean}) this implies $v= \hat{\mathcal{R}}_{\lambda}g$ and thus we may rewrite the relationship as,
\begin{align*}
\mathcal{D}_m^{-1}\mathcal{R}_{\lambda}f= \hat{\mathcal{R}}_{\lambda}\mathcal{D}_m^{-1}f.
\end{align*}
Let now $f$ approximate $\delta_x$ with $x \in I^\circ$ to obtain with $r_{\lambda}(x,z)$ the resolvent density of $\mathcal{R}_{\lambda}$ with respect to the speed measure in $I^\circ \times I^\circ$,
\begin{align*}
\int_{l}^{y}r_{\lambda}(z,x)m(z)dz=m(x)\hat{\mathcal{R}}_{\lambda}\textbf{1}_{[x,r]}(y).
\end{align*}
Since $r_{\lambda}(z,x)m(z)=m(x)r_{\lambda}(x,z)$ we obtain,
\begin{align*}
\mathcal{R}_{\lambda}\textbf{1}_{[l,y]}(x)=\hat{\mathcal{R}}_{\lambda}\textbf{1}_{[x,r]}(y),
\end{align*}
and the result follows by uniqueness of Laplace transforms.
\end{proof}

It is certainly clear to the reader that the proof only works for $x,y$ in the interior $I^\circ$. In fact the lemma is not always true if we allow $x,y$ to take the values $l,r$. To wit, first assume $x=l$ so that we would like,
\begin{align*}
\mathsf{P}_t \textbf{1}_{[l,y]}(l)\overset{?}{=}\mathsf{\hat{P}}_t \textbf{1}_{[l,r]}(y)=1 \ \forall y.
\end{align*}
This is true if and only if $l$ is either absorbing, exit or natural for the $L$-diffusion (where in the case of a natural boundary we understand $\mathsf{P}_t \textbf{1}_{[l,y]}(l)$ as $\lim_{x\to l}\mathsf{P}_t \textbf{1}_{[l,y]}(x)$). Analogous considerations give the following: The statement of Lemma \ref{ConjugacyLemma} remains true with $x=r$ if $r$ is either a natural, reflecting or entrance boundary point for the $L$-diffusion. Enforcing the exact same boundary conditions gives that the statement remains true with $y$ taking values on the boundary of $I$.

\begin{rmk}
For the reader who is familiar with the close relationship between duality and intertwining first note that with the $L$-diffusion satisfying the boundary conditions in the paragraph above and denoting as in Section 2 by $P_t$ the semigroup associated with an $L$-diffusion killed (not absorbed) at $l$ our duality relation becomes,
\begin{align*}
P_t \textbf{1}_{[x,r]}(y)=\mathsf{\hat{P}}_t \textbf{1}_{[l,y]}(x) .
\end{align*}
It is then a simple exercise, see Proposition 5.1 of \cite{CarmonaPetitYor} for the general recipe of how to do this, that this is equivalent to the intertwining relation,
\begin{align*}
P_t\Lambda=\Lambda\mathsf{\hat{P}}_t,
\end{align*}
where $\Lambda$ is the unnormalized kernel given by $(\Lambda f)(x)=\int_{l}^{x}\hat{m}(z)f(z)dz$. This is exactly the intertwining relation obtained in (\ref{KMintertwining}) with $n_1=n_2=1$.
\end{rmk}

\paragraph{Entrance Laws} For $x \in I$ and $\mathfrak{h}_n$ a positive eigenfunction of $P_t^n$ we would like to compute the following limit that defines our entrance law $\mu_t^x\left(\vec{y}\right)$ (with respect to Lebesgue measure) and corresponds to starting the Markov process $P^{n,\mathfrak{h}_n}_t$ from $(x,\cdots,x)$,
\begin{align*}
\mu_t^x\left(\vec{y}\right):=\lim_{(x_1,\cdots,x_n)\to x \vec{1}}e^{-\lambda t}\frac{\mathfrak{h}_n(y_1,\cdots,y_n)}{\mathfrak{h}_n(x_1,\cdots,x_n)}\det\left(p_t(x_i,y_j)\right)^n_{i,j=1} .
\end{align*}

Note that, since as proven in subsection \ref{subsectioneigen} all eigenfunctions built from the intertwining kernels are of the form $\det\left(h_i(x_j)\right)^n_{i,j=1}$ we will restrict to computing,
\begin{align*}
\mu_t^x\left(\vec{y}\right):= e^{-\lambda t}\det\left(h_i(y_j)\right)^n_{i,j=1}\lim_{(x_1,\cdots,x_n)\to x\vec{1}}\frac{\det\left(p_t(x_i,y_j)\right)^n_{i,j=1}}{\det\left(h_i(x_j)\right)^n_{i,j=1}}.
\end{align*}
If we now assume that $p_t(\cdot,y) \in C^{n-1} \forall t>0, y\in I^{\circ}$ and similarly $h_i(\cdot) \in C^{n-1}$ (in fact we only need to require this in a neighbourhood of $x$) we have,
\begin{align*}
\lim_{(x_1,\cdots,x_n)\to x\vec{1}}\frac{\det\left(p_t(x_i,y_j)\right)^n_{i,j=1}}{\det\left(h_i(x_j)\right)^n_{i,j=1}}&=\lim_{(x_1,\cdots,x_n)\to x\vec{1}}\frac{\det\left(x_j^{i-1}\right)^n_{i,j=1}}{\det\left(h_i(x_j)\right)^n_{i,j=1}}\times \frac{\det\left(p_t(x_i,y_j)\right)^n_{i,j=1}}{\det\left(x_j^{i-1}\right)^n_{i,j=1}}\\
&=\frac{1}{\det\left(\partial^{i-1}_xh_j(x)\right)^n_{i,j=1}}\det\left(\partial^{i-1}_xp_t(x,y_j)\right)^n_{i,j=1}.
\end{align*}
For the fact that the Wronskian, $\det\left(\partial^{i-1}_xh_j(x)\right)^n_{i,j=1}>0$ and in particular does not vanish see subsection \ref{subsectioneigen}. Thus,
\begin{align*}
\mu_t^x\left(\vec{y}\right)=const_{x,t}\times \det\left(h_i(y_j)\right)^n_{i,j=1}\det\left(\partial^{i-1}_xp_t(x,y_j)\right)^n_{i,j=1},
\end{align*}
is given by a biorthogonal ensemble as in (\ref{biorthogonalensemble}). The following lemma, which is an adaptation of Lemma 3.2 of \cite{O Connell} to our general setting, gives some more explicit information.

\begin{lem}
Assume that for $x'$ in a neighbourhood of $x$ there is a convergent Taylor expansion $\forall t>0, y \in I^{\circ}$,
\begin{align*}
\frac{p_t(x',y)}{p_t(x,y)}=f(t,x')\sum_{i=0}^{\infty}\left(x'-x\right)^i\phi_i(t,y),
\end{align*}
for some functions $f, \{\phi_{i}\}_{i\ge 0}$ that in particular satisfy $f(t,x)\phi_0(t,y)\equiv 1$. Then $\mu_t^x\left(\vec{y}\right)$ is given by the biorthogonal ensemble,
\begin{align*}
const_{x,t}\times \det\left(h_i(y_j)\right)^n_{i,j=1} \det\left(\phi_{i-1}(t,y_j)\right)^n_{i,j=1} \prod_{i=1}^{n}p_t(x,y_i).
\end{align*}
If moreover we assume that we have a factorization $\phi_i(t,y)=y^ig_i(t)$ then $\mu_t^x\left(\vec{y}\right)$ is given by the polynomial ensemble,
\begin{align*}
const'_{x,t} \times \det\left(h_i(y_j)\right)^n_{i,j=1}\det\left(y^{i-1}_j\right)^n_{i,j=1} \prod_{i=1}^{n}p_t(x,y_i).
\end{align*}
\end{lem}
\begin{proof}
By expanding the Karlin-McGregor determinant and plugging in the Taylor expansion above we obtain,
\begin{align*}
\frac{\det\left(p_t(x_i,y_j)\right)^n_{i,j=1}}{\prod_{i=1}^{n}p_t(x,y_i)}&=\prod_{i=1}^{n}f(t,x_i)\sum_{k_1,\cdots,k_n \ge 0}^{}\prod_{i=1}^{n}\left(x_i-x\right)^{k_i}\sum_{\sigma \in  \mathfrak{S}_n}^{} sign(\sigma)\prod_{i=1}^{n}\phi_{k_i}(t,y_{\sigma(i)})\\
&=\prod_{i=1}^{n}f(t,x_i)\sum_{k_1,\cdots,k_n \ge 0}^{}\prod_{i=1}^{n}\left(x_i-x\right)^{k_i}\det\left(\phi_{k_i}(t,y_j)\right)^n_{i,j=1}.
\end{align*}
First, note that we can restrict to $k_1,\cdots k_n$ distinct otherwise the determinant vanishes. Moreover, we can in fact restrict the sum over $k_1,\cdots,k_n \ge 0$ to $k_1,\cdots,k_n$ ordered by replacing $k_1,\cdots,k_n$ by $k_{\tau(1)},\cdots,k_{\tau(n)}$ and summing over $\tau \in \mathfrak{S}_n$ to arrive at the following expansion,
\begin{align*}
\frac{\det\left(p_t(x_i,y_j)\right)^n_{i,j=1}}{\prod_{i=1}^{n}p_t(x,y_i)}=\prod_{i=1}^{n}f(t,x_i)\sum_{0\le k_1<k_2<\cdots<k_n }^{}\det\left(\left(x_j-x\right)^{k_i}\right)^n_{i,j=1}\det\left(\phi_{k_i}(t,y_j)\right)^n_{i,j=1}.
\end{align*}
Now, write with $\vec{k}=(0\le k_1< \cdots <k_n)$ ,
\begin{align*}
\chi_{\vec{k}}(z_1,\cdots,z_n)=\frac{\det\left(z^{k_i}_j\right)^n_{i,j=1}}{\det\left(z^{i-1}_j\right)^n_{i,j=1}},
\end{align*}
for the Schur function and note that $\lim_{(z_1,\cdots,z_n)\to 0}\chi_{\vec{k}}(z_1,\cdots,z_n)=0$ unless $\vec{k}=(0,\cdots,n-1)$ in which case we have $\chi_{\vec{k}}\equiv 1$. We can now finally compute,
\begin{align*}
&\lim_{(x_1,\cdots,x_n)\to x\vec{1}} \frac{\det\left(p_t(x_i,y_j)\right)^n_{i,j=1}}{\det\left(x_j^{i-1}\right)^n_{i,j=1}}=\lim_{(x_1,\cdots,x_n)\to x\vec{1}} \frac{\det\left(p_t(x_i,y_j)\right)^n_{i,j=1}}{\det\left((x_j-x)^{i-1}\right)^n_{i,j=1}}=\prod_{i=1}^{n}p_t(x,y_i)\times\\
&\times \lim_{(x_1,\cdots,x_n)\to x\vec{1}} \prod_{i=1}^{n}f(t,x_i)\sum_{0\le k_1<k_2<\cdots<k_n }^{}\chi_{\vec{k}}(x_1-x,\cdots,x_n-x)\det\left(\phi_{k_i}(t,y_j)\right)^n_{i,j=1}=\\
&=f^n(t,x)\times \prod_{i=1}^{n}p_t(x,y_i)\det\left(\phi_{i-1}(t,y_j)\right)^n_{i,j=1}.
\end{align*}
The first statement of the lemma now follows with,
\begin{align*}
const_{x,t}=e^{-\lambda t} f^n(t,x)\frac{1}{\det\left(\partial^{i-1}_xh_j(x)\right)^n_{i,j=1}}.
\end{align*}
The fact that when  $\phi_i(t,y)=y^ig_i(t)$ we obtain a polynomial ensemble is then immediate.
\end{proof}

\bigskip
\noindent
{\sc School of Mathematics, University of Bristol, U.K.}\newline
\href{mailto:T.Assiotis@bristol.ac.uk}{\small T.Assiotis@bristol.ac.uk}

\bigskip
\noindent
{\sc School of Mathematics and Statistics, University College Dublin, Belfield, Dublin 4, Ireland}\newline
\href{mailto:neil.oconnell@ucd.ie}{\small neil.oconnell@ucd.ie}

\bigskip
\noindent
{\sc Department of Statistics, University of Warwick, Coventry CV4 7AL, U.K.}\newline
\href{mailto:J.Warren@warwick.ac.uk}{\small J.Warren@warwick.ac.uk}

\end{document}